\documentclass[11pt]{article}
\usepackage[T1]{fontenc}
\usepackage{lmodern}
\usepackage[a4paper,margin=27mm]{geometry}
\usepackage{amsmath,amssymb,amsthm,mathtools}
\usepackage{enumitem,needspace,microtype}
\usepackage{tikz}
\usetikzlibrary{arrows.meta}
\usepackage[hidelinks]{hyperref}
\hypersetup{pdftitle={Nonuniqueness for the planar Euler equation}}
\numberwithin{equation}{section}
\newtheorem{theorem}{Theorem}[section]
\newtheorem{proposition}[theorem]{Proposition}
\newtheorem{lemma}[theorem]{Lemma}
\newtheorem{corollary}[theorem]{Corollary}
\theoremstyle{remark}
\newtheorem*{remark}{Remark}

\usepackage{titlesec}
\renewcommand{\theparagraph}{\thesubsection.\arabic{paragraph}}
\titleformat{\paragraph}[runin]{\normalfont\normalsize\bfseries}{\theparagraph.}{0.5em}{}
\titlespacing*{\paragraph}{0pt}{3.25ex plus 1ex minus .2ex} {1.5ex plus .2ex}

\newcommand{\upindex}[1]{^{(#1)}}

\title{Nonuniqueness of Solutions to 2D Incompressible Euler Equations with Unbounded Vorticity}
\author{Hyungjun Choi}
\date{25 October 2026}
\begin{document}
\maketitle

\begin{abstract}
We construct two distinct solutions of the two-dimensional incompressible Euler equation with the same initial vorticity in $\bigcap_{1\le p<\infty}L^p(\mathbb R^2)$. The vorticities are compactly supported in space and have mean zero, so the corresponding velocities have finite energy. The nonuniqueness mechanism is a cascade of transitions in quadrupole dynamics across infinitely many scales.
\end{abstract}

\tableofcontents

\section{Introduction}
The incompressible Euler equation on $\mathbb R^2$ can be written in terms of velocity $u$ and vorticity $\omega=\partial_1u_2-\partial_2u_1$ as
\[ \left\{ \begin{aligned} & \partial_t\omega+u\cdot\nabla\omega=0, \\ & u=K*\omega,\quad K(x)=\frac{(-x_2,x_1)^t}{2\pi|x|^2} \end{aligned} \right. \]
We construct two distinct solutions of the two-dimensional incompressible Euler equation with the same initial vorticity. The initial vorticity is compactly supported, mean zero, and lies in every finite $L^p(\mathbb R^2)$.

\paragraph*{Existence, uniqueness, and the Yudovich conjecture.}
For smooth, sufficiently decaying initial vorticity, the planar Euler equation has a unique global smooth solution. The classical theory originates with Wolibner~\cite{Wolibner} and was developed further by Kato~\cite{Kato}. The vorticity is transported by an area-preserving flow, so its $L^p$ norms are conserved. Yudovich~\cite{Yudovich1963} extended global existence and uniqueness to bounded integrable vorticity, $\omega_0\in L^1\cap L^\infty$. The corresponding velocity is log-Lipschitz, which supplies the modulus of continuity needed for uniqueness. His later work~\cite{Yudovich1995} also permits unbounded vorticity under suitable restrictions on the growth of $\|\omega_0\|_{L^p}$ as $p\to\infty$. Thus membership in every finite $L^p$ must be distinguished from quantitative control of these norms. For a broader account of Euler dynamics and regularity, see~\cite{ElgindiDrivas}.

Weak existence reaches substantially rougher data. DiPerna and Majda~\cite{DiPernaMajda} developed compactness and concentration arguments yielding global weak solutions for integrable $L^p$ vorticity, $p>1$. Delort~\cite{Delort} reached distinguished-sign measure vorticity in $H^{-1}_{\mathrm{loc}}$ with locally finite energy. These existence results do not assert uniqueness. Uniqueness for a transport equation with a prescribed Sobolev velocity, as in DiPerna--Lions~\cite{DiPernaLions}, does not settle Euler uniqueness, because two Euler solutions transport their vorticities by different velocities.

The \emph{Yudovich nonuniqueness conjecture} has the following Lebesgue-space formulation in~\cite{Mengual,VishikBook}. For each $2<p<\infty$, there exist initial velocity and vorticity satisfying
\[ u_0\in L^2(\mathbb R^2),\qquad \omega_0=\operatorname{curl}u_0\in L^1(\mathbb R^2)\cap L^p(\mathbb R^2), \]
from which the unforced Euler equation admits distinct weak solutions with
\[ u\in C([0,T];L^2),\qquad \omega\in L^\infty([0,T];L^1\cap L^p). \]
The regularity is required throughout the evolution, not only initially. We discuss several developments relevant to this question before explaining the point-vortex ingredients of the present construction.

\paragraph*{Self-similar solutions and instability.}
The use of instability in similarity coordinates has been successful for the similar non-uniqueness problem in three-dimensional Navier--Stokes equations. Jia and \v Sver\'ak~\cite{JiaSverak2014,JiaSverak2015} constructed forward self-similar solutions and developed a route from an unstable self-similar profile to nonuniqueness. Guillod and \v Sver\'ak~\cite{GuillodSverak} supplied numerical evidence for the required bifurcation. More recently, Hou, Wang and Yang~\cite{HouWangYang} presented a computer-assisted proof of nonuniqueness of Leray--Hopf weak solutions, based on an unstable self-similar profile and localization. Ionescu, Jia and Palasek~\cite{IonescuJiaPalasek} numerically construct unstable self-similar profiles in the axisymmetric swirl-free class. An unstable trajectory in logarithmic time $\tau=\log t$ can yield a second solution with the same initial trace at $t=0$, provided the spatial behavior and nonlinear remainder are controlled. Vishik~\cite{Vishik} used instability and similarity variables to prove nonuniqueness for the forced two-dimensional Euler equation with $L^p$ vorticity. A detailed account is given in~\cite{VishikBook}. Albritton, Bru\`e and Colombo~\cite{ABC} implemented this mechanism for the \emph{forced} three-dimensional Navier--Stokes equations to obtain two Leray--Hopf weak solutions with the same force and initial data.

Self-similar solutions have been constructed for the \emph{unforced} two-dimensional Euler equation. Elling~\cite{Elling2013,Elling2016} constructed algebraic spiral solutions, subsequently developed by Shao, Wei and Zhang~\cite{ShaoWeiZhang}, and Choi~\cite{ChoiSpiral} removed the rotational symmetry restriction in a class of spiral profiles. Choi and Coiculescu~\cite{ChoiCoiculescu} constructed homogeneous steady solutions whose pseudo-velocity field has several stagnation points. Choi~\cite{ChoiLargeData} and Hou, Qin, Song and Zhang~\cite{HouQinSongZhang} constructed forward self-similar solutions from $-\alpha$-homogeneous initial vorticity for $1<\alpha<2$, without smallness or sign assumptions. Bressan and Shen~\cite{BressanShen} studied self-similar solutions in connection with nonuniqueness and numerically observed competing spiral configurations. However, transferring the Navier--Stokes strategy to the two-dimensional Euler equation presents difficulties~\cite{BinzCoiculescu,ChoiPreparation,CoiculescuIllposed}. In addition, having two distinct self-similar solutions from the same initial homogeneous data does not guarantee resolution of the problem with integrable vorticity.

\paragraph*{Multiscale constructions.}
While self-similarity provides a concrete route to nonuniqueness, it imposes a fixed continuous relation between spatial scale and time. Multiscale constructions organizing simpler dynamics at separated scales can offer greater flexibility. With infinitely many scales, an infinitesimal difference in the initial data can lead to a distinction between solutions at larger scales. Dyadic models and averaged equations provide examples of singularity formation~\cite{KatzPavlovic,TaoAveraged,PalasekModel}. See~\cite{CheskidovDaiFriedlander} for a survey. Filonov and Khodunov~\cite{FilonovKhodunov} obtained Leray--Hopf nonuniqueness for a forced dyadic model, and Palasek~\cite{Palasek} for an unforced model. For the three-dimensional Navier--Stokes equations, Coiculescu and Palasek~\cite{CoiculescuPalasek} constructed nonunique solutions from critical $\mathrm{BMO}^{-1}$ data on the torus, smooth for positive times. Cheskidov, Dai and Palasek~\cite{CheskidovDaiPalasek} constructed weak solutions that branch from a classical solution through transfer from arbitrarily high to lower frequencies.

Related methods yield Navier--Stokes norm growth~\cite{PalasekGrowth}, loss of Sobolev regularity for two-dimensional Euler~\cite{CordobaMartinezOzanski}, and blow-up of classical solutions of the three-dimensional Euler equations~\cite{CordobaMartinezZhengEuler}. These growth and blow-up results are distinct from nonuniqueness.

\paragraph*{Convex integration.}
The approach of De Lellis and Sz\'ekelyhidi~\cite{DeLellisSzekelyhidi} provides another mechanism for nonuniqueness of weak Euler solutions. Controlling vorticity in two dimensions is an additional constraint. Bru\`e and Colombo~\cite{BrueColombo} obtained nonuniqueness with vorticity in the weak-$L^1$ Lorentz space $L^{1,\infty}$, and Bru\`e, Colombo and Kumar~\cite{BrueColomboKumar} constructed periodic solutions with vorticity in $C([0,1];L^p)$ for some $p>1$, using concentrated traveling vortices in a convex-integration scheme. Mengual~\cite{Mengual} obtained nonunique bounded admissible velocity solutions from initial vorticity in $L^1\cap L^p$ for every finite $p>2$, but the competing solutions are not shown to retain $L^p$ vorticity at positive times.

\paragraph*{Point-vortex dynamics and their realization by patches.}
Point vortices provide explicit Hamiltonian motions. See, for instance, Aref~\cite{Aref} and the references therein. Grotto and Pappalettera~\cite{GP} provide a closely related motivation for using point-vortex dynamics to obtain nonuniqueness. They constructed a burst of one point vortex into three and obtained nonuniqueness for atomic vorticity in a symmetrized weak formulation. Dirac masses are not $L^p$ functions, so a separate mechanism is needed for integrable vorticity. Our building block is a quadrupole with dynamics connecting two distinct motions, a rotating rhombus and separating dipoles. The rhombus belongs to the relative equilibria classified by Hampton, Roberts and Santoprete~\cite{HRS}. Its connection with dipole scattering appeared in Eckhardt and Aref~\cite{EA} and a trajectory connecting the rhombus to separating dipoles appeared in Lydon, Nazarenko and Laurie~\cite{LNL}. Section~\ref{sec:four-point} derives this trajectory and the estimates needed below directly from the four-vortex equations.

We replace each point vortex by a small circular patch to obtain $L^p$ vorticity. Approximation of point vortices by concentrated Euler vorticity has been widely studied. Marchioro and Pulvirenti~\cite{MarchioroPulvirenti1983,MarchioroPulvirenti1993} established the connection between the Euler evolution of concentrated vorticity and point-vortex trajectories, including localization of the transported supports. Turkington~\cite{Turkington} studied a concentrated vortex in a bounded domain. Butt\`a and Marchioro~\cite{ButtaMarchioro} obtained longer concentration times in several configurations. Meyer~\cite{Meyer} proves long-time confinement of several nearly circular concentrated vortices while they remain separated. Such confinement controls core sizes, but should not be confused with an arbitrarily long, uniformly accurate comparison with an unstable point-vortex trajectory. D\'avila, del Pino, Musso and Wei~\cite{DDMW} used gluing to construct smooth concentrated vortices following prescribed noncolliding point-vortex motions with detailed core asymptotics.

Circular-patch stability and preservation of boundary regularity are also distinct issues. Wan and Pulvirenti~\cite{WanPulvirenti} proved nonlinear stability of circular patches, and Sideris and Vega~\cite{SiderisVega} gave a direct $L^1$ stability argument on the plane. Chemin~\cite{Chemin} and Bertozzi and Constantin~\cite{BertozziConstantin} established global persistence of regular patch boundaries. These theorems do not by themselves give uniform closeness to a circle in a boundary norm during a coupled multiscale motion. Multiplying an isolated Euler vorticity by a positive constant merely rescales time, so amplitude alone does not separate stable and unstable disks. For a patch subject to an external field, the relevant comparison is between its internal rotation and the external strain. Kida's exact elliptic-patch dynamics~\cite{Kida} illustrate this competition, while Constantin~\cite{Constantin} studied interactions between distant patches.

The patch estimates in Section~\ref{meta:section} are tailored to this last issue. They control boundary deformation and give centroid and moment equations, together with a relative error between two distinct dynamics. The relative estimate is needed because the signal passed from one scale to the next is much smaller than the available absolute concentration error. We then arrange quadrupoles at decreasing length scales with increasing rotation frequencies. For each finite number of scales, one solution stays near rotating rhombi while a second starts with separating dipoles only at the smallest scale. The separating dipoles activate successively larger scales. The initial discrepancy tends to zero as the number of scales grows, whereas the later discrepancy survives at the largest scale at time $1$. The two limits share a common initial datum and differ at time $1$. The following subsections state the theorem and explain the quantitative choices and the limiting argument.

\paragraph*{Main theorem.}
\begin{theorem} \label{thm:main}
There exist a compactly supported function $\omega_0$ and two vorticities $\omega^{(1)},\omega^{(2)}$ such that 
\begin{enumerate}[label=(\roman*)]
	\item $\displaystyle\omega_0\in\bigcap_{1\le p<\infty}L^p(\mathbb R^2)$ and $\displaystyle\int_{\mathbb R^2}\omega_0\,dx=0$.
	\item For every $1\le p<\infty$ and $j=1,2$, each pair $(\omega\upindex{j}, u\upindex{j})$ is a renormalized solution of the vorticity transport equation with initial data $\omega\upindex{j}(0) = \omega_0$ and also a weak solution of the momentum equation. Indeed,
\[ \omega^{(j)}\in C([0,\infty);L^p), \quad u^{(j)} = K\ast \omega\upindex{j} \in C([0,\infty);L^2),\quad \|\omega^{(j)}(t)\|_{L^p}=\|\omega_0\|_{L^p},\qquad t\ge0. \]
	\item The two solutions are distinct since these satisfy
\[ \|\omega^{(2)}(1)-\omega^{(1)}(1)\|_{L^1}>0. \]
\end{enumerate}
\end{theorem}

\begin{remark}
Zero total circulation gives $u_0 = K*\omega_0\in L^2$. For each fixed $0<\alpha<1$, the datum can be chosen with $\|\omega_0\|_{L^p}\asymp p^{4/\alpha}$ for $p\ge2$, as shown in Section~\ref{sec:quartic-growth}.
\end{remark}

\subsection{The quadrupole motion}

Identify $\mathbb R^2$ with $\mathbb C$. Consider circulation $+\Gamma$ at $z_1,-z_1$ and $-\Gamma$ at $z_2,-z_2$, with $\Gamma>0$. Fix the length $\ell>0$ by
\[ |z_1|^2-|z_2|^2=\ell^2,\qquad \Omega=\frac{\Gamma}{2\pi\ell^2}. \]
The four positions
\[ \left\{\pm\ell\sqrt{\frac{\sqrt{2} +1}{2}}\,e^{i(\theta_0-\Omega t)},\quad \pm i\ell\sqrt{\frac{\sqrt{2} -1}{2}}\,e^{i(\theta_0-\Omega t)}\right\} \]
form an exact rotating rhombus. The two vertices on the longer diagonal have circulation $+\Gamma$, and those on the shorter diagonal have circulation $-\Gamma$. This relative equilibrium is part of the classification in Hampton, Roberts and Santoprete~\cite[Sections 4.2 and 7.4]{HRS}.

There is also a complete, collision-free solution with the following behavior. As $t\to-\infty$, its shape converges to this rhombus. As $t\to+\infty$, each positive vortex pairs with a negative vortex at distance tending to $\ell$. The two pair centers move in opposite directions, each with speed tending to $\Omega\ell$. Such a translating positive-negative pair is called a dipole. Lydon, Nazarenko and Laurie~\cite[Section III.1]{LNL} describe the limiting incoming motion in which two dipoles approach a rotating rhombus. The symmetry $z_j(t)\mapsto\overline{z_j(-t)}$ converts an incoming motion into a motion of separating dipoles with the same circulations. Eckhardt and Aref~\cite{EA} studied integrable four-vortex scattering earlier. Section~\ref{sec:four-point} constructs the separating motion using explicit first integrals. It also shows that a small positive unstable shape perturbation grows at rate $2\sqrt2\Omega$ near the rhombus, after the internal quadrupole rotation has been removed. The unstable and stable shape variables $q$ and $p$, defined in Section~\ref{sec:shape-variables}, determine the shape of a centrally symmetric quadrupole and describe its unstable and stable deformations from the rhombus.

The exactly rotating solution and the separating solution never coincide at a finite time. In particular, the point-vortex calculation by itself gives no nonuniqueness. Its role is to provide two different behaviors and a way for a smaller, faster quadrupole to influence a larger, slower one. If $\theta_\infty$ is the limiting direction from the negative to the positive member of a dipole, then the second complex moment of the configuration of separating dipoles obeys
\[ M_2(t):=2\Gamma\bigl(z_1(t)^2-z_2(t)^2\bigr)=-\frac{2i}{\pi}\Gamma^2e^{2i\theta_\infty}t+O(1). \]
For a rhombus, by contrast, $|M_2(t)|=2\sqrt2\Gamma\ell^2$. At a distant point $z$, the spatial derivative of the velocity generated by the quadrupole is proportional, to leading order, to $\overline{M_2}/|z|^4$. The linearly growing moment of a configuration of separating dipoles can deform another rhombus if the two configurations are oriented so that instability amplifies the effect. Section~\ref{sec:four-point} computes the contribution of this influence to the unstable shape variable.

\subsection{Scales, construction of two solutions, and cascade} \label{intro:scales}

Fix $0<\alpha<1$. For each positive integer $n$, choose
\begin{equation} \label{eq:scales}
d_n=2^{-n},\qquad \Omega_n=\Lambda B^n,\qquad \ell_n=e^{-\Omega_n^\alpha},\qquad \Gamma_n=2\pi\Omega_n\ell_n^2.
\end{equation}
Here $d_n$ is an initial location on the real axis, $\ell_n$ is the length scale of the quadrupole, $\Gamma_n$ is the absolute circulation of each of its four vortices, and $\Omega_n$ is its rhombus rotation frequency. As $n$ increases, the quadrupoles become smaller, lie closer to the origin, and rotate faster. The constant $B>2$ is chosen sufficiently large depending on $0<\alpha<1$, and then $\Lambda$ is chosen sufficiently large.

Replace each point vortex at scale $n$ by a circular patch of circulation $\pm\Gamma_n$ and vorticity value $\pm A\Omega_n^4$, where $A$ is a fixed sufficiently large constant. The circulation determines the radius $a_n$ through
\[ \pi a_n^2 A\Omega_n^4=\Gamma_n,\qquad a_n^2=\frac{2\ell_n^2}{A\Omega_n^3},\qquad \frac{a_n}{\ell_n}=\sqrt{\frac2A}\,\Omega_n^{-3/2}. \]
The individual patches are consequently much smaller than the distances between the four vortex centers. The construction maintains this separation even when the two dipoles move apart. There is also much greater separation between different scales, since
\[ |d_n-d_{n+1}|=2^{-(n+1)},\qquad \sup_{n\ge1}\frac{\ell_n(1+\Omega_n)}{|d_n-d_{n+1}|}\longrightarrow0\quad\text{as }\Lambda\to\infty\text{ with }B\text{ fixed}. \]
The factor $1+\Omega_n$ includes the distance traveled by the dipoles during a bounded physical time. Their configurations stay small relative to the distances between scales.

For each $N$, the first solution is denoted by $\omega_N^{(1)}$. It contains four patches at each scale $1\le n\le N$. The center of a quadrupole is the arithmetic mean of its four patch centroids. The initial centroids are chosen so that the shape about this moving center stays close to a rotating rhombus throughout $0\le t\le2$. Simply placing disks at the point-vortex rhombus would not cancel the fields from the other patches. For any prescribed initial phases, the topological argument in Section~\ref{sec:first-solution} finds an initial unstable shape variable at each scale, simultaneously, so that all remain small through time $2$.

The second solution is denoted by $\omega_N^{(2)}$. At time zero its patches at scales $1,\ldots,N-1$ are identical to those of $\omega_N^{(1)}$. Only the centers of the four disks at scale $N$ are changed. They are placed at a fixed configuration sufficiently far along the separating point-vortex trajectory, with the same radii, vorticity values, and circulations. The intended motion is
\[ \begin{array}{c|c|c} &t=0&0<t\le1\\ \hline \omega_N^{(1)}&\text{all scales near rhombi}&\text{all scales remain near rhombi}\\ \omega_N^{(2)}&\text{only scale }N\text{ already separating}&N\text{ influences }N-1,\;\ldots,\;2\text{ influences }1. \end{array} \]
These descriptions concern the shape of the quadrupole at each scale.

Here is why the second motion reaches scale $1$ in a time bounded independently of $N$. Once scale $n+1$ separates, the time derivative of its complex second moment has magnitude of order $\Gamma_{n+1}^2$. Its resulting strain grows linearly in elapsed time and oscillates in the rotating coordinates of scale $n$. Formula~\eqref{eq:positive-coefficient} integrates this forcing against $e^{-2\sqrt2\Omega_nt}$ and gives a dimensionless coefficient of the exponentially growing response, with positive maximum comparable to
\[ \frac{\Gamma_{n+1}^2}{\Omega_n^2|d_n-d_{n+1}|^4}. \]
The two factors of $\Omega_n^{-1}$ come from elapsed time in the linearly growing forcing and from integration over the growth time scale. The weighted integral is a sinusoidal function of the orientations with nonzero amplitude. The initial orientations are chosen recursively so that its predicted leading effect is maximized. Exponential growth at rate $2\sqrt2\Omega_n$ then reaches a fixed positive size after a time delay of order
\begin{equation} \label{eq:time-delay}
\frac{1}{2\sqrt2\Omega_n}\left[\log\!\left(\frac{\Omega_n^2|d_n-d_{n+1}|^4}{\Gamma_{n+1}^2}\right)+O(1)\right]=\frac{4\Omega_{n+1}^\alpha+O(n+\log B+1)}{2\sqrt2\Omega_n}.
\end{equation}
The leading delays are summable since
\begin{equation} \label{eq:delay-sum}
\sum_{n\ge1}\frac{\Omega_{n+1}^\alpha}{\Omega_n}=\frac{\Lambda^{\alpha-1}B^{2\alpha-1}}{1-B^{\alpha-1}}.
\end{equation}
Thus all departures occur before time $1/2$ after $\Lambda$ is chosen large enough. Section~\ref{array:section} establishes these estimates for the fully coupled patch solutions, retaining the responses of slower scales and the effects of every other scale. The scalar calculation supplies the leading coefficient, while the coupled estimates control its sign and accumulated errors.

Theorem~\ref{meta:reduction} compares the point-vortex approximation errors in the centroid equations of the two distinct patch solutions at the same scale. Their difference is bounded by a time integral of centroid and external-field differences from time zero. This integral retains the earlier motion until the predicted transition time. Having this relative bound is essential because an absolute error in the centroid equation that is only a negative power of $\Omega_n$ would exceed the exponentially small effect from scale $n+1$ described above.

\subsection{Identical limiting initial data and different later solutions}

Only the four disks at scale $N$ differ initially. Their amplitudes are $A\Omega_N^4$ and their total absolute circulation is $4\Gamma_N$, so, for every finite $p\ge1$,
\begin{equation} \label{eq:initial-difference}
\begin{aligned}
\|\omega_N^{(2)}(0)-\omega_N^{(1)}(0)\|_{L^p}^p&\le 2^{p+2}\Gamma_N(A\Omega_N^4)^{p-1}\\
&\le C_{p,A}\Omega_N^{4p-3}e^{-2\Omega_N^\alpha}\longrightarrow0.
\end{aligned}
\end{equation}
To obtain a common initial limit, we also need compactness of either initial solution family. At each fixed scale, the first solution family's disk centers lie in a compact set, while its radii and amplitudes are fixed. A diagonal subsequence makes all these centers converge. Area preservation gives
\begin{equation} \label{eq:lp-sum}
\|\omega_N^{(k)}(t)\|_{L^p}^p=8\pi A^{p-1}\sum_{n=1}^N\Omega_n^{4p-3}e^{-2\Omega_n^\alpha}\le C_{p,A,\alpha,B,\Lambda},\qquad k=1,2.
\end{equation}
The summable tails, together with continuity of translated disks in $L^p$, give $\omega_N^{(1)}(0)\to\omega_0$ strongly in every finite $L^p$. Equation~\eqref{eq:initial-difference} gives the same limit for the second family. The phases may depend on $N$, and successive finite constructions need not be compatible.

The center estimate in Section~\ref{array:observation-section} gives displacement $o(\ell_1)$ from $d_1$ for the outer quadrupole of each solution, uniformly in $N$ as $\Lambda\to\infty$. At time $1$, the first solution's outer patches therefore remain within $C\ell_1$ of $d_1$, whereas the second solution's outer patches lie at distances comparable to $\Omega_1\ell_1$. A fixed annulus centered at $d_1$ with radius comparable to $\Omega_1 \ell_1$ contains all four outer patches of the second solution and no other patches of either solution. These radial bounds are uniform in the $N$-dependent phases. Thus one can choose a fixed smooth function $0\le\chi\le1$, supported in this annulus, such that
\begin{equation} \label{eq:finite-distinction}
\int\chi|\omega_N^{(1)}(1)|\,dx=0,\qquad \int\chi|\omega_N^{(2)}(1)|\,dx=4\Gamma_1.
\end{equation}
Absolute vorticity detects these four patches despite their zero signed circulation. All parameters, including $\Gamma_1>0$, remain fixed as $N\to\infty$.

Passing to a subsequence as $N\to\infty$ gives two solutions $\omega^{(1)},\omega^{(2)}$ with the same initial data $\omega^{(1)}(0)=\omega^{(2)}(0)=\omega_0$, but
\[ \|\omega^{(2)}(1)-\omega^{(1)}(1)\|_{L^1}\ge4\Gamma_1>0. \]

\subsection{Patch amplitude and vorticity norms} \label{sec:quartic-growth}

The vorticity magnitude $A\Omega_n^4$ is chosen to keep each patch close to a disk for a fixed physical time, uniformly over all scales. Its internal rotation has angular speed $A\Omega_n^4/2$, whereas the strain produced by the other three patches in its quadrupole has size $C\Omega_n$. Merely making the first speed larger than the second would not control the accumulated nonlinear deformation.

To explain the choice of exponent, consider one patch initially equal to a disk, with vorticity magnitude $M$, whose normalized external coefficients and their time derivatives are bounded by $L$ and $L^2$, respectively, as in Lemma~\ref{core:normal-form}. After removing the internal rotation and the nonresonant quadratic deformation terms, Lemma~\ref{core:normal-form} bounds the transformed boundary deformation $h(t)$ in physical time by
\[ h(t)\le C\frac{L}{M}+C\int_0^t\left(Mh(s)^3+\frac{L^2}{M}\right)\,ds. \]
The bounded linear evolution is oscillatory. The term $Mh^3$ is the remaining nonlinear deformation, and $L^2/M$ is the residual external forcing. Balancing these terms gives a deformation size $(L/M)^{2/3}$. The time over which the residual forcing accumulates to that size is
\[ \frac{(L/M)^{2/3}}{L^2/M}=\frac{M^{1/3}}{L^{4/3}}. \]
Lemma~\ref{core:confinement} makes this computation precise, including control of the boundary labels. The deformation stays at most $C(L/M)^{2/3}$ for $0\le t\le cM^{1/3}/L^{4/3}$. With $L\asymp\Omega_n$ and a hypothetical amplitude $M=A\Omega_n^\gamma$ with $\gamma>1$, the guaranteed time is therefore
\[ cA^{1/3}\Omega_n^{(\gamma-4)/3}. \]
For $\gamma=4$, a sufficiently large fixed $A$ gives the required interval $[0,2]$ and deformation $O(\Omega_n^{-2})$. For $\gamma<4$, the guaranteed interval shrinks to zero as $n\to\infty$, for every fixed $A$. Thus the fourth power is the threshold of the confinement estimate used here. A lower power would require a stronger estimate for the residual forcing or its effect on deformation than is available in this construction.

For the length scales \eqref{eq:scales}, the circulation identity gives $a_n^2=2\ell_n^2/(A\Omega_n^3)$. Since each quadrupole consists of four disjoint patches, its contribution to the $p$th power of the norm is its total absolute circulation times the $(p-1)$st power of its amplitude. Consequently,
\[ \|\omega_0\|_{L^p}^p=8\pi A^{p-1}\sum_{n\ge1}\Omega_n^{4p-3}e^{-2\Omega_n^\alpha}. \]
Estimating the series and taking $p$th roots gives
\[ cp^{4/\alpha}\le\|\omega_0\|_{L^p}\le Cp^{4/\alpha},\qquad p\ge 2.\]
These constants may depend on $\alpha,A,B,\Lambda$, but not on $p$. Thus these scales yield every norm-growth exponent $4/\alpha>4$. This range does not establish sharpness of Yudovich's uniqueness condition $\int_2^\infty dp/(p\|\omega_0\|_{L^p})=\infty$ in \cite{Yudovich1995}.

The proof proceeds as follows. Section~\ref{sec:four-point} derives the four-vortex motions and their interaction and outlines the cascade mechanism. Section~\ref{meta:section} controls patch deformation and derives the centroid and moment equations, including the relative error estimate. Section~\ref{sec:first-solution} constructs the first solution by choosing its initial unstable shape variables. Section~\ref{sec:second-solution} specifies the second initial configuration and explains how successive departures give a fixed later distinction. Section~\ref{transfer:material-section} proves the transition estimates for a single quadrupole with prescribed external strain. Section~\ref{array:section} verifies the hypotheses of these estimates by induction in the coupled system of quadrupoles. Section~\ref{sec:limit} passes both sequences to limits with the same initial datum and a positive difference at time $1$.

\subsection*{History of this work and use of AI/LLM}
The author has worked on nonuniqueness for 2D Euler for many years, including a series of works on self-similar solutions~\cite{ChoiSpiral,ChoiCoiculescu,ChoiLargeData}. Trial and error in pursuing this problem led the author to consider discrete multiscale construction for the reasons discussed in introduction. Guided by this strategy, the author found the main quadrupole dynamics used in this work and identified a cascade mechanism by which one quadrupole activates the next, as described in Sections 2, 5, and 6.

ChatGPT 5.6, 6.0, and 6.1 were used under the author's guidance to realize the cascade strategy through a sequence of quadrupole patches. In particular, they helped complete the technical details of replacing point vortices by patches, mainly in Section~\ref{meta:section}. ChatGPT 6.0 and Claude Opus 5.5 were also used for reference searches and language editing. The main ideas of the paper were generated by the author, who takes full responsibility for the correctness of the proof.

\subsection*{Recent works on the Yudovich conjecture}
While completing this work, the author learned of an independent preprint by Coiculescu~\cite{CoiculescuRecent} and the Hexagon preprint~\cite{HexagonRecent} prepared by Pappalettera and Sorella.

\subsection*{Acknowledgments}
The author thanks Professor Alexandru Ionescu for his encouragement in pursuing this problem.

\section{Quadrupole dynamics and cascade mechanism} \label{sec:four-point}

We first construct a rotating rhombus quadrupole and an orbit that approaches it in backward time and becomes separating dipoles in forward time. We then calculate the strain produced at another rotating quadrupole. This is the local interaction used in each step of the cascade. The passage from these finite-dimensional point-vortex dynamics to patches is justified in Section~\ref{meta:section} and the subsequent construction.

\subsection{Equations and the rotating rhombus}
The point-vortex system is a Hamiltonian system of ordinary differential equations which describes 2D Euler dynamics with vorticity data being a finite sum of point masses. For point vortices with positions $x_j\in\mathbb R^2$ and circulations $\Gamma_j$, put
\[ H=-\frac1{2\pi}\sum_{i<j}\Gamma_i\Gamma_j\log|x_i-x_j|. \]
Then, the point-vortex system is
\[ \Gamma_j\dot x_j=-(\nabla_{x_j}H)^\perp=\Gamma_j\sum_{i\ne j}\Gamma_iK(x_j-x_i), \]
and hence $H$ is conserved in time.
\[ \frac{dH}{dt}=\sum_j \nabla_{x_j}H\cdot\dot x_j=-\sum_j\frac{\nabla_{x_j}H\cdot(\nabla_{x_j}H)^\perp}{\Gamma_j}=0. \]
Use complex coordinates $z=x_1+ix_2$, so that $K(z)=i/(2\pi\overline z)$. In our configuration, fix $\Gamma>0$ and put circulation $+\Gamma$ at $z_1,-z_1$ and $-\Gamma$ at $z_2,-z_2$. In this case,
\[ H =-\frac{\Gamma^2}{2\pi}\left[\log(2|z_1|)+\log(2|z_2|)-2\log|z_1-z_2|-2\log|z_1+z_2|\right] =\frac{\Gamma^2}{2\pi}\log\frac{|z_1^2-z_2^2|^2}{4|z_1||z_2|}. \]
Translating all four points by a constant $d\in\mathbb C$ gives the same internal motion about $d$. Also, central symmetry is preserved by the equations. Summing the three velocities at each independent point gives
\[ \dot z_1=\frac{i\Gamma}{2\pi}\left(\frac{1}{2\overline z_1}-\frac{1}{\overline z_1-\overline z_2}-\frac{1}{\overline z_1+\overline z_2}\right),\qquad \dot z_2=\frac{i\Gamma}{2\pi}\left(\frac{1}{\overline z_2-\overline z_1}+\frac{1}{\overline z_2+\overline z_1}-\frac{1}{2\overline z_2}\right). \]
Equivalently,
\begin{equation} \label{eq:point-ode}
\dot z_1=-\frac{i\Gamma(3\overline z_1^2+\overline z_2^2)}{4\pi\overline z_1(\overline z_1^2-\overline z_2^2)},\qquad \dot z_2=-\frac{i\Gamma(\overline z_1^2+3\overline z_2^2)}{4\pi\overline z_2(\overline z_1^2-\overline z_2^2)}.
\end{equation}
The angular impulse $2\Gamma (|z_1|^2-|z_2|^2)$ is also conserved, since
\[ \frac{d}{dt}(|z_1|^2-|z_2|^2)=2\operatorname{Re}(\overline z_1\dot z_1-\overline z_2\dot z_2)=2\operatorname{Re}\left(-\frac{i\Gamma}{2\pi}\right)=0. \]
We restrict attention to configurations with $|z_1|>|z_2|$ and fix
\[ |z_1|^2-|z_2|^2=\ell^2>0,\qquad \Omega=\frac{\Gamma}{2\pi\ell^2}. \]
Dividing positions by $\ell$ and measuring time in the variable $\Omega t$ changes the coefficient in \eqref{eq:point-ode} to $\Gamma/(4\pi\ell^2\Omega)=1/2$. The resulting dimensionless equations do not depend on $\Gamma$ or $\ell$.

\begin{proposition} \label{prop:rhombus}
A uniformly rotating solution of \eqref{eq:point-ode} is, for $\theta_0 \in \mathbb R$,
\begin{equation} \label{eq:rhombus}
z_1(t)=\ell\sqrt{\frac{\sqrt2 + 1}{2}}\,e^{i(\theta_0-\Omega t)},\qquad z_2(t)=i\ell\sqrt{\frac{\sqrt2-1}{2}}\,e^{i(\theta_0-\Omega t)}.
\end{equation}
\end{proposition}

\begin{proof}
The displayed positions satisfy $|z_1|^2-|z_2|^2=\ell^2$ and $(z_2/z_1)^2=-(3-2\sqrt2)$. Substituting them into \eqref{eq:point-ode} gives $\dot z_j=-i\Omega z_j$ for $j=1,2$, which is their stated time dependence.
\end{proof}

\subsection{The shape variables} \label{sec:shape-variables}

To distinguish deformation from rigid rotation, define
\[ Re^{i\phi} = \bigg(\frac{z_2}{z_1}\bigg)^2, \qquad \theta=\arg z_1. \]
Here $0<R<1$ and angles are chosen continuously along the trajectory. The pair $(R,\phi)$ is unchanged by a rigid rotation, while $\theta$ records that rotation. The factor $2$ in $\phi$ makes the shape, with $\phi$ understood modulo $2\pi$, unchanged when the two vortices of either sign are relabeled. Since the angular impulse is fixed, the original positions are recovered by
\[ z_1=\frac{\ell e^{i\theta}}{\sqrt{1-R}},\qquad z_2=\frac{\ell\sqrt R\,e^{i(\theta+\phi/2)}}{\sqrt{1-R}}. \]
Thus, after fixing $\ell$ and removing the common rotation, two real variables remain.

For completeness, the logarithmic derivatives from \eqref{eq:point-ode} are
\[ \frac{\dot z_1}{z_1}=-\frac{i\Omega(1-R)}2\frac{3+Re^{-i\phi}}{1-Re^{-i\phi}},\qquad \frac{\dot z_2}{z_2}=-\frac{i\Omega(1-R)}{2R}\frac{1+3Re^{-i\phi}}{1-Re^{-i\phi}}. \]
Taking twice the real part of their difference, twice its imaginary part, and the imaginary part of the first expression gives, respectively,
\begin{equation} \label{eq:shape-ode}
\begin{aligned}
\dot R&=-\frac{4\Omega R(1-R)^2\sin\phi}{1+R^2-2R\cos\phi},\\
\dot\phi&=-\frac{\Omega(1-R^2)(1-4R+R^2+2R\cos\phi)}{R(1+R^2-2R\cos\phi)},\\
\dot\theta&=-\frac{\Omega(1-R)(3-2R\cos\phi-R^2)}{2(1+R^2-2R\cos\phi)}.
\end{aligned}
\end{equation}
In the first operation one uses $\dot R/R=2\operatorname{Re}(\dot z_2/z_2-\dot z_1/z_1)$. The rhombus corresponds to the fixed shape
\[ R_*=3-2\sqrt2,\qquad \phi=\pi. \]
The third equation then gives $\dot\theta=-\Omega$, in agreement with \eqref{eq:rhombus}. In a neighborhood of this fixed shape, introduce the dimensionless deviations
\[ x=\frac{R-R_*}{R_*},\qquad y=\frac{\phi-\pi}{\sqrt2}. \]
Differentiating the first two equations in \eqref{eq:shape-ode} at $(R_*,\pi)$ gives
\[ \left.\partial_\phi\dot R\right|_{(R_*,\pi)}=2\Omega R_*,\qquad \left.\partial_R\dot\phi\right|_{(R_*,\pi)}=\frac{4\Omega}{R_*},\qquad \left.\partial_R\dot R\right|_{(R_*,\pi)}=\left.\partial_\phi\dot\phi\right|_{(R_*,\pi)}=0. \]
Consequently,
\begin{equation} \label{eq:shape-linearization}
\binom{\dot x}{\dot y}=2\sqrt2\Omega\begin{pmatrix}0&1\\1&0\end{pmatrix}\binom{x}{y}+O\!\left(\Omega(x^2+y^2)\right).
\end{equation}
The linearized flow expands on $x=y$ and contracts on $x=-y$. Define the unstable shape variable $q$ and the stable shape variable $p$ by
\[ q=\frac{x+y}{2}, \quad p = \frac{x-y}{2}. \]
Along $p=0$, increasing the unstable shape variable $q$ from zero increases both $R$ and $\phi$ from their rhombus values. The next subsection constructs the exact nonlinear trajectory with this behavior.

The phase equation has a linear correction that will be useful when comparing orientations between these quadrupoles. Expansion of the third equation in \eqref{eq:shape-ode} gives
\[ \dot\theta=-\Omega+\Omega(\sqrt2-1)x+O\!\left(\Omega(x^2+y^2)\right),\qquad \dot\phi=4\Omega x+O\!\left(\Omega(x^2+y^2)\right). \]
Therefore the adjusted phase
\begin{equation} \label{eq:adjusted-phase}
\Theta=\theta-\frac{\sqrt2-1}{4}(\phi-\pi)
\end{equation}
satisfies $\dot\Theta=-\Omega+O(\Omega(x^2+y^2))$. This subtraction removes the linear change of orientation caused by the shape deformation.

\Needspace{15\baselineskip}
\subsection{The complete separating trajectory}

\begin{proposition} \label{prop:separating}
There is a solution of \eqref{eq:point-ode}, defined for all $t\in\mathbb R$, whose shape tends to $(R_*,\pi)$ with $q>0$ and $p/q\to0$ as $t\to-\infty$, and for which $R\to1$, $\phi\to2\pi$ as $t\to+\infty$. Every distance between two vortices is bounded below by a positive constant times $\ell$. After relabeling the negative pair, there are constants $\theta_\infty\in\mathbb R$ and $Z_\infty\in\mathbb C$ such that
\begin{equation} \label{eq:dipole-asymptotics}
\begin{aligned}
z_1(t)-z_2(t)&=\ell e^{i\theta_\infty}+O(t^{-2}),\\
\frac{z_1(t)+z_2(t)}2&=-i\Omega\ell e^{i\theta_\infty}t+Z_\infty+O(t^{-1})\qquad(t\to+\infty).
\end{aligned}
\end{equation}
After fixing $\theta_\infty$, the constant $Z_\infty$ is determined modulo time shift by the conservation of the impulse. The two dipoles have limiting internal distance $\ell$ and limiting velocities $\pm i\Omega\ell e^{i\theta_\infty}$.
\end{proposition}

\begin{proof}
The polar coordinates $R,\phi$ and the fixed impulse $\ell^2 = |z_1|^2 - |z_2|^2$ give
\[ |z_1^2-z_2^2|^2=|z_1|^4(1+R^2-2R\cos\phi),\qquad |z_1||z_2|=|z_1|^2\sqrt R,\qquad |z_1|^2=\frac{\ell^2}{1-R}. \]
Substitution yields
\[ H=\frac{\Gamma^2}{2\pi}\log\left[\frac{\ell^2}{4}\frac{1+R^2-2R\cos\phi}{\sqrt R(1-R)}\right],\qquad \frac{1+R^2-2R\cos\phi}{\sqrt R(1-R)}=\frac4{\ell^2}e^{2\pi H/\Gamma^2}. \]
Since both $H$ and $\ell$ are constant, the last expression is conserved and its value at $(R_*,\pi)$ is $4$. Every trajectory approaching this uniformly rotating rhombus lies on this energy level. We select the branch with $R_*<R<1$ and $\pi<\phi<2\pi$. The scalar equation below establishes its time direction and its two limits.

Put $s=\sqrt R$. The invariant and the choice $\sin\phi<0$ give
\[ \cos\phi=\frac{s^4+4s^3-4s+1}{2s^2},\qquad \sin\phi=-\frac{s^2+2s-1}{2s^2}\sqrt{(1-s^2)(s^2+4s-1)}. \]
These formulas specify a unique continuous function $\phi(s)\in(\pi,2\pi)$ for $\sqrt2-1<s<1$. Substituting the sine into the first equation of \eqref{eq:shape-ode} yields
\begin{equation} \label{eq:scalar-orbit}
\dot s=\frac{\Omega}{4}\frac{(1-s^2)^{3/2}}{s^2}(s^2+2s-1)\sqrt{s^2+4s-1}.
\end{equation}
The solution is specified by the explicit quadrature
\[ t=\frac4\Omega\int_{s(0)}^{s(t)}\frac{\sigma^2\,d\sigma}{(1-\sigma^2)^{3/2}(\sigma^2+2\sigma-1)\sqrt{\sigma^2+4\sigma-1}}. \]
The integrand is positive. It has a simple pole at $\sqrt2-1$ and is comparable to $(1-\sigma)^{-3/2}$ near $1$. Thus the integral tends to $-\infty$ and $+\infty$ at the respective endpoints. It gives a unique increasing $s(t)$ modulo time translation, and the function $\phi(t)$ is given by the preceding trigonometric formulas. The derivative of the invariant with respect to $\phi$ is nonzero on $\pi<\phi<2\pi$, so differentiating its constant value together with the equation for $R=s^2$ recovers the second equation in \eqref{eq:shape-ode}. Integrating the third equation gives $\theta(t)$, and reconstruction of $z_1,z_2$ gives the asserted trajectory.

Differentiating the right side of \eqref{eq:scalar-orbit} at $s=\sqrt2-1$ gives $2\sqrt2\Omega$. Hence $s-(\sqrt2-1)$ is a positive constant times $e^{2\sqrt2\Omega t}(1+o(1))$ as $t\to-\infty$. The sine formula gives $y/x\to1$. The phase equation shows that $\dot\theta+\Omega$ is integrable on $(-\infty,0]$, so $\theta(t)+\Omega t$ has a limit there. This proves convergence to a rotating rhombus with an appropriate phase.

To verify all distance claims, the two squared opposite-sign distances have sum and product
\[ |z_1-z_2|^2+|z_1+z_2|^2=\frac{2\ell^2(1+s^2)}{1-s^2},\qquad |z_1-z_2|^2|z_1+z_2|^2=\frac{4\ell^4s}{1-s^2}. \]
Solving this quadratic gives
\begin{equation} \label{eq:pair-distances}
\min\{|z_1-z_2|^2,|z_1+z_2|^2\}=\frac{2\ell^2}{1+s},\qquad \max\{|z_1-z_2|^2,|z_1+z_2|^2\}=\frac{2\ell^2s}{1-s}.
\end{equation}
The same-sign distances are $2\ell/\sqrt{1-s^2}$ and $2\ell s/\sqrt{1-s^2}$. Since $\sqrt2-1<s<1$, none of these four expressions can approach zero. The smaller opposite-sign distance tends to $\ell$, and the larger one diverges.

The large-time rate also follows directly from \eqref{eq:scalar-orbit}. Expanding at $R=1$ gives
\[ \frac{d}{dt}(1-R)=-2\Omega(1-R)^{3/2}\bigl(1+O(1-R)\bigr),\qquad \frac{d}{dt}(1-R)^{-1/2}=\Omega\bigl(1+O(1-R)\bigr). \]
First this proves $1-R\asymp t^{-2}$. The error in the second identity is consequently integrable, so for a constant $t_*$,
\[ (1-R)^{-1/2}=\Omega(t-t_*)+O(t^{-1}),\qquad 1-R=\frac{1}{\Omega^2(t-t_*)^2}+O(t^{-4}). \]
Here $O$ bounds the absolute error by a constant times the displayed quantity, and $\asymp$ bounds the ratio above and below by positive constants.

In the labeling used to construct the orbit, $z_2/z_1\to-1$. Relabel the two negative vortices so that $z_1-z_2$ is the smaller separation. Write $Z=(z_1+z_2)/2$ for the center of this dipole. The other dipole has center $-Z$. Formula \eqref{eq:pair-distances} and the preceding rate give $|Z|\asymp t$ and $|z_1-z_2|\asymp\ell$. A positive-negative pair with bounded internal distance, viewed at distance comparable to $t$, has velocity $O(t^{-2})$ and velocity gradient $O(t^{-3})$. These bounds follow by subtracting its two kernels and applying the mean value theorem to $1/\overline z$ and its derivative. The self-induced velocities of the two members of one dipole are equal to $-i\Gamma/(2\pi\overline{z_1-z_2})$. It follows that
\[ \frac{d}{dt}(z_1-z_2)=O(t^{-3}),\qquad \dot Z=-\frac{i\Gamma}{2\pi\overline{z_1-z_2}}+O(t^{-2}). \]
Integrating the first estimate and using the limiting distance proves $z_1-z_2=\ell e^{i\theta_\infty}+O(t^{-2})$. Substitution in the second proves $\dot Z=-i\Omega\ell e^{i\theta_\infty}+O(t^{-2})$. One more integration gives \eqref{eq:dipole-asymptotics}.

It remains to identify $Z_\infty$. The exact impulse identity and the two asymptotic expansions just proved give
\[ \ell^2=|z_1|^2-|z_2|^2=2\operatorname{Re}\bigl(Z\overline{z_1-z_2}\bigr)=2\ell\operatorname{Re}(e^{-i\theta_\infty}Z_\infty)+O(t^{-1}). \]
Indeed, the term of order $t$ in this real part vanishes because $-i\Omega\ell^2t$ is purely imaginary, and multiplying the $O(t^{-2})$ separation error by $Z=O(t)$ gives $O(t^{-1})$. Letting $t\to+\infty$ proves $\operatorname{Re}(e^{-i\theta_\infty}Z_\infty)=\ell/2$.

A time translation changes only the component of $e^{-i\theta_\infty}Z_\infty$ perpendicular to the limiting dipole separation. The later argument uses the coefficient of $t$ in $M_2$ and a bounded constant term, so no further normalization of $Z_\infty$ is needed.
\end{proof}

\begin{remark}
Time translation and rotation preserve the equations. Thus, if $z_j(t)$ is the trajectory just constructed, then $e^{i(\varphi-\Omega T)}z_j(t-T)$ is also a solution for any $T,\varphi\in\mathbb R$.
\end{remark}

\subsection{How separating dipoles force a rotating quadrupole}

\paragraph{The far field of a quadrupole.}
The second complex moment of the four points is
\[ M_2=\Gamma z_1^2+\Gamma(-z_1)^2-\Gamma z_2^2-\Gamma(-z_2)^2=2\Gamma(z_1^2-z_2^2). \]
For the uniformly rotating rhombus solution \eqref{eq:rhombus}, this gives 
\[ M_2=2\sqrt2\Gamma\ell^2e^{2i(\theta_0-\Omega t)}. \]
Along the separating trajectory, the identity $M_2=4\Gamma Z(z_1-z_2)$ and the estimates just proved give a constant $M_{2,\infty}\in\mathbb C$ such that
\begin{equation} \label{eq:moment-growth}
M_2(t)=-\frac{2i}{\pi}\Gamma^2e^{2i\theta_\infty}t+M_{2,\infty}+O(t^{-1}),\qquad \dot M_2(t)=-\frac{2i}{\pi}\Gamma^2e^{2i\theta_\infty}+O(t^{-2}).
\end{equation}
For $|z|>\max\{|z_1|,|z_2|\}$, the velocity from the quadrupole is
\[ \begin{aligned} u(z)&=\frac{i\Gamma}{2\pi}\left(\frac1{\overline z-\overline z_1}+\frac1{\overline z+\overline z_1}-\frac1{\overline z-\overline z_2}-\frac1{\overline z+\overline z_2}\right)\\ &=\frac{i\overline{M_2}}{2\pi\overline z^3}\frac1{(1-\overline z_1^2/\overline z^2)(1-\overline z_2^2/\overline z^2)}. \end{aligned} \]
The zero total circulation and the central symmetry have canceled the first two orders. If $|z| \gg |z_1|, |z_2|$,
\begin{equation} \label{eq:distant-strain}
\partial_{\overline z}u(z)=-\frac{3i\overline{M_2}}{2\pi\overline z^4}+O\!\left(\frac{|M_2|\max(|z_1|^2,|z_2|^2)}{|z|^6}\right).
\end{equation}
Here $\partial_{\overline z}=(\partial_{x_1}+i\partial_{x_2})/2$. Since this velocity depends analytically on $\overline z$, its linear relative velocity near a point $z$ is $S\overline\zeta$, where $S=\partial_{\overline z}u(z)$ and $\zeta$ is the displacement from $z$. For a quadrupole centered at $d_s$ and a second one centered at $d_r$, the leading strain at the latter is obtained by substituting $z=d_r-d_s$ in \eqref{eq:distant-strain}.

\paragraph{Deformation of the rhombus by external strain.}
Compute how a prescribed external strain acts on the rhombus. Add the velocity $S(t)\overline z$ to each point equation. At the rhombus, $z_1$ has argument $\theta$ and $z_2$ has argument $\theta+\pi/2$. The external contributions to their logarithmic velocities are therefore $Se^{-2i\theta}$ and $-Se^{-2i\theta}$. The definitions of $R,\phi$ give
\[ \left.\dot R\right|_{\mathrm{ext}}=-4R_*\operatorname{Re}(Se^{-2i\theta}),\qquad \left.\dot\phi\right|_{\mathrm{ext}}=-4\operatorname{Im}(Se^{-2i\theta}). \]
The linearized equation for the unstable shape variable $q$ is
\begin{equation} \label{eq:forced-growing}
\dot q=2\sqrt2\Omega q-2\operatorname{Re}(Se^{-2i\theta})-\sqrt2\operatorname{Im}(Se^{-2i\theta}).
\end{equation}

\paragraph{Cascade of quadrupole separation.}
Now distinguish the source and receiving configurations by subscripts $s$ and $r$. The separating source has circulation magnitude $\Gamma_s$, length $\ell_s$, and frequency $\Omega_s=\Gamma_s/(2\pi\ell_s^2)$. The receiving rhombus has frequency $\Omega_r$ and phase $\theta_r-\Omega_rt$. We first make an ideal linear calculation with fixed centers $d_s,d_r$. Put $\vartheta=\arg(d_r-d_s)$ and retain the term linear in time from \eqref{eq:moment-growth}. Its leading strain at the rhombus is
\[ S(t)=\frac{3\Gamma_s^2}{\pi^2|d_r-d_s|^4}\,t\,e^{i(4\vartheta-2\theta_\infty)}. \]
Let $f(t)$ denote the last two terms in \eqref{eq:forced-growing}, with this prescribed $S(t)$ and $\theta=\theta_r-\Omega_rt$. The solution of this linear equation is
\[ q(t)=e^{2\sqrt2\Omega_rt}\left(q(0)+\int_0^t e^{-2\sqrt2\Omega_rs}f(s)\,ds\right). \]
The weighted forcing integral contributes to the coefficient of the growing exponential. Using $\int_0^\infty te^{-at}\,dt=a^{-2}$ for $\operatorname{Re}a>0$ gives
\begin{equation} \label{eq:positive-coefficient}
\int_0^\infty e^{-2\sqrt2\Omega_rt}f(t)\,dt=\frac{\sqrt3\,\Gamma_s^2}{2\sqrt2\pi^2\Omega_r^2|d_r-d_s|^4}\cos\!\left(4\vartheta-2\theta_r-2\theta_\infty+\arctan\frac1{\sqrt2}-\pi\right).
\end{equation}
Its positive maximum is attained at
\[ 4\vartheta-2\theta_r-2\theta_\infty=\pi-\arctan \frac{1}{\sqrt2}\pmod{2\pi}. \]
At this orientation, an error of size $\delta$ in the displayed phase combination changes the cosine only by $O(\delta^2)$. Denote the positive maximum by $\kappa=\sqrt3\,\Gamma_s^2/(2\sqrt2\pi^2\Omega_r^2|d_r-d_s|^4)$. The same integration, now over $[0,t]$, gives
\[ q(t)=(q(0)+\kappa)e^{2\sqrt2\Omega_rt} -\kappa\bigl[(1+2\sqrt2\Omega_rt)\cos(2\Omega_rt)+2\Omega_rt\sin(2\Omega_rt)\bigr]. \]
In particular, $|q(t)-(q(0)+\kappa)e^{2\sqrt2\Omega_rt}|\le\kappa(1+2\sqrt3\Omega_rt)$ for $t\ge0$. Thus an initially unperturbed rhombus, $q(0)=0$, has a positive growing response. Fix a small threshold $q_*>0$. When $\kappa/q_* \ll 1$, its first crossing time satisfies
\[ T=\frac1{2\sqrt2\Omega_r}\log\frac{q_*}{\kappa}+O(\Omega_r^{-1}). \]

\paragraph{Controlling the errors and completing the separation.}
The preceding calculation suggests two steps. First the forcing increases the unstable shape variable to a fixed small threshold $q_*$. Then the four points follow the separating trajectory to two well-separated dipoles. We explain the estimates needed for each step.

For the first step, retain the ideal forcing $f$ above and write the actual equation as $\dot q=2\sqrt2\Omega_rq+f+h$. Thus $h$ contains the nonlinear terms and the errors in the prescribed forcing. The finite-time formula for $q$ gives
\[ \left|e^{-2\sqrt2\Omega_rT}q(T)-q(0)-\kappa\right|\le\int_0^T e^{-2\sqrt2\Omega_rt}|h(t)|\,dt+\kappa(1+2\sqrt3\Omega_rT)e^{-2\sqrt2\Omega_rT}. \]
The last term bounds the tail of the ideal forcing integral. Near the predicted crossing time, it is $o(\kappa)$ when $\kappa \ll q_*$. If $q(0)+\kappa\asymp\kappa>0$ and the weighted integral of $|h|$ is sufficiently small compared with $\kappa$, then $q(T)\asymp\kappa e^{2\sqrt2\Omega_rT}$.

The largest part of $h$ comes from replacing the source's second moment by its linear growth. Start the source at a fixed point sufficiently far along the separating trajectory. Formula~\eqref{eq:moment-growth} gives $|M_2(0)|\le C_{\mathrm{sec}}\Gamma_s\ell_s^2$ and $|\dot M_2+(2i/\pi)\Gamma_s^2e^{2i\theta_\infty}|\le\varepsilon\Gamma_s^2$, and hence
\[ \left|M_2(t)+\frac{2i}{\pi}\Gamma_s^2e^{2i\theta_\infty}t\right|\le C_{\mathrm{sec}}\Gamma_s\ell_s^2+\varepsilon\Gamma_s^2t. \]
By \eqref{eq:distant-strain}, their contributions to the weighted forcing integral are bounded by
\[ \frac{C_{\mathrm{sec}}\Gamma_s\ell_s^2}{\Omega_r|d_r-d_s|^4}+\frac{C\varepsilon\Gamma_s^2}{\Omega_r^2|d_r-d_s|^4}\le\left(C_{\mathrm{sec}}\frac{\Omega_r}{\Omega_s}+C\varepsilon\right)\kappa. \]
Here we used $\Gamma_s=2\pi\Omega_s\ell_s^2$. The constant $C_{\mathrm{sec}}$ depends on the chosen source configuration, whereas $C$ is universal. Fixing the initial source configuration and then choosing $\Omega_s/\Omega_r$ sufficiently large makes this error small relative to $\kappa$.

For the second step, fix two points $\sqrt{2} - 1 < s_0<s_1 < 1$ of the scalar orbit~\eqref{eq:scalar-orbit}, where $q(s_0)=q_*$ and the two dipoles are well separated at $s=s_1$. On this fixed interval, $c\Omega_r\le\dot s \le C\Omega_r$, so the travel time is at most $(s_1-s_0)/(c\Omega_r)$. The same $O(\Omega_r^{-1})$ transition bound holds for sufficiently small perturbations of all four positions and their equations, by continuous dependence along this collision-free segment. Adding this time to the logarithmic growth time gives the total delay
\[ \frac1{2\sqrt2\Omega_r}\left[\log\left(\frac{\Omega_r^2|d_r-d_s|^4}{\Gamma_s^2}\right)+O(1)\right]. \]
Set $r=n$ and $s=n+1$ and recall the scales \eqref{eq:scales}. Since $\Omega_{n+1}=B\Omega_n$, $|d_n-d_{n+1}|=2^{-n-1}$, and $\Gamma_{n+1}=2\pi\Omega_{n+1}e^{-2\Omega_{n+1}^\alpha}$, the logarithm equals $4\Omega_{n+1}^\alpha+O(n+\log B+1)$. This gives \eqref{eq:time-delay}.

Sections~\ref{transfer:material-section} and~\ref{array:section} establish these estimates for the coupled patch solutions. They apply the weighted growth argument to the unstable shape difference between the two solutions, control its initial value and the remaining errors, and then show that the second configuration follows the fixed separating trajectory.

\section{Small patch approximation of point vortices} \label{meta:section}

We replace point vortices by small patches. First, the fast rotation of each patch limits its deformation by the external velocity. Second, the patch centroids and moments satisfy the point-vortex equations up to controlled errors. Third, we compare the boundary deformations, centroids, and moments of two patch solutions. The resulting centroid and moment equations are collected in Theorem~\ref{meta:reduction}. Their geometric hypotheses are verified in the later bootstrap arguments in Sections~4--7.

\subsection{Contour equation and label--deformation coordinates $(\eta,z)$ near a disk}

A uniform patch of circulation $\Gamma$ and area $\pi a^2$ has vorticity $\Gamma/(\pi a^2)$. For vorticity $\sigma\Omega_{\mathrm{patch}}$, $\sigma=\pm1$, let $w(t)$ be its area centroid and let $X(t,\theta)$ follow its boundary particles. Remove translation and the disk's internal rotation by
\[ F(t,\theta)=e^{-i\sigma\Omega_{\mathrm{patch}}t/2}\frac{X(t,\theta)-w(t)}a,\qquad \theta\in\mathbb T=\mathbb R/(2\pi\mathbb Z). \]
The normalized domain has area $\pi$ and centroid zero. An initial disk with standard labels has $F(0,\theta)=e^{i\theta}$. Unless specified otherwise, we use the contour norm
\[ \|f\|_{C^{1,1/2}}=\|f\|_{L^\infty}+\|f'\|_{L^\infty}+\sup_{\theta\ne\beta}\frac{|f'(\theta)-f'(\beta)|}{\operatorname{dist}_{\mathbb T}(\theta,\beta)^{1/2}},\qquad \operatorname{dist}_{\mathbb T}(\theta,\beta)=\min_{k\in\mathbb Z}|\theta-\beta+2\pi k|. \]
We use $K(z)=i/(2\pi\bar z)$. The external velocity of a patch is the sum of the velocities induced by all other patches, including those in the same quadrupole. We estimate its Taylor coefficients with weights $2^k$, corresponding to the disk $B(w,2a)$ after scaling lengths by $a$. Constants may depend on the fixed contour neighborhood, but not on the number of patches or their vorticity magnitudes.

Write $D_F$ for the bounded interior of $F$. By a normalized contour we mean a positively oriented $C^{1,1/2}$ embedding near $F_0(\theta)=e^{i\theta}$ satisfying
\[ |D_F|=\pi,\qquad \int_{D_F}w\,dA(w)=0,\qquad |F(\theta)-F(\beta)|\ge c\operatorname{dist}_{\mathbb T}(\theta,\beta). \]
The last inequality prevents distinct boundary labels from approaching the same point too closely. Among these contours,
\[ E(\eta)(\theta)=e^{i(\theta+\eta(\theta))},\qquad \eta\text{ real},\quad \|\eta\|_{C^{1,1/2}}\text{ small}, \]
all describe the unit circle with different labels. Since $1+\eta'>0$, the labels preserve orientation. We call these contours parameterized unit circles. The coordinates below describe all nearby normalized contours, while $E(\eta)$ describes only the circles.

\begin{lemma}[Contour equation] \label{core:analytic-contour}
Let $X=C^{1,1/2}(\mathbb T;\mathbb C)$, regarded as a real Banach space, and let $F_0(\theta)=e^{i\theta}$. In a neighborhood of $F_0$ consisting of positively oriented embeddings satisfying the separation bound above, the unit-vorticity boundary velocity
\[ V(F)(\theta)=-\frac{1}{2\pi}\int_0^{2\pi}\log|F(\theta)-F(\beta)|F'(\beta)\,d\beta \]
is a real-analytic map from $X$ to $X$. On a fixed smaller neighborhood its derivatives satisfy $\|D^mV\|\le C m!r^{-m}$ for some $r>0$. Area and the two real components of the area centroid are analytic contour functionals. The normalized contours near $F_0$ admit local analytic coordinates with an analytic inverse.
\end{lemma}

\begin{proof}
Regard $F$ and its conjugate as independent complex functions $F_+,F_-$ near $e^{i\theta},e^{-i\theta}$. On a sufficiently small $C^1$ neighborhood,
\[ \sup_{\theta\ne\beta}\left|\frac{F_\pm(\theta)-F_\pm(\beta)}{e^{\pm i\theta}-e^{\pm i\beta}}-1\right|<\frac12. \]
Hence
\[ \log\frac{F_\pm(\theta)-F_\pm(\beta)}{e^{\pm i\theta}-e^{\pm i\beta}}=\sum_{m\ge1}\frac{(-1)^{m+1}}m\left(\frac{F_\pm(\theta)-F_\pm(\beta)}{e^{\pm i\theta}-e^{\pm i\beta}}-1\right)^m. \]
The uniformly convergent series define the complexified logarithm and make the velocity holomorphic on each complex line. The original logarithm is the circle logarithm plus half their sum. Since the circle logarithm is integrable, $V$ is locally bounded in $L^\infty$. Differentiating once gives
\[ \partial_\theta V=-\frac{1}{4\pi}\left(F_+'(\theta)\operatorname{pv}\!\int\frac{F_+'(\beta)}{F_+(\theta)-F_+(\beta)}\,d\beta+F_-'(\theta)\operatorname{pv}\!\int\frac{F_+'(\beta)}{F_-(\theta)-F_-(\beta)}\,d\beta\right). \]
The first integral is constant. In the second, write $F_+'=(F_+'/F_-')F_-'$. Since $F_-'$ is bounded away from zero, the quotient is uniformly bounded in $C^{1/2}$.

For a nearby curve $F$, the Cauchy kernel $F'(\beta)/(F(\theta)-F(\beta))$ is bounded by $C|\theta-\beta|^{-1}$. Replacing $\theta$ by $\theta+h$ changes it by at most $C|h||\theta-\beta|^{-2}$ when $|\theta-\beta|>2|h|$. Its principal-value integral on constants is constant. Its truncated integrals are uniformly bounded because subtracting $-(\beta-\theta)^{-1}$ leaves a local remainder $O(|\theta-\beta|^{-1/2})$. For a density $f\in C^{1/2}$, subtract $f(\theta)$. The near and far parts of the H\"older difference are bounded, respectively, by
\[ C\|f\|_{C^{1/2}}|h|^{1/2},\qquad C\|f\|_{C^{1/2}}|h|\int_{|h|}^1s^{-3/2}\,ds\le C\|f\|_{C^{1/2}}|h|^{1/2}. \]
The difference of the subtracted constants is multiplied by a bounded truncated integral. Thus the Cauchy integral is uniformly bounded on $C^{1/2}$. Both complexified curves satisfy these estimates, so $V$ is locally bounded in $C^{1,1/2}$.

Apply the one-variable Cauchy formula along each complex line to the velocity and its first derivative. The uniform H\"older bound gives an $X$ bound $Cr^{-m}$ for the Taylor coefficient of order $m$. The resulting series converges in $X$ and agrees with the velocity in $C^0$. Polarization, with a smaller $r$, gives $\|D^mV\|\le Cm!r^{-m}$.

Green's formula expresses area and first moments as polynomial contour functionals,
\[ \int_{D_F}z^k\,dA=\frac1{2i}\int_0^{2\pi}F^k\bar F F'\,d\theta. \]
The area is nonzero near $F_0$, so dividing the first moment by area proves analyticity of the centroid. Write $F=e^{i\theta}(1+h+ik)$, with $h,k$ real. At the circle, the differentials of area and the first complex moment are $\int h\,d\theta$ and $\int e^{i\theta}h\,d\theta$. Their restriction to $h=1,\cos\theta,\sin\theta$ is invertible. The analytic implicit-function theorem solves for these three radial coefficients in terms of the others. The coordinate map $\chi$ retains the remaining radial modes and all tangential modes. Its inverse inserts the three solved coefficients. Both maps have bounded derivatives on a smaller neighborhood. Since the removed radial directions are invariant under angular translation, uniqueness gives $\chi\bigl(e^{-i\varphi}F(\cdot+\varphi)\bigr)=\chi(F)(\cdot+\varphi)$.
\end{proof}

The bounded first derivative gives local existence and uniqueness for the contour equation with continuous time-dependent analytic external coefficients. For finitely many separated patches, each interaction is a nonsingular boundary integral, analytic in the two contours and their centers. The coupled field is locally Lipschitz, so solutions depend continuously on their initial contours and continue while they remain in the chart and stay separated. Reynolds' formula and the divergence-free boundary velocity give the weak transport equation for each patch indicator. Patch areas and vorticity values are preserved.

The periodic Hilbert transform is defined by $He^{im\theta}=-i\operatorname{sgn}(m)e^{im\theta}$. On the modes $|m|\ge1$, it satisfies $H^2=-I$.

\begin{lemma} \label{core:disk-spectrum}
For a unit-vorticity patch, let $\tau$ denote time in the spatial frame rotating at angular velocity $1/2$. The unforced equation for normalized contours is $F_\tau=G(F)$, where $G(F)=V(F)-iF/2$. Every $E(\eta)$ is stationary. There are local analytic coordinates $(\eta,z)$, with $\eta,z$ real in $C^{1,1/2}$ and $\widehat z(0)=\widehat z(1)=\widehat z(-1)=0$, such that $z=0$ corresponds exactly to $F=E(\eta)$. Thus $\eta$ records the labels on a circle and $z$ measures deformation from it. In these coordinates,
\[ \eta_\tau=O(\|z\|^2),\qquad z_\tau=-\tfrac12Hz+Q(\eta)[z,z]+O(\|z\|^3). \]
The linearized equations at any $(\eta,0)$ are $\eta_\tau=0$ and $z_\tau=-Hz/2$. After allowing complex-valued perturbations, the deformation part is the sum of the eigenspaces with eigenvalues $i/2$ and $-i/2$. The projections onto these two spaces and the stationary label space are uniformly bounded and depend analytically on $\eta$. The remainders and the bilinear map $Q$ have uniformly bounded analytic coefficients on a slightly larger fixed neighborhood.
\end{lemma}

\begin{proof}
The disk velocity is $iz/2$ on its boundary, so $G(E(\eta))=0$. At $F_0$, write a perturbation as $e^{i\theta}(h+ik)$, with $h,k$ real. For the Fourier coefficients $\widehat h(m)=(2\pi)^{-1}\int_0^{2\pi}h(\theta)e^{-im\theta}\,d\theta$, the linearized area and centroid constraints are $\widehat h(0)=\widehat h(1)=\widehat h(-1)=0$. For $|z|<1$, the domain variation under a radial displacement $h$ is
\[ \frac{i}{2\pi}\int_0^{2\pi}\frac{h(\beta)}{\bar z-e^{-i\beta}}\,d\beta=-i\sum_{m\ge1}\widehat h(-m)\bar z^{m-1}. \]
Take the interior trace of this formula and combine it with the interior evaluation-point derivative $ie^{i\theta}(h+ik)/2$, which cancels the rotating-frame term in $DG$. Equivalently, differentiating the contour formula and its circle Fourier series gives
\[ DG(F_0)[e^{i\theta}(h+ik)]=-ie^{i\theta}\sum_{m\le-2}\widehat h(m)e^{im\theta}=e^{i\theta}\left(-\tfrac12Hh-\tfrac i2h\right). \]
The resulting linear equations are
\[ h_\tau=-\tfrac12Hh,\qquad k_\tau=-\tfrac12h. \]
Thus $\eta=k+Hh$ is stationary, and the positive and negative modes of $h$ have eigenvalues $i/2$ and $-i/2$, respectively. The Hilbert transform is bounded by the cancellation estimate in Lemma~\ref{core:analytic-contour} and commutes with differentiation. Hence $(h,k)\mapsto(h,k+Hh)$ is a bounded change of coordinates. Moreover,
\[ e^{-\tau H/2}=\cos(\tau/2)I-\sin(\tau/2)H,\qquad \|z\|_*:=\sup_{r\in\mathbb R}\|e^{-rH/2}z\|_{C^{1,1/2}},\qquad \|e^{-\tau H/2}z\|_*=\|z\|_*. \]
The norm $\|\cdot\|_*$ is equivalent to the $C^{1,1/2}$ norm.

In the original $(h,k)$ coordinates, the projections onto these three spaces are
\[ P_0(h,k)=(0,k+Hh),\qquad P_\pm(h,k)=(\Pi_\pm h,-H\Pi_\pm h),\qquad \Pi_\pm=\frac12(I\pm iH). \]
These formulas allow complex-valued $(h,k)$ satisfying the linearized constraints. For real inputs, $P_0$ is real and $P_-$ is the conjugate of $P_+$. The oscillating deformation has both components $(h,-Hh)$ and evolves by $-H/2$.

Fix $\psi(\theta)=\theta+\eta(\theta)$, so that $F_0\circ\psi=E(\eta)$. A change of variable gives $G(F\circ\psi)=G(F)\circ\psi$. Composition with $\psi$ is bounded on $C^{1,1/2}$, since
\[ [(f\circ\psi)']_{1/2}\le [f']_{1/2}\|\psi'\|_{L^\infty}^{3/2}+\|f'\|_{L^\infty}[\psi']_{1/2}. \]
The same estimate holds for $\psi^{-1}$ in the fixed small neighborhood. Thus a fixed change of labels conjugates $DG(F_0)$ to the differential at any nearby parameterized disk.

Use the constraint chart $\chi$ from Lemma~\ref{core:analytic-contour}. Write $e(\eta)=\chi(E(\eta))$ and let $A(\eta)$ be the differential of the field in this chart at $e(\eta)$. Fixed relabeling conjugates it to the circle differential, so $A(\eta)(A(\eta)^2+I/4)=0$. Consequently the projections onto its three eigenspaces are
\[ P_0(\eta)=I+4A(\eta)^2,\qquad P_\pm(\eta)=-2A(\eta)^2\mp iA(\eta). \]
They sum to $I$, have pairwise zero products, and satisfy $AP_0=0$ and $AP_\pm=\pm(i/2)P_\pm$. They are analytic and uniformly bounded. Define
\[ T(\eta)z=\sum_{\pm}P_\pm(\eta)(\Pi_\pm z,-H\Pi_\pm z),\qquad A(\eta)T(\eta)=-\tfrac12T(\eta)H. \]
Since $T(\eta)$ and $P_\pm(\eta)$ differ by $O(\|\eta\|)$ from their values at zero, $T(\eta)$ maps the $z$ variables invertibly onto the range of $P_+(\eta)+P_-(\eta)$, with uniformly bounded inverse. The map $\mathcal C(\eta,z)=\chi^{-1}(e(\eta)+T(\eta)z)$ has invertible differential at $(0,0)$ because its $\eta$ derivative gives the stationary label changes and its $z$ derivative gives the two oscillating parts. Thus $\mathcal C$ gives local analytic coordinates for normalized contours. The field in this chart, again denoted by $G$, satisfies
\[ G(\eta,0)=0,\qquad DG(\eta,0)=\begin{pmatrix}0&0\\0&-H/2\end{pmatrix}. \]
\[ G(\eta,z)=\bigl(O(\|z\|^2),-\tfrac12Hz+Q(\eta)[z,z]+O(\|z\|^3)\bigr). \]
Coordinate changes in this chart preserve area and centroid after applying $\mathcal C$.

Define angular translation by $S_\varphi f(\theta)=f(\theta+\varphi)$, acting on each scalar component, and the circle-fixing action by $\mathcal T_\varphi F(\theta)=e^{-i\varphi}F(\theta+\varphi)$. The constraint chart satisfies $\chi(\mathcal T_\varphi F)=S_\varphi\chi(F)$, and $E(S_\varphi\eta)=\mathcal T_\varphi E(\eta)$. Equivariance of the original field and the polynomial projection formulas therefore give
\[ \begin{aligned} A(S_\varphi\eta)&=S_\varphi A(\eta)S_\varphi^{-1},&P_j(S_\varphi\eta)&=S_\varphi P_j(\eta)S_\varphi^{-1},\\ T(S_\varphi\eta)S_\varphi&=S_\varphi T(\eta). \end{aligned} \]
Consequently,
\[ \mathcal C(S_\varphi\eta,S_\varphi z)=\mathcal T_\varphi\mathcal C(\eta,z). \]
\end{proof}

Fix $\delta<1/8$ sufficiently small that the preceding coordinates are defined whenever
\begin{equation} \label{meta:boundary-chart}
\|\eta\|_{C^{1,1/2}}<\delta,\qquad \|F-e^{i(\cdot+\eta)}\|_{C^{1,1/2}}<\delta.
\end{equation}
The first inequality gives $1-\delta<1+\eta'<1+\delta$, so the labels preserve orientation. The two inequalities give a uniform chord-arc bound and place the patch in $B(w,2a)$. We retain these material labels when the quadrupole changes configuration.

\subsection{Normal form with external field}

We change coordinates to remove the unforced quadratic term in the $z$ equation and the value and first $z$ derivative of the external forcing at $z=0$. The quadratic term in the $\eta$ equation remains.

\begin{lemma} \label{core:normal-form}
Let $L>0$ and let a positive-vorticity patch of area $\pi a^2$ have centroid $d(t)$. Assume its external velocity $u_{\mathrm{ext}}(t,\cdot)$ is curl-free and divergence-free on $B(d(t),2a)$, and write
\[ \frac{u_{\mathrm{ext}}(t,d(t)+a\zeta)-u_{\mathrm{ext}}(t,d(t))}{a}=\sum_{k\ge1}a_k(t)\bar\zeta^k. \]
Assume the coefficients satisfy
\[ \sum_{k\ge1}2^k|a_k(t)|\le L,\qquad \sum_{k\ge1}2^k|\dot a_k(t)|\le L^2. \]
Let $\Omega_{\mathrm{patch}}>0$ be the patch vorticity and put $\varepsilon=L/\Omega_{\mathrm{patch}}$. For sufficiently small $\varepsilon$, in time $\tau=\Omega_{\mathrm{patch}} t$ and the spatial frame rotating at speed $\Omega_{\mathrm{patch}}/2$ there are analytic coordinate changes in which
\begin{equation} \label{core:normal-form-equations}
\begin{aligned}
\eta_\tau&=R_\eta(\eta,z,\tau),&z_\tau&=-\tfrac12Hz+R_0(\eta,z,\tau),\\
R_\eta&=O(\|z\|^2+\varepsilon^2),&
R_0&=O(\|z\|^3+\varepsilon^2).
\end{aligned}
\end{equation}
The changes differ by $O(\varepsilon)$ from their zero-forcing versions, including their first two coordinate derivatives. The zero-forcing versions include the unforced quadratic change. The remainders are analytic in $(\eta,z)$ on a fixed neighborhood and satisfy
\[ \|D_{\eta}R_\eta\|\le C(\|z\|^2+\varepsilon^2),\qquad \|D_zR_\eta\|\le C(\|z\|+\varepsilon^2), \]
\[ \|D_{\eta}R_0\|\le C(\|z\|^3+\varepsilon^2),\qquad \|D_zR_0\|\le C(\|z\|^2+\varepsilon^2). \]
The changes and their inverses have uniformly bounded first two derivatives in the contour coordinates. Their time dependence uses $a_k$, and their time derivatives use $a_k$ and $\dot a_k$. We again write $(\eta,z)$ for the final coordinates. With zero external field, $z=0$ still describes $E(\eta)$. The forcing-dependent change moves the coordinates of such a circle by $O(\varepsilon)$.
\end{lemma}

\begin{proof}
\emph{The external field in contour coordinates.}

By antisymmetry of the Biot--Savart kernel, the self-induced velocity has zero patch average. The centroid therefore follows the patch average of the external velocity. The centered $k$th external field is
\[ \overline F^k-\frac1\pi\int_{D_F}\overline w^k\,dA(w). \]
For independent complex variables $\mathbf F=(F_+,F_-)$, define
\[ \begin{aligned} m_{k,-}(\mathbf F)&=\frac1{2\pi i(k+1)}\int_0^{2\pi}F_-^{k+1}F_+'\,d\theta,\\ m_{k,+}(\mathbf F)&=-\frac1{2\pi i(k+1)}\int_0^{2\pi}F_+^{k+1}F_-'\,d\theta. \end{aligned} \]
When $F_+=F$ and $F_-=\overline F$, these are the normalized area averages of $\overline w^k$ and $w^k$. Define the complex vector fields and the forcing in the rotating frame by
\[ X_{k,+}(\mathbf F)=\bigl(F_-^k-m_{k,-}(\mathbf F),0\bigr),\qquad X_{k,-}(\mathbf F)=\bigl(0,F_+^k-m_{k,+}(\mathbf F)\bigr), \]
\[ \begin{aligned} f_a(\tau,\mathbf F)=\frac1{\Omega_{\mathrm{patch}}}\sum_{k\ge1}\bigl(&a_k(t)e^{-i(k+1)\tau/2}X_{k,+}(\mathbf F)\\ &+\overline{a_k(t)}e^{i(k+1)\tau/2}X_{k,-}(\mathbf F)\bigr),\qquad t=\frac{\tau}{\Omega_{\mathrm{patch}}}. \end{aligned} \]
Each field preserves the complexified area and centroid constraints. The area derivative integrates a derivative of $F_-^{k+1}$ or $F_+^{k+1}$, and the potentially nonzero first-moment derivative is $\pi m_{k,\pm}-\pi m_{k,\pm}=0$. When $F_-=\overline F_+$, the components of $f_a$ are conjugate, so they represent the physical velocity. Using $\mathcal C$ also for its complexification, set
\[ X_{k,\pm}^{\mathcal C}(\eta,z)=\bigl(D\mathcal C(\eta,z)\bigr)^{-1}X_{k,\pm}(\mathcal C(\eta,z)). \]
Consider one term $\xi e^{i\nu\tau}X$, where $\xi$ is constant, $X=X_{k,\pm}^{\mathcal C}$, and $\nu=\mp(k+1)/2$. Put $A=\operatorname{diag}(0,-H/2)$ and let $P_0,P_+,P_-$ be its projections. Under the substitution $x=y+\xi e^{i\nu\tau}Y(y)$, the coefficient of $\xi e^{i\nu\tau}$ in the transformed forcing is
\[ X+DG\,Y-DY\,G-i\nu Y. \]
Put $Y_0(\eta)=Y(\eta,0)$ and $Y_1(\eta)=DY(\eta,0)$, and define
\[ B(\eta)=DX(\eta,0)+D^2G(\eta,0)[\,\cdot\,,Y_0(\eta)]. \]
The residual and its first derivative vanish at $z=0$ if
\[ (i\nu-A)Y_0=X(\eta,0),\qquad i\nu Y_1+Y_1A-AY_1=B(\eta). \]
Projecting onto the eigenspaces gives
\[ Y_0=\sum_j\frac{P_jX(\eta,0)}{i\nu-\lambda_j},\qquad P_jY_1P_l=\frac{P_jB(\eta)P_l}{i\nu+\lambda_l-\lambda_j},\qquad \lambda_0=0,\quad\lambda_\pm=\pm i/2. \]
The denominators in $Y_0$ have modulus at least $1/2$. We use the formula for $P_jY_1P_l$ when its denominator is nonzero and verify below that its numerator vanishes otherwise. Since $A$ is constant, differentiating the equation for $Y_0$ gives $(i\nu-A)D_{\eta}Y_0=D_{\eta}X(\eta,0)$. Since $D_{\eta}DG(\eta,0)=0$, this agrees with the equation for $Y_1$ on inputs of the form $(h,0)$. We set $P_jY_1P_l=0$ whenever its denominator vanishes.

\noindent\emph{Vanishing numerators when the denominator is zero.}

Only a forcing term between the two deformation eigenspaces can have zero denominator. We rule it out by computing its angular modes. The circle-fixing action extends to the independent variables by
\[ \mathcal T_\varphi(F_+,F_-)=(e^{-i\varphi}S_\varphi F_+,e^{i\varphi}S_\varphi F_-). \]
Substitution in the contour moments gives
\[ \mathcal T_\varphi X_{k,+}(\mathcal T_\varphi^{-1}\mathbf F)=e^{-i(k+1)\varphi}X_{k,+}(\mathbf F). \]
The field $X_{k,-}$ has the conjugate factor, and the centroid subtraction has the same factor as the corresponding power. Also, $\mathcal T_\varphi G(\mathcal T_\varphi^{-1}\mathbf F)=G(\mathbf F)$. By equivariance of $\mathcal C$, these identities hold in contour coordinates. At the standard circle the projections commute with $S_\varphi$, so
\[ S_\varphi Y_0=e^{-i(k+1)\varphi}Y_0,\qquad S_\varphi B S_\varphi^{-1}=e^{-i(k+1)\varphi}B. \]
The second identity includes $D^2G[\cdot,Y_0]$ by equivariance of $D^2G$ and the first identity.

Let $B_{nm}$ map the coefficients of angular mode $m$ to those of mode $n$. Then
\[ e^{i(n-m)\varphi}B_{nm}=e^{-i(k+1)\varphi}B_{nm}\quad\hbox{for every }\varphi,\qquad B_{nm}=0\ \hbox{unless }n=m-(k+1). \]
For $X_{k,+}$, the denominator can vanish only when $k=1$, $\lambda_l=i/2$, and $\lambda_j=-i/2$. The corresponding numerator would require $m\ge2$ and $n=m-2\le-2$, which is impossible. For $X_{k,-}$ the inequalities are reversed, with the same conclusion. Every nonzero denominator has modulus at least $1/2$.

The calculation first applies to trigonometric polynomials. Fix $0<\beta<1/2$. The Cauchy integrals, chart derivatives, and projections are also bounded on $C^{1,\beta}$. Fourier smoothing converges in that norm and commutes with $\Pi_\pm$, so the numerator vanishes on every $C^{1,1/2}$ input.

Fix $\psi(\theta)=\theta+\eta(\theta)$ and $\Psi(x)=\mathcal C^{-1}(\mathcal C(x)\circ\psi)$, so $\Psi(0)=(\eta,0)$. Relabeling preserves the contour moments, giving $G\circ\Psi=D\Psi\,G$ and $X\circ\Psi=D\Psi\,X$. Choose an affine $Y$ realizing $Y_0(0),Y_1(0)$ and put $\widetilde Y\circ\Psi=D\Psi\,Y$. Since $D^2\Psi[Y,G]=D^2\Psi[G,Y]$,
\[ (X+DG\,\widetilde Y-D\widetilde Y\,G-i\nu\widetilde Y)\circ\Psi=D\Psi(X+DG\,Y-DY\,G-i\nu Y). \]
Thus the residual and its first derivative vanish at $(\eta,0)$. Uniqueness in the value equation gives $\widetilde Y(\eta,0)=Y_0(\eta)$. The derivative equation therefore gives $P_jB(\eta)P_l=0$ whenever $i\nu+\lambda_l-\lambda_j=0$.

The displayed formulas define analytic $Y_0(\eta),Y_1(\eta)$ at $z=0$. Since $Y_1(\eta)(h,0)=D_\eta Y_0(\eta)h$, the extension $Y(\eta,z)=Y_0(\eta)+Y_1(\eta)(0,z)$ has the prescribed value and first derivative at $z=0$.

\noindent\emph{Summation and remainder estimates.}

Shrink the contour neighborhood so that $|F_\pm|\le3/2$, and denote the solutions just constructed by $Y_{k,\pm}$. Conjugation in the complexified chart allows the choice $Y_{k,-}(x)=\overline{Y_{k,+}(\overline x)}$. For each fixed derivative order, the derivatives of the powers and contour moments are bounded by a polynomial in $k$ times $(3/2)^k$. The projection, denominator, and chart bounds therefore give fixed $p,C$ such that
\[ \|Y_{k,+}\|_{C^3}+\|Y_{k,-}\|_{C^3}\le C(1+k)^p(3/2)^k. \]
Here $C^m$ denotes the supremum of the function and its first $m$ Fr\'echet derivatives on a fixed smaller chart neighborhood. Summing over $k\ge1$ and $\sigma\in\{+,-\}$ gives
\[ \sum_{k,\sigma}\frac{|a_k|}{\Omega_{\mathrm{patch}}}\|Y_{k,\sigma}\|_{C^3}\le C\varepsilon,\qquad \sum_{k,\sigma}\frac{|\dot a_k|}{\Omega_{\mathrm{patch}}^2}\|Y_{k,\sigma}\|_{C^2}\le C\varepsilon^2. \]
Here we used $\sup_{k\ge1}(1+k)^p(3/4)^k<\infty$. Set
\[ \mathcal S(\tau,x)=\frac1{\Omega_{\mathrm{patch}}}\sum_{k\ge1}\bigl(a_k(t)e^{-i(k+1)\tau/2}Y_{k,+}(x)+\overline{a_k(t)}e^{i(k+1)\tau/2}Y_{k,-}(x)\bigr),\qquad t=\frac{\tau}{\Omega_{\mathrm{patch}}}. \]
Then $\mathcal S$ is real on the real chart. Differentiating an exponential adds a factor bounded by $k+1$, which is absorbed by increasing $p$. Hence $\|\mathcal S\|_{C^3}+\|\partial_\tau\mathcal S\|_{C^2}\le C\varepsilon$. The part containing $\dot a_k$ is $O(\varepsilon^2)$ in $C^2$, while $\mathcal S$ uses only $a_k$. For small $\varepsilon$, $\|D\mathcal S\|\le1/2$, so $x=y+\mathcal S(\tau,y)$ has an inverse on a smaller fixed neighborhood with uniformly bounded first two derivatives. Products of two corrections, including their required derivatives, are $O(\varepsilon^2)$. The first-order residual and its first $z$ derivative vanish at $z=0$, so the residual is bounded by $C\varepsilon\|z\|^2$.

For the unforced quadratic term in the $z$ equation, define
\[ H_2(\eta)[v,w]=\sum_{j,a,b\in\{+,-\}}\frac{\Pi_jQ(\eta)[\Pi_av,\Pi_bw]}{\lambda_a+\lambda_b-\lambda_j}. \]
Every denominator has modulus at least $1/2$. Under $z=w+H_2(\eta)[w,w]$, the quadratic terms in the $w$ equation cancel because
\[ \Pi_jQ(\eta)[\Pi_aw,\Pi_bw]-(\lambda_a+\lambda_b-\lambda_j)\Pi_jH_2(\eta)[\Pi_aw,\Pi_bw]=0. \]
On a sufficiently small fixed ball, the derivative of this change differs from the identity by less than $1/2$, so its inverse has the same derivative bounds. The term $D_\eta H_2[\eta_\tau][w,w]$ is fourth order in the unforced equation, since $\eta_\tau=O(\|w\|^2)$. Its forcing-dependent part is bounded by $C(\varepsilon\|w\|^2+\varepsilon^2)$. The quadratic term in the $\eta$ equation remains, and no order-$\varepsilon$ constant or linear term is created.

Thus the remainder in the $\eta$ equation is the sum of terms of orders $\|z\|^2$, $\varepsilon\|z\|^2$, and $\varepsilon^2$, with analytic coefficients bounded on a fixed larger complex neighborhood. The remainder in the $z$ equation has the same form, with unforced order $\|z\|^3$. If $\mathcal R(\eta,0)=D_z\mathcal R(\eta,0)=0$, then
\[ \mathcal R(\eta,z)=\int_0^1(1-s)D_z^2\mathcal R(\eta,sz)[z,z]\,ds. \]
This identity, its third-order analogue, and the uniform analytic derivative bounds give the asserted estimates, including the $C\varepsilon^2$ bound for the $\eta$ derivative of the forcing residual. The mixed terms are absorbed by
\[ \varepsilon\|z\|^2\le C(\|z\|^3+\varepsilon^3),\qquad \varepsilon\|z\|\le C(\|z\|^2+\varepsilon^2),\qquad \varepsilon\le1. \]
This proves \eqref{core:normal-form-equations} and its derivative bounds.
\end{proof}

We next verify the coefficient bounds of Lemma~\ref{core:normal-form} from the distances and circulations of the other patches.

\begin{lemma}[External velocity estimates] \label{core:material-derivative}
Consider a finite collection of transported bounded uniform patches with separated $C^{1,1/2}$ boundaries. Let a distinguished one-sign patch of circulation $\Gamma$ and centroid $d(t)$ be contained in $B(d(t),2a)$. Suppose the remaining vorticity $\omega_S$ is at distance at least $\rho\ge8a$ from $d(t)$, and set
\[ u_S=K*\omega_S,\qquad L_1(x)=\int\frac{|\omega_S(y)|}{|x-y|^2}\,dA(y),\qquad L_0=L_1(d)+|\Gamma|/\rho^2, \]
where $x$ is outside the source supports. For the coupled Euler evolution,
\[ |D^ku_S(d)|\le C^{k+1}k!L_1(d)\rho^{-(k-1)},\qquad \left|\frac{d}{dt}D^ku_S(t,d(t))\right|\le C^{k+1}(k+1)!L_0^2\rho^{-(k-1)},\qquad k\ge1. \]
Hence, for $a/\rho$ sufficiently small, the normalized external coefficients $a_k$ of Lemma~\ref{core:normal-form}, with $u_{\mathrm{ext}}=u_S$, satisfy
\[ \sum_{k\ge1}2^k|a_k|\le C L_0,\qquad \sum_{k\ge1}2^k|\dot a_k|\le C L_0^2. \]
The constants are independent of the internal vorticity magnitudes and the number of patches. On a time interval where $L_0$ is uniformly bounded, these are the coefficient bounds of Lemma~\ref{core:normal-form} with $L=C\sup_t L_0(t)$.
\end{lemma}

\begin{proof}
Differentiating $K$ outside the supports gives the spatial derivative bound. If the sources are transported by $u_S$ alone, transport and integration by parts give
\[ (\partial_t+u_S(x)\cdot\nabla_x)D^ku_S(x)=\int D^{k+1}K(x-y)[u_S(x)-u_S(y)]\omega_S(y)\,dA(y). \]
Substituting $u_S=K*\omega_S$ and interchanging $y,z$ gives the double-integral kernel
\[ \begin{aligned} B_k(x,y,z)=\tfrac12\bigl(&D^{k+1}K(x-y)[K(x-z)-K(y-z)]\\ &+D^{k+1}K(x-z)[K(x-y)-K(z-y)]\bigr). \end{aligned} \]
Here the last derivative index is contracted with the velocity. Translate $x$ to zero and put $r=|y|$, $s=|z|$. For $r\le s/2$, $|K(-z)-K(y-z)|\le Cr/s^2$, and the two terms are bounded by
\[ C^{k+1}(k+1)!r^{-k-1}s^{-2},\qquad C^{k+1}(k+1)!s^{-k-2}(r^{-1}+s^{-1})\le C^{k+1}(k+1)!r^{-2}s^{-k-1}. \]
The case $s\le r/2$ is symmetric. When $r$ and $s$ are comparable and $|y-z|\ge cr$, each term is at most $C^{k+1}(k+1)!r^{-k-3}$. In the remaining case the singular part is
\[ \begin{gathered} -\tfrac12[D^{k+1}K(-y)-D^{k+1}K(-z)]K(y-z),\\ |D^{k+1}K(-y)-D^{k+1}K(-z)|\le C^{k+1}(k+1)!|y-z|r^{-k-3}. \end{gathered} \]
The factor $|y-z|$ cancels the singularity of $K(y-z)$, and the nonsingular terms have the same bound. We used $|D^mK(x)|\le C^{m+1}m!|x|^{-m-1}$, absorbing the additional factor $k+2$ into $C^{k+1}$. Thus
\[ |B_k(x,y,z)|\le C^{k+1}(k+1)!\left(|x-y|^{-k-1}|x-z|^{-2}+|x-y|^{-2}|x-z|^{-k-1}\right). \]
At $x=d$, integration gives $C^{k+1}(k+1)!\rho^{-(k-1)}L_1(d)^2$. The identities follow for bounded patch vorticity by first excluding $|y-z|<\epsilon$ and then using this integrable bound to pass to the limit.

The distinguished patch induces velocity at most $C|\Gamma|/|d-y|$ on the source supports. Its additional transport contribution is bounded by
\[ C^{k+1}(k+1)!|\Gamma|\int\frac{|\omega_S(y)|}{|d-y|^{k+3}}\,dA(y)\le C^{k+1}(k+1)!\rho^{-(k-1)}\frac{|\Gamma|}{\rho^2}L_1(d). \]
The centroid equation and the one-sign assumption give
\[ \dot d=\Gamma^{-1}\int\omega_{\rm patch}(z)u_S(z)\,dA(z),\qquad |\dot d-u_S(d)|\le CaL_1(d). \]
Multiplication by the spatial bound for $D^{k+1}u_S$ shows that changing to the path $d(t)$ adds at most
\[ C^{k+1}(k+1)!\rho^{-(k-1)}\frac a\rho L_1(d)^2. \]
Since $a/\rho\le1/8$, the sum is bounded by $C^{k+1}(k+1)!\rho^{-(k-1)}L_0^2$.

The external velocity is anti-holomorphic, so $a_k=a^{k-1}\partial_{\bar z}^ku_S(d)/k!$. The derivative estimates yield geometric series with ratio $2Ca/\rho$, and the time-derivative series has an additional factor $k+1$. Both sums have the asserted bounds when $2Ca/\rho<1/2$.
\end{proof}

\subsection{Uniform boundary control}

Decrease the fixed $\delta$ in \eqref{meta:boundary-chart}, if necessary, so that the coordinate changes of Lemma~\ref{core:normal-form} and their inverses are defined throughout that neighborhood. We keep this choice of $\delta$ below.

\begin{lemma} \label{core:confinement}
Let a vortex patch have uniform vorticity magnitude $\Omega_{\mathrm{patch}}$, area $\pi a^2$, and centroid $d(t)$. Assume that an external velocity field satisfies the bounds of Lemma~\ref{core:normal-form}, with constant $L$. Suppose the initial normalized contour satisfies, for some real $\eta_0$,
\[ \|\eta_0\|_{C^{1,1/2}}\le\delta/2,\qquad \|F(0)-E(\eta_0)\|_{C^{1,1/2}}\le C_0\varepsilon,\qquad \varepsilon=\frac L{\Omega_{\mathrm{patch}}}. \]
There are $c,C,\varepsilon_0>0$, depending on the fixed coordinate maps, $C_0$, and $\delta$, such that for $\varepsilon\le\varepsilon_0$, the normalized contour stays within $C\varepsilon^{2/3}$ of a parameterized unit circle for
\[ 0\le t\le c\frac{\Omega_{\mathrm{patch}}^{1/3}}{L^{4/3}}. \]
The original label coordinate satisfies $\|\eta(t)-\eta_0\|_{C^{1,1/2}}\le\delta/4$. The patch is therefore contained in $B(d(t),a(1+C\varepsilon^{2/3}))$. The condition $\Omega_{\mathrm{patch}}\ge C(T^3L^4+L)$ suffices for a fixed time $T$, provided the external hypotheses hold on $[0,T]$.
\end{lemma}

\begin{proof}
Reflection reduces to positive vorticity and preserves the norms and external hypotheses. Apply Lemma~\ref{core:normal-form}. The map $\mathcal C$ satisfies $\mathcal C(\eta,0)=E(\eta)$ and has bounded derivatives and inverse derivatives. The unforced correction is $O(\|z\|^2)$ and the forcing correction is $O(\varepsilon)$. Hence the coordinates after Lemma~\ref{core:normal-form} satisfy
\[ \|z(0)\|\le C\varepsilon,\qquad \|\eta(0)-\eta_0\|\le C\varepsilon. \]
Put $\rho=\varepsilon^{2/3}$ and use the equivalent norm $\|\cdot\|_*$ of Lemma~\ref{core:disk-spectrum}. By variation of constants,
\[ \|z(\tau)\|_*\le C\varepsilon+C\int_0^\tau(\|z(\sigma)\|_*^3+\varepsilon^2)\,d\sigma. \]
Under the bootstrap assumptions $\|z\|_*\le C_*\rho$ and $0\le\tau\le c\rho^{-2}$, choose $C_*$ large, then $c$ small, and finally $\varepsilon_0$ small. Since $\rho^3=\varepsilon^2$,
\[ \|z(\tau)\|_*\le C\varepsilon+C(C_*^3+1)c\rho\le\frac{C_*\rho}{2}. \]
Also,
\[ \|\eta(\tau)-\eta(0)\|\le C\int_0^{c\rho^{-2}}(C_*^2\rho^2+\varepsilon^2)\,d\sigma\le C(C_*^2+1)c. \]
Choose $c$ so that $C(C_*^2+1)c\le\delta/8$. Returning to the original coordinates changes $(\eta,z)$ by $O(\rho^2+\varepsilon)=O(\varepsilon)$. Decrease $\varepsilon_0$ so that the initial and coordinate-change errors together are at most $\delta/8$ and $C\rho<\delta/2$. In the original coordinates of Lemma~\ref{core:disk-spectrum}, this gives
\[ \|\eta(t)-\eta_0\|_{C^{1,1/2}}\le\delta/4,\qquad \|F(t)-E(\eta(t))\|_{C^{1,1/2}}\le C\rho<\delta/2. \]
Since $\|\eta_0\|_{C^{1,1/2}}\le\delta/2$, both inequalities are strictly inside \eqref{meta:boundary-chart}, so local existence continues the solution. Since $\tau=\Omega_{\mathrm{patch}}t$, the time interval is the one stated. The radial displacement is bounded by the contour distance, which gives the support estimate.
\end{proof}

For $\Omega_{\mathrm{patch}}=A\Omega^4$ and $L=C_0\Omega$, Lemma~\ref{core:confinement} gives time $cA^{1/3}C_0^{-4/3}$ and deformation $C(C_0/A)^{2/3}\Omega^{-2}$. Thus choosing $A$ large and then $\Omega$ large gives the following consequence.

\begin{corollary} \label{meta:confinement}
For an initial disk with standard labels $F(0,\theta)=e^{i\theta}$, $\Omega_{\mathrm{patch}}=A\Omega^4$, and $L=C_0\Omega$, choose $A$ large and then $\Omega$ large. On every initial time interval $[0,T]\subset[0,2]$ where the normalized external coefficients satisfy the bounds of Lemma~\ref{core:normal-form}, the original contour coordinates satisfy
\[ \|\eta(t)\|_{C^{1,1/2}}\le\delta/2,\qquad \|F(t)-e^{i(\cdot+\eta(t))}\|_{C^{1,1/2}}\le C_A\Omega^{-2}<\delta/2. \]
Thus the contour cannot first exit through \eqref{meta:boundary-chart}.
\end{corollary}

For later use, the hypotheses for each patch are, in the physical spatial axes,
\begin{equation} \label{meta:individual-input}
\begin{aligned}
\frac{u_{\mathrm{ext}}(t,w(t)+a\zeta)-u_{\mathrm{ext}}(t,w(t))}{a}&=\sum_{k\ge1}a_k(t)\bar\zeta^k,\\
\sum_{k\ge1}2^k|a_k(t)|&\le C_0\Omega,\qquad \sum_{k\ge1}2^k|\dot a_k(t)|\le C_0^2\Omega^2.
\end{aligned}
\end{equation}
Here $w(t)$ is the patch centroid, $\dot a_k$ differentiates along $w(t)$, and $C_0$ is fixed. The weight $2^k$ controls the series on $|\zeta|\le2$, corresponding to $B(w(t),2a)$, where the external field is curl-free and divergence-free. These coefficients exclude the internal rotation $A\Omega^4/2$. The same $A$ works for a larger set of quadrupole configurations provided \eqref{meta:individual-input} holds with the same $C_0$.

\subsection{Relative estimates in common coordinates}

We compare two distinct patch solutions under different external velocity field in the coordinates of Lemma~\ref{core:normal-form} constructed from the first one's external field. Consider two patches with the same signed circulation $\Gamma$, area $\pi a^2$, and vorticity magnitude $\Omega_{\mathrm{patch}}$. Center each normalized contour at its centroid, use the same initial boundary-label convention, and rotate both frames at the same signed internal rate. Suppose both patches satisfy the hypotheses of Lemma~\ref{core:confinement}, with the same bound $L$ for their normalized external coefficients. Denote these coefficients by $a_k^{(j)}$, $j=1,2$, as defined in Lemma~\ref{core:normal-form} at the respective centroids with the common radius $a$.

\begin{lemma} \label{core:relative}
Put $\rho=(L/\Omega_{\mathrm{patch}})^{2/3}$ and apply the coordinate change of Lemma~\ref{core:normal-form}, constructed from the first patch's normalized external coefficients, to both contours. In these coordinates $(\eta^{(j)},z^{(j)})$, $j=1,2$, write $\triangle$ for the second value minus the first at the same physical time. Define
\[ \triangle z(t)=z^{(2)}(t)-z^{(1)}(t),\qquad W_{\mathrm{bdry}}=\|\triangle z\|_*+\rho\|\triangle\eta\|, \]
where $\|\cdot\|_*$ is the norm of Lemma~\ref{core:disk-spectrum}. Then, on the common time interval from Lemma~\ref{core:confinement}, $0\le t\le c\Omega_{\mathrm{patch}}^{1/3}/L^{4/3}$,
\begin{equation} \label{core:relative-estimate}
W_{\mathrm{bdry}}(t)\le e^{C\Omega_{\mathrm{patch}}\rho^2t}W_{\mathrm{bdry}}(0)+C\int_0^t e^{C\Omega_{\mathrm{patch}}\rho^2(t-r)}\sum_{k\ge1}2^k|\triangle a_k(r)|\,dr.
\end{equation}
Identical initial normalized embeddings give $W_{\mathrm{bdry}}(0)=0$. The estimate requires no time derivative of $\triangle a_k$.
\end{lemma}

\begin{proof}
Let $\tau=\Omega_{\mathrm{patch}}t$ and let $\Phi(\tau,F)$ be the coordinate change of Lemma~\ref{core:normal-form} constructed from the first list of normalized external coefficients. Denote by $f_j(\tau,F)$ the centered external contour field determined by the normalized external coefficients $a_k^{(j)}$. At the same contour $F$, the difference of the transformed fields is
\[ D_F\Phi(\tau,F)\bigl(f_2(\tau,F)-f_1(\tau,F)\bigr),\qquad \left\|D_F\Phi(f_2-f_1)\right\|\le\frac C{\Omega_{\mathrm{patch}}}\sum_{k\ge1}2^k|\triangle a_k|. \]
Indeed, the external contour field, including subtraction of its patch average, is linear in the coefficients. Both equations contain the same $\partial_\tau\Phi(\tau,F)$. Its change with $F$ is included in the state derivatives of the first transformed field, so no derivative of $\triangle a_k$ occurs.

Choose a smaller convex coordinate ball. Confinement and the $O(L/\Omega_{\mathrm{patch}})$ forcing-dependent change give $\|z\|\le C\rho$ at both states and on the segment joining them. The first transformed field is
\[ \eta_\tau=R_\eta,\qquad z_\tau=-\tfrac12Hz+R_0. \]
By Lemma~\ref{core:normal-form}, $L/\Omega_{\mathrm{patch}}=\rho^{3/2}$ gives
\[ \|D_\eta R_\eta\|\le C\rho^2,\qquad \|D_zR_\eta\|\le C\rho,\qquad \|D_\eta R_0\|\le C\rho^3,\qquad \|D_zR_0\|\le C\rho^2. \]
Integrating these derivatives along the segment gives
\[ D^+\|\triangle\eta\|\le C\rho^2\|\triangle\eta\|+C\rho\|\triangle z\|_*+\frac C{\Omega_{\mathrm{patch}}}\sum_{k\ge1}2^k|\triangle a_k|, \]
\[ D^+\|\triangle z\|_*\le C\rho^3\|\triangle\eta\|+C\rho^2\|\triangle z\|_*+\frac C{\Omega_{\mathrm{patch}}}\sum_{k\ge1}2^k|\triangle a_k|. \]
Here $D^+h(\tau)=\limsup_{s\downarrow0}(h(\tau+s)-h(\tau))/s$. For the second inequality we used variation of constants and $\|e^{-\tau H/2}z\|_*=\|z\|_*$ from Lemma~\ref{core:disk-spectrum}. Multiplying the first inequality by $\rho$ and adding yields
\[ D^+W_{\mathrm{bdry}}\le C\rho^2W_{\mathrm{bdry}}+\frac C{\Omega_{\mathrm{patch}}}\sum_{k\ge1}2^k|\triangle a_k|. \]
Gronwall's inequality and $\tau=\Omega_{\mathrm{patch}}t$ prove \eqref{core:relative-estimate}, with prefactor one on the initial value because $e^{-\tau H/2}$ preserves $\|\cdot\|_*$. The differential inequality applies from any later time in the interval. Since both contours use the same map, identical initial embeddings give $W_{\mathrm{bdry}}(0)=0$. Reflection as in Lemma~\ref{core:confinement} gives the assertion for negative vorticity.
\end{proof}

\begin{lemma} \label{core:relative-moments}
For the patches of Lemma~\ref{core:relative}, let $d^{(j)}$ denote their centroids and set
\[ M_k^{(j)}=\int(z-d^{(j)})^k\omega^{(j)}(z)\,dA(z),\qquad I^{(j)}=\int|z-d^{(j)}|^2\omega^{(j)}(z)\,dA(z). \]
Then, for $k\ge2$,
\[ |\triangle M_k|\le C_k|\Gamma|a^kW_{\mathrm{bdry}},\qquad |\triangle I|\le C|\Gamma|a^2\rho W_{\mathrm{bdry}}. \]
Each patch satisfies $|I^{(j)}-\Gamma a^2/2|\le C|\Gamma|a^2\rho^2$. The constants may be chosen so that $C_k\le C(k+1)^2(3/2)^k$, where $C$ depends only on the fixed contour neighborhood and coordinate maps, not on $k$, $a$, or $\Gamma$. For the same harmonic field $U$ on a fixed disk larger than both centered patches,
\[ \left|\frac1\Gamma\int U(z-d^{(2)})\omega^{(2)}(z)\,dA(z)-\frac1\Gamma\int U(z-d^{(1)})\omega^{(1)}(z)\,dA(z)\right|\le Ca^2\sup|D^2U|W_{\mathrm{bdry}}. \]
\end{lemma}

\begin{proof}
For fixed $k\ge2$, Green's formula gives the normalized moment functional
\[ \mathcal M(F)=\frac1\pi\int_{D_F}w^k\,dA(w)=\frac1{2\pi i}\int_0^{2\pi}F^k\overline F\,\partial_\theta F\,d\theta. \]
The common rotation of the two frames multiplies both moments by the same factor of modulus one, so it does not affect their difference in absolute value. On the fixed contour neighborhood, $\|F\|_{L^\infty}\le3/2$ and $\|F\|_{C^{1,1/2}}\le C$. Differentiating the displayed formula once or twice gives
\[ \|D\mathcal M(F)\|+\|D^2\mathcal M(F)\|\le C(k+1)^2(3/2)^k. \]
Indeed, each derivative of $F^k$ contributes at most a factor $k$, while $\overline F$ and $\partial_\theta F$ depend linearly on $F$. The coordinate maps of Lemma~\ref{core:normal-form} have bounded first two derivatives, so the same bound holds in contour coordinates. This functional vanishes on every $E(\eta)$. In the unforced coordinates it therefore satisfies
\[ \mathcal M(\eta,z)=\int_0^1D_z\mathcal M(\eta,rz)z\,dr,\qquad \|D_z\mathcal M\|\le C_k,\qquad \|D_\eta\mathcal M\|\le C_k\|z\|. \]
The forcing-dependent change is $O(L/\Omega_{\mathrm{patch}})=O(\rho^{3/2})$ in $C^2$. Thus the $z$ and $\eta$ derivatives of the normalized $k$th moment are bounded by $C_k$ and $C_k\rho$ on $\|z\|\le C\rho$. Integrating along the segment between the states gives
\[ \frac{|\triangle M_k|}{|\Gamma|a^k}\le C_k\bigl(\|\triangle z\|_*+\rho\|\triangle\eta\|\bigr)=C_kW_{\mathrm{bdry}}. \]

For the averaging error, subtract $U(0)+DU(0)x$. The linear term has zero patch average by the centroid constraint, and the remainder has zero disk average by the mean-value property. Taylor's formula and Cauchy estimates give $z$ and $\eta$ derivatives bounded by $Ca^2\sup|D^2U|$ and $Ca^2\rho\sup|D^2U|$. Integrating along the segment gives the asserted bound. Here $U$ is fixed, and its difference is estimated separately in the centroid equation.

For the radial moment, closeness in $C^1$ gives $\partial_\theta\arg F=\operatorname{Im}(F'/F)>0$. The boundary is therefore a radial graph $a(1+h(\theta))e^{i\theta}$ in centered physical coordinates, with $\|h\|_{L^\infty}\le C\rho$. Area conservation and polar integration give
\[ \int_0^{2\pi}(2h+h^2)\,d\theta=0,\qquad I=\frac{\Gamma a^2}{4\pi}\int_0^{2\pi}(1+h)^4\,d\theta. \]
Consequently,
\[ I-\frac{\Gamma a^2}{2}=\frac{\Gamma a^2}{\pi}\int_0^{2\pi}\left(h+\frac{h^2}{2}\right)^2d\theta. \]
The normalized impulse $I/(\Gamma a^2)=(4\pi i)^{-1}\int F\overline F^2F'\,d\theta$ is an analytic contour functional. In the unforced coordinates, $\|h\|_{L^\infty}\le C\|z\|$, so $I/(\Gamma a^2)-1/2=O(\|z\|^2)$. Its $z$ and $\eta$ derivatives are therefore $O(\|z\|)$ and $O(\|z\|^2)$. The forcing-dependent change is $O(\rho^{3/2})$ in $C^2$ and adds at most $C\rho\rho^{3/2}\le C\rho^2$ to the $\eta$ derivative. The $z$ derivative remains bounded by $C\rho$. Thus
\[ \frac{|\triangle I|}{|\Gamma|a^2}\le C\bigl(\rho\|\triangle z\|_*+\rho^2\|\triangle\eta\|\bigr)=C\rho W_{\mathrm{bdry}}. \]
\end{proof}

\subsection{The centroid and moment theorem}

Consider four patches with circulations $\Gamma_1=\Gamma_2=\Gamma$ and $\Gamma_3=\Gamma_4=-\Gamma$, where
\[ \Gamma=2\pi\Omega\ell^2,\qquad 0<\ell\le1,\qquad a^2=\frac{2\ell^2}{A\Omega^3}. \]
Write $\omega_i$ for the vorticity of patch $i$, $\omega_{\mathrm{quad}}=\sum_{i=1}^4\omega_i$, and $\omega$ for the total vorticity. Let $\zeta_i$ be their centroids, $d=\frac14\sum_i\zeta_i$, and $U=K*(\omega-\omega_{\mathrm{quad}})$ the velocity from all other quadrupoles. Write the positive centroids as $d+\delta\pm z_1$ and the negative centroids as $d-\delta\pm z_2$. Recall $D=|z_1|^2-|z_2|^2$ and define the signed moments
\[ M_1=\int z\omega_{\mathrm{quad}}(z)\,dA(z)=4\Gamma\delta,\qquad I=\int|z-d|^2\omega_{\mathrm{quad}}(z)\,dA(z). \]
\[ M_k=\int(z-d)^k\omega_{\mathrm{quad}}(z)\,dA(z),\qquad k\ge2. \]
The integrals include all four patches with their signs. Since $\sum_{i=1}^4\Gamma_i=0$, $M_1$ is independent of the origin. Put $\rho=C_A\Omega^{-2}$, with $C_A$ large enough for Corollary~\ref{meta:confinement}.

For the comparison of two motions, fix $c,R>0$ and suppose that, for $j=1,2$ and throughout the comparison interval,
\[ \max_{1\le i\le4}\frac{|\zeta_i^{(j)}-d^{(j)}|}{\ell}\le\frac R2,\qquad \min_{i\ne m}\frac{|\zeta_i^{(j)}-\zeta_m^{(j)}|}{\ell}\ge c. \]
For $j=1,2$, assume $U^{(j)}$ is anti-holomorphic on $B(d^{(j)},2R\ell)$ and satisfies $|DU^{(j)}|\le C\mathsf S$ and $|D^2U^{(j)}|\le C\mathsf T$ on $B(d^{(j)},R\ell)$, with $\mathsf S,\mathsf T\le1$. Define
\begin{equation} \label{meta:comparison-quantities}
\begin{gathered}
|Z|=\ell^{-1}\sum_{i=1}^4|\triangle(\zeta_i-d)|,\qquad \|\triangle A\|_R=\sum_{k\ge1}R^k|\triangle A_k|,\\
\ell^{-1}[U^{(j)}(t,d^{(j)}+\ell z)-U^{(j)}(t,d^{(j)})]=\sum_{k\ge1}A_k^{(j)}(t)\bar z^k.
\end{gathered}
\end{equation}
The expansions use the same physical axes and their respective centers. Constants in the relative estimates may depend on $c$ and $R$, but not on the scale or the number of quadrupoles.

\begin{theorem} \label{meta:reduction}
Consider four patches in an Euler evolution of a finite collection of disjoint uniform patches with $C^{1,1/2}$ boundaries on an initial time interval $[0,T]\subset[0,2]$. Suppose these four patches start from disjoint disks with the parameters above and the standard labels $F_i(0,\theta)=e^{i\theta}$. Assume their centroids have mutual distances at least $c\ell$ and their contours satisfy \eqref{meta:boundary-chart}. For each patch $i$, assume \eqref{meta:individual-input} with $w=\zeta_i$ and
\[ u_{\mathrm{ext}}=K*(\omega-\omega_i)=U+\sum_{m\ne i}K*\omega_m. \]
Thus \eqref{meta:individual-input} controls the field induced by the other three patches together with $U$. Suppose $U$ is anti-holomorphic on a neighborhood of the convex hull of $\bigcup_i\overline{B(\zeta_i,2a)}$, where $\sup|DU|\le C\mathsf{S}$ and $\sup|D^2U|\le C\mathsf{T}$ with $\mathsf{S},\mathsf{T}\le1$. Choose $A$ and the lower threshold for $\Omega$ as in Corollary~\ref{meta:confinement}. The constants below may depend on the fixed contour charts, separation and analytic margins, $C_0$, and $A$, but not on the scale or the number of other patches. For the relative assertions, assume also the normalized size, separation, and analyticity bounds stated above. Then the following conclusions hold.

\begin{enumerate}
\item[(i)] For each patch $i$, the label coordinate $\eta_i$ and the normalized contour $F_i$ satisfy
\[ \|\eta_i(t)\|_{C^{1,1/2}}\le\delta/2,\qquad \|F_i(t)-e^{i(\cdot+\eta_i(t))}\|_{C^{1,1/2}}\le\rho<\delta/2,\qquad 0\le t\le T. \]
These bounds improve \eqref{meta:boundary-chart}. The exact centroid equations are
\begin{equation} \label{transfer:centroid-equation}
\dot\zeta_i=\sum_{j\ne i}\Gamma_jK(\zeta_i-\zeta_j)+U(\zeta_i)+E_i,\qquad |E_i|\le C\Omega\ell\frac{a^2}{\ell^2}\rho+Ca^2\rho\mathsf{T}\le C\ell\Omega^{-4}.
\end{equation}
Each patch has zero first centered moment and centered harmonic moments bounded by $C_k\Gamma a^k\rho$ for $k\ge2$.

\item[(ii)] Consider two such evolutions with the same radii and signed circulations and identical initial normalized disks. On an initial time interval in $[0,2]$ satisfying the comparison bounds above, use $s=\Omega t$ and put $F(s)=\Omega^{-1}\|\triangle A(s/\Omega)\|_R$. Define
\begin{equation} \label{transfer:core-memory}
W(s)=C\int_0^s e^{\sigma_c(s-r)}\bigl(|Z(r)|+F(r)\bigr)\,dr,\qquad \sigma_c=C\left(\Omega^{-1}+\frac{a^2}{\ell^2}\right).
\end{equation}
The fixed constant $C$ can be chosen so that, with $|\triangle E|=\sum_i|\triangle E_i|$,
\begin{equation} \label{meta:relative-error}
\frac{|\triangle E|}{\Omega\ell}\le C\frac{a^2}{\ell^2}\bigl(\rho|Z|+W+\rho F\bigr).
\end{equation}
For $k\ge2$, the centered complex moments satisfy $|\triangle M_k|\le C_k\Gamma\ell^k(|Z|+(a^2/\ell^2)W)$. For every later $s_0$,
\begin{equation} \label{meta:history-restart}
W(s)=e^{\sigma_c(s-s_0)}W(s_0)+C\int_{s_0}^s e^{\sigma_c(s-r)}\bigl(|Z(r)|+F(r)\bigr)\,dr.
\end{equation}
In particular $W(0)=0$, and no derivative of $\triangle A$ is required.

\item[(iii)] The first moment and angular impulse satisfy
\begin{equation} \label{transfer:neutral-identities}
\begin{aligned}
\dot M_1&=\int U(z)\omega_{\mathrm{quad}}(z)\,dA(z),\\
\dot I&=2\operatorname{Re}\int\overline{z-d}[U(z)-U(d)]\omega_{\mathrm{quad}}(z)\,dA(z)+2\operatorname{Re}\bigl((U(d)-\dot d)\overline{M_1}\bigr).
\end{aligned}
\end{equation}
Expanding about each patch centroid gives
\[ I-2\Gamma D=\sum_{i=1}^4\int|z-\zeta_i|^2\omega_i(z)\,dA(z). \]
Initially each disk contributes $\Gamma_i a^2/2$, so this sum is zero because $\sum_i\Gamma_i=0$. It satisfies
\begin{equation} \label{transfer:action-correction}
|I-2\Gamma D|\le C\Gamma a^2\rho^2,\qquad |\triangle(I-2\Gamma D)|\le C\Gamma a^2\rho W,
\end{equation}
where the second estimate assumes the comparison hypotheses of (ii). Its time derivative satisfies
\begin{equation} \label{meta:radial-derivative}
\left|\frac d{dt}(I-2\Gamma D)\right|\le C\Gamma a^2\Omega\rho.
\end{equation}
The second complex moment satisfies
\begin{equation} \label{meta:second-moment}
\begin{aligned}
\dot M_2={}&\frac{i}{2\pi}\iint\frac{z-w}{\bar z-\bar w}\,\omega_{\mathrm{quad}}(z)\omega_{\mathrm{quad}}(w)\,dA(z)dA(w)\\
&+2\int(z-d)[U(z)-\dot d]\omega_{\mathrm{quad}}(z)\,dA(z).
\end{aligned}
\end{equation}
Each same-patch integral contributes at most $C\Gamma^2\rho$. Replacing all different-patch integrals by their centroid expressions has total error at most $C\Gamma^2a^2/\ell^2$.

\item[(iv)] The contour bounds cannot cause a first exit while the geometric and external hypotheses retain their margins. To verify \eqref{meta:individual-input} for patch $i$, it suffices that the supports of $\omega-\omega_i$ have distance at least $c\ell$ from $\zeta_i$ and
\[ \int\frac{|(\omega-\omega_i)(y)|}{|\zeta_i-y|^2}\,dA(y)\le C\Omega. \]
For $a/\ell$ sufficiently small, the coefficient and material-derivative bounds then hold with fixed $C_0$, since $|\Gamma|/(c\ell)^2=O(\Omega)$. Local existence and continuous dependence hold for every finite collection before a geometric or contour exit.

If the quadrupole is contained in $B(d,r/4)$, the support of $\omega-\omega_{\mathrm{quad}}$ lies outside $B(d,r)$, and $\int |(\omega-\omega_{\mathrm{quad}})(y)|/|d-y|^2\,dA(y)+\Gamma/r^2\le C\mathsf{S}$, then
\begin{equation} \label{meta:group-strain}
\left|\frac d{dt}DU(t,d(t))\right|\le C\mathsf{S}^2+C\mathsf{T}|\dot d(t)-U(t,d(t))|.
\end{equation}
The separation radius $r$ in this assertion does not enter the choice of $A$.
\end{enumerate}
\end{theorem}

\begin{proof}
(i) Corollary~\ref{meta:confinement} gives the boundary bounds. Each patch differs from its centered disk on a set of area at most $Ca^2\rho$. By the centroid constraint and the vanishing disk moments,
\[ M_{1,\mathrm{patch}}=0,\qquad |M_{k,\mathrm{patch}}|\le C\Gamma a^k(1+C\rho)^k\rho,\qquad k\ge2. \]
For $|z-\zeta_i|\ge c\ell/2$,
\[ \begin{aligned} u_{\mathrm{patch}}(z)&=\frac{i\Gamma_i}{2\pi(\bar z-\bar\zeta_i)}+\frac{i}{2\pi}\sum_{k\ge2}\frac{\overline{M_{k,\mathrm{patch}}}}{(\bar z-\bar\zeta_i)^{k+1}},\\ \left|u_{\mathrm{patch}}(z)-\Gamma_iK(z-\zeta_i)\right|&\le\frac{C\Gamma a^2\rho}{|z-\zeta_i|^3}. \end{aligned} \]
The self-field has zero patch average by antisymmetry of $K$. For the external field, subtract its constant and linear Taylor terms. The centroid constraint and the mean-value property give
\[ \left|\frac1{\Gamma_i}\int u_{\mathrm{ext}}(z)\omega_i(z)\,dA(z)-u_{\mathrm{ext}}(\zeta_i)\right|\le Ca^2\rho\sup|D^2u_{\mathrm{ext}}|\le Ca^2\rho\left(\frac\Gamma{\ell^3}+\mathsf T\right). \]
Combining the two errors proves \eqref{transfer:centroid-equation}.

(ii) The derivative bounds for the external series give
\[ \sum_{k\ge1}k^2R^{2k-2}|A_k|^2\le C\mathsf S^2,\qquad \sum_{k\ge2}k^2(k-1)^2R^{2k-4}|A_k|^2\le C\ell^2\mathsf T^2. \]
By Cauchy--Schwarz,
\[ \sum_{k\ge1}R^k|A_k|\le C_R\mathsf S,\qquad \sum_{k\ge2}R^k|A_k|\le C_R\ell\mathsf T. \]
The contribution of $U$ to the individual coefficient $a_m$ is
\[ \left(\frac a\ell\right)^{m-1}\sum_{k\ge m}\binom{k}{m}A_k\left(\frac{\overline{\zeta_i-d}}\ell\right)^{k-m}. \]
Since $|\zeta_i-d|/\ell\le R/2$ and $2a/\ell\le R/4$, summing with weight $2^m$, also after one normalized centroid derivative, gives a bound $C_R\sum R^k|A_k|$. Its difference is therefore bounded by $C_R(\mathsf S|Z|+\|\triangle A\|_R)$.

Let $W_{\mathrm{bdry}}$ be the sum of the four quantities $\|\triangle z\|_*+\rho\|\triangle\eta\|$, one for each patch in the common contour coordinates of Lemma~\ref{core:relative}. The weight in that lemma is $(L/\Omega_{\mathrm{patch}})^{2/3}=(C_0/A)^{2/3}\Omega^{-2}$. Replacing it by $\rho=C_A\Omega^{-2}$ changes the distance and its differential inequality only by fixed factors. For the other three patches, the point fields contribute $C\Gamma\ell^{-2}|Z|$. Their moment differences start at order two and are bounded by Lemma~\ref{core:relative-moments}. Summing the convergent moment series gives
\[ \begin{aligned} \sum_{k\ge1}2^k|\triangle a_k|&\le C\left(\frac\Gamma{\ell^2}+\mathsf S\right)|Z|+C\frac{\Gamma a^2}{\ell^4}W_{\mathrm{bdry}}+C\|\triangle A\|_R\\ &\le C\Omega\left(|Z|+\frac{a^2}{\ell^2}W_{\mathrm{bdry}}+F\right). \end{aligned} \]
The differential inequality in Lemma~\ref{core:relative}, $\Omega_{\mathrm{patch}}\rho^2=AC_A^2$, and $s=\Omega t$ imply
\[ D_s^+W_{\mathrm{bdry}}\le C\left(\Omega^{-1}+\frac{a^2}{\ell^2}\right)W_{\mathrm{bdry}}+C(|Z|+F),\qquad W_{\mathrm{bdry}}(0)=0. \]
The initial value is zero because both normalized disks and their labels agree. Choosing $C$ in \eqref{transfer:core-memory} sufficiently large, Gronwall gives
\[ W_{\mathrm{bdry}}(s)\le C\int_0^s e^{\sigma_c(s-r)}(|Z(r)|+F(r))\,dr=W(s). \]

The centroid error is
\[ E_i= \sum_{m\ne i}\left[\frac1{\Gamma_i}\int(K*\omega_m)(z)\omega_i(z)\,dA(z)-\Gamma_mK(\zeta_i-\zeta_m)\right] + \frac1{\Gamma_i}\int[U(z)-U(\zeta_i)]\omega_i(z)\,dA(z). \]
Both terms vanish for disks, and their moment expansions start at order two. By (i) and Lemma~\ref{core:relative-moments}, each moment is bounded by $C_k\Gamma a^k\rho$ and its difference by $C_k\Gamma a^kW_{\mathrm{bdry}}$. Thus changes of the kernel coefficients retain $\rho$, while changes of the moments give $W_{\mathrm{bdry}}$. For the external term, Cauchy's estimates give
\[ \sup|D^2U|\le C_R\mathsf S/\ell,\qquad \sup|D^3U|\le C_R\mathsf S/\ell^2. \]
The centered external-field difference has Hessian bounded by $C_R\|\triangle A\|_R/\ell$. Combining the internal and external bounds gives
\[ \begin{aligned} |\triangle E|\le{}&C\frac{\Gamma a^2}{\ell^3}(\rho|Z|+W_{\mathrm{bdry}})\\ &+C_R\frac{a^2}{\ell}\left[\mathsf S(\rho|Z|+W_{\mathrm{bdry}})+\rho\|\triangle A\|_R\right]. \end{aligned} \]
For nearby centroid configurations, apply the derivative bounds along their connecting segment. Otherwise $|Z|$ is bounded below, and the absolute estimates give the same bound. Since $\Gamma=2\pi\Omega\ell^2$ and $\mathsf S\le\Omega$,
\[ \frac{|\triangle E|}{\Omega\ell}\le C\frac{a^2}{\ell^2}(\rho|Z|+W_{\mathrm{bdry}}+\rho F)\le C\frac{a^2}{\ell^2}(\rho|Z|+W+\rho F). \]
This proves \eqref{meta:relative-error}.

Expanding $(z-d)^k$ about each patch centroid and using the same moment bounds gives
\[ |\triangle M_k|\le C_k\Gamma\ell^k|Z|+C_k\Gamma a^2\ell^{k-2}(\rho|Z|+W)\le C_k\Gamma\ell^k\left(|Z|+\frac{a^2}{\ell^2}W\right). \]
Splitting the integral defining $W$ at $s_0$ proves \eqref{meta:history-restart}.

(iii) Differentiate $M_1$ and $I$. The internal terms cancel by
\[ K(z-w)+K(w-z)=0,\qquad \operatorname{Re}\bigl(\overline{z-w}K(z-w)\bigr)=0. \]
Differentiating the moving center gives the last term in \eqref{transfer:neutral-identities}. The centroid contribution to $I$ is $2\Gamma D$, and the four signed disk impulses cancel. Lemma~\ref{core:relative-moments} and $W_{\mathrm{bdry}}\le W$ therefore give \eqref{transfer:action-correction}.

For each patch, the centroid term and self-field term again vanish. The centered external series gives
\[ \frac d{dt}\int|z-\zeta_i|^2\omega_i(z)\,dA(z)=2\operatorname{Re}\sum_{k\ge1}a_k a^{1-k}\overline{M_{k+1,\mathrm{patch}}}. \]
Since $(1+C\rho)^k<2^k$, the moment bound in (i) yields
\[ \left|\frac d{dt}\int|z-\zeta_i|^2\omega_i(z)\,dA(z)\right|\le C\Gamma a^2\rho\sum_{k\ge1}2^k|a_k|\le C\Gamma a^2\Omega\rho. \]
Summing proves \eqref{meta:radial-derivative}.

For $M_2$, symmetrization gives
\[ \begin{aligned} &2\iint(z-d)K(z-w)\omega_{\mathrm{quad}}(z)\omega_{\mathrm{quad}}(w)\,dA(z)dA(w)\\ &\qquad=\frac{i}{2\pi}\iint\frac{z-w}{\bar z-\bar w}\omega_{\mathrm{quad}}(z)\omega_{\mathrm{quad}}(w)\,dA(z)dA(w), \end{aligned} \]
which proves \eqref{meta:second-moment}. The quotient has modulus one off the diagonal and zero integral over two copies of a disk, by rotation. Replacing one patch by its disk changes the vorticity measure by at most $C\Gamma\rho$ in total variation, so each same-patch integral is $O(\Gamma^2\rho)$. For distinct patches, the quotient has second derivatives bounded by $C\ell^{-2}$. Taylor expansion and the zero first centered moments give error $C\Gamma^2a^2/\ell^2$.

(iv) Lemma~\ref{core:material-derivative} gives \eqref{meta:individual-input} under the stated geometric criterion, with $C_0$ independent of the number and vorticity magnitudes of the other patches. For \eqref{meta:group-strain}, its symmetrization bounds the external sources' contribution to $(\partial_t+U(d)\cdot\nabla)DU(d)$ by $C\mathsf S^2$. The four patches induce velocity at most $C\Gamma/|d-y|$ on the external supports, and hence contribute at most
\[ C\Gamma\int\frac{|(\omega-\omega_{\mathrm{quad}})(y)|}{|d-y|^4}\,dA(y)\le C\frac\Gamma{r^2}\int\frac{|(\omega-\omega_{\mathrm{quad}})(y)|}{|d-y|^2}\,dA(y)\le C\mathsf S^2. \]
Replacing $U(d)$ by $\dot d$ adds $D^2U(d)[\dot d-U(d)]$, bounded by $C\mathsf T|\dot d-U(d)|$. This proves \eqref{meta:group-strain}.

Local existence and continuous dependence follow from Lemma~\ref{core:analytic-contour}. Part (i) strictly improves the contour bounds, so continuation fails only if a geometric or external bound loses its margin.
\end{proof}

We record how the estimate is used when the centroid coordinates or comparison region change. Near a rotating rhombus, $(p,q)$, $D/\ell^2$, $\theta$, and $\delta/\ell$ are local coordinates for the centered centroids, so their difference norm is equivalent to $|Z|$. Replacing $D/\ell^2$ by $I/(2\Gamma\ell^2)$ introduces an error at most $C(a^2/\ell^2)\rho W$ by \eqref{transfer:action-correction}. Gronwall then increases $\sigma_c$ by at most $C(a^2/\ell^2)\rho$. Fix these constants in the rhombus neighborhood before fixing how far to follow the separating point-vortex trajectory.

On a subsequent fixed region of bounded, separated centroid configurations, continue $W$ by $W_s=\sigma_cW+C(|Z|+F)$, using the incoming value of $W$ and the constants required in that region. Formula~\eqref{meta:history-restart} holds on each interval, and Lemma~\ref{core:relative} preserves $W_{\mathrm{bdry}}\le W$. Across successive intervals the exponential factor is $\exp(\int_r^s\sigma_c(v)\,dv)$, with the input constant chosen on each interval. A fixed normalized duration gives a fixed amplification, without changing the constants on earlier intervals.

Once the normalized diameter grows like $1+\Omega t$, use the absolute conclusions (i), (iii), and (iv), while their separation and external-field assumptions hold. They apply on the original interval starting from disks at time zero, including all later subintervals. No relative comparison with a rhombus is required there.

The following consequence bounds the accumulated patch error when the centroid and external-field differences grow at the unstable rate $2\sqrt2$.

\begin{corollary} \label{meta:memory}
Use the notation and fixed constants of Theorem~\ref{meta:reduction}(ii), with $\Omega$ sufficiently large that $\sigma_c<\sqrt2$. Let $0<\kappa<1$. On a normalized interval $[s_0,s_1]$, suppose
\[ W(s_0)\le C\kappa,\qquad |Z(s)|+F(s)\le C\kappa e^{2\sqrt2(s-s_0)}. \]
Then $W(s)\le C\kappa e^{2\sqrt2(s-s_0)}$ and
\[ \frac1\kappa\int_{s_0}^{s_1}e^{-2\sqrt2(s-s_0)}\frac{|\triangle E(s)|}{\Omega\ell}\,ds\le C\frac{a^2}{\ell^2}(s_1-s_0). \]
For an interval of length $O(\log(1/\kappa))$, this is $O(\Omega^{-3}\log(1/\kappa))$. The normalized angular-impulse correction satisfies
\[ \frac1\kappa\int_{s_0}^{s_1}e^{-2\sqrt2(s-s_0)}\frac{|\triangle(I-2\Gamma D)(s)|}{2\Gamma\ell^2}\,ds\le C\frac{a^2}{\ell^2}\rho(s_1-s_0). \]
\end{corollary}

\begin{proof}
Substitute the assumed bound into \eqref{meta:history-restart} and use
\[ \int_{s_0}^s e^{\sigma_c(s-r)}\kappa e^{2\sqrt2(r-s_0)}\,dr=\frac{\kappa\bigl(e^{2\sqrt2(s-s_0)}-e^{\sigma_c(s-s_0)}\bigr)}{2\sqrt2-\sigma_c}\le C\kappa e^{2\sqrt2(s-s_0)}. \]
The initial term $e^{\sigma_c(s-s_0)}W(s_0)$ has the same bound. Equations~\eqref{meta:relative-error} and \eqref{transfer:action-correction} give the two weighted integrals. Finally, $a^2/\ell^2=2/(A\Omega^3)$ gives the stated accumulated error.
\end{proof}

\section{First solution: rotating quadrupoles at all scales} \label{sec:first-solution}

We choose the initial unstable shape variables in the exact centroid equations of Theorem~\ref{meta:reduction} so that every quadrupole stays close to a rotating rhombus about its moving center for $0\le t\le2$. The estimates on $[0,1]$ will be sharper because the unstable variables remain bounded through time 2.

Use the scales \eqref{eq:scales} and the signed circular patches of Section~\ref{intro:scales}, with $B>2$. First fix $\delta>0$ sufficiently small for the contour and shape coordinates and the estimates below. Next choose $A$ large enough for Theorem~\ref{meta:reduction} on $[0,2]$, and finally choose $\Lambda$ large. Constants may depend on $A,\alpha,B$, but are independent of $n,N$ and the prescribed phases. We distinguish the constant $C_0$ in the quadratic point-vortex remainder, which is independent of $A$, from constants $C_A$ that may depend on $A$.

For quadrupole $n$, use the centroid coordinates $d_n,\delta_n,z_{1,n},z_{2,n},D_n$ of Theorem~\ref{meta:reduction} and the shape variables $(p_n,q_n)$ of Section~\ref{sec:shape-variables}. Let $\theta_n=\arg z_{1,n}$ be continuous, with $\theta_n(0)=\varphi_n$. The displacement $\delta_n$ is complex-valued, whereas $\delta>0$ is the fixed small radius chosen above.

Write the unstable shape equation as $\dot q_n=2\sqrt2\Omega_nq_n+f_n$, where $f_n$ includes the nonlinear terms and the forcing from the other patches. Under the bounds below, whenever $|q_n|$ reaches a fixed threshold $0<\delta\ll1$, the estimate $|f_n|\lesssim\Omega_n\delta^2+C$ gives $|f_n|<2\sqrt2\Omega_n\delta$ for sufficiently small $\delta$ and large $\Lambda$. We obtain the initial vector by varying $(q_1(0),\ldots,q_N(0))$ and applying Brouwer's theorem to the first exit through $q_n=\pm\delta$. This gives an initial choice of $(q_1(0), \cdots, q_N(0))$ that satisfies $|q_n(t)|\le\delta$ through time 2 without a contraction argument. The selected solution need not have $q_n(2)=0$, but its terminal contribution is exponentially small on $[0,1]$.

Prescribe $d_n(0)=d_n$, $\delta_n(0)=0$, $D_n(0)=\ell_n^2$, $p_n(0)=0$, and $\theta_n(0)=\varphi_n$, and vary $q_n(0)$ in $[-\delta,\delta]$. The corresponding initial positions relative to $d_n$ are
\[ z_{1,n}(0)=\frac{\ell_ne^{i\varphi_n}}{\sqrt{1-R_*(1+q_n(0))}},\qquad z_{2,n}(0)=\frac{\ell_n\sqrt{R_*(1+q_n(0))}}{\sqrt{1-R_*(1+q_n(0))}}e^{i(\varphi_n+\pi/2+q_n(0)/\sqrt2)}. \]
Set
\[ \begin{aligned} \omega_N^{(1)}(0,x)=A\sum_{n=1}^N\Omega_n^4\Bigl(&\mathbf1_{B(d_n+z_{1,n}(0),a_n)}(x)+\mathbf1_{B(d_n-z_{1,n}(0),a_n)}(x)\\ &-\mathbf1_{B(d_n+z_{2,n}(0),a_n)}(x)-\mathbf1_{B(d_n-z_{2,n}(0),a_n)}(x)\Bigr). \end{aligned} \]
The four centroid separations at scale $n$ are comparable to $\ell_n$. Since $a_n/\ell_n\to0$ and $\ell_n/|d_n-d_j|\to0$ uniformly for $j\ne n$, these disks are disjoint for sufficiently large $\Lambda$, uniformly over the initial cube.

Give each initial disk the normalized boundary labels $F(0,\theta)=e^{i\theta}$ and transport these labels with its Euler flow, as in Theorem~\ref{meta:reduction}.

Let $I_n$ be the angular impulse in Theorem~\ref{meta:reduction}(iii). The difference $I_n-2\Gamma_nD_n$ is the patch impulse. It vanishes for the initial equal disks, so $I_n(0)=2\Gamma_n\ell_n^2$.

The external velocity of quadrupole $n$ is the velocity induced by all other quadrupoles. Near this quadrupole it is divergence-free and curl-free, so its gradient is symmetric and trace-free. Set
\[ \mathsf{S}=\sum_{j\ge1}\frac{\Gamma_j}{d_j^2} \le 8\Gamma_1,\qquad \mathsf{T}=\sum_{j\ge1}\frac{\Gamma_j}{d_j^3} \le 16\Gamma_1. \]
Under the separation bounds below, the gradients and Hessians of the external velocities are bounded by $C\mathsf{S}$ and $C\mathsf{T}$, respectively. Both sums tend to zero as $\Lambda\to\infty$ and give bounds independent of $N$.

\begin{proposition} \label{prep:reference}
Fix $B>2$, a sufficiently small $\delta>0$, and a sufficiently large $A>0$. There exists $\Lambda_0$, independent of $N$, such that for every $\Lambda\ge\Lambda_0$, every $N\ge1$, and every $(\varphi_1,\ldots,\varphi_N) \in (\mathbb{R}/2\pi\mathbb{Z})^N$, some choice $(q_1(0),\ldots,q_N(0))\in[-\delta,\delta]^N$ gives an exact Euler patch solution $\omega_N^{(1)}$ on $[0,2]$, with $u_N^{(1)}=K*\omega_N^{(1)}$ and the initial disks prescribed above. In particular, $p_n(0)=0$, $\delta_n(0)=0$, $D_n(0)=\ell_n^2$, and $\theta_n(0)=\varphi_n$. The following constants are independent of $N,n$ and the phases.

\begin{enumerate}
\item[(i)] For every $1\le n\le N$ and $0\le t\le2$,
\[ |q_n(t)|,|p_n(t)|\le\delta,\qquad \frac12\le\frac{D_n(t)}{\ell_n^2}\le2,\qquad |\delta_n(t)|\le C\mathsf{T}\ell_n^2t e^{C\mathsf{S}t}. \]
Each normalized patch boundary has $C^{1,1/2}$ distance at most $C\Omega_n^{-2}$ from a parameterized circle. For $i\ne j$,
\[ \frac12|d_i-d_j|\le |d_i(t)-d_j(t)|\le2|d_i-d_j|,\qquad 0\le t\le2. \]
Thus all patches remain disjoint, and every quadrupole stays near a rhombus about its moving center.

\item[(ii)] For every $1\le n\le N$ and $0<T\le1$,
\[ \sup_{0\le t\le T}\left(|q_n(t)|+|p_n(t)|+\left|\frac{I_n(t)}{2\Gamma_n\ell_n^2}-1\right|+\left|\frac{D_n(t)}{\ell_n^2}-1\right|\right)\le C\left(\frac{\mathsf{S}+T\mathsf{S}^2}{\Omega_n}+\Omega_n^{-5}\right). \]
The phase satisfies
\begin{equation} \label{prep:phase}
\sup_{0\le t\le T}|\theta_n(t)-\varphi_n+\Omega_nt|\le CT\bigl(\mathsf{S}+T\mathsf{S}^2+\Omega_n^{-4}\bigr),\qquad 0<T\le1.
\end{equation}
In particular, $|q_n(0)|\le C(\mathsf{S}\Omega_n^{-1}+\Omega_n^{-5})$, and the phase error on $[0,1]$ tends to zero uniformly over $n\le N$ as $\Lambda\to\infty$.

\item[(iii)] The outer quadrupole satisfies
\[ \sup_{0\le t\le2}\frac{|d_1(t)-d_1|}{\ell_1}\le C\left(\frac{\sum_{j\ge2}\Gamma_j}{\ell_1}+\mathsf{T}\Omega_1\ell_1+\mathsf{T}\ell_1+\Omega_1^{-4}\right)\longrightarrow0\qquad\text{as }\Lambda\to\infty. \]
\end{enumerate}
Parts (ii) and (iii) hold for every initial choice whose solution satisfies part (i) through time 2.
\end{proposition}

Theorem~\ref{meta:reduction}(iv) gives local existence and continuous dependence for each initial choice. On an initial time interval $[0,T]\subset[0,2]$, impose the following \emph{bootstrap assumptions} for every $1\le n\le N$ and $i\ne j$,
\[ |q_n(t)|,|p_n(t)|\le\delta,\qquad \tfrac12\le D_n(t)/\ell_n^2\le2,\qquad |\delta_n(t)|\le\delta\ell_n,\qquad |d_i(t)-d_j(t)|\ge\tfrac12|d_i-d_j|, \]
and require every contour boundary to satisfy \eqref{meta:boundary-chart}. All these conditions hold for the initial disks. The next two lemmas improve \eqref{meta:boundary-chart} and the bounds on $\delta_n,D_n$ and $|d_i(t)-d_j(t)|$, and derive the shape equations.

\begin{lemma} \label{prep:patch-bounds}
On every initial time interval in $[0,2]$ satisfying the bootstrap assumptions, the geometric hypotheses of Theorem~\ref{meta:reduction} hold. Each centroid satisfies \eqref{transfer:centroid-equation} with $|E_i|\le C\ell_n\Omega_n^{-4}$. Every normalized patch boundary stays within $C\Omega_n^{-2}$ of a parameterized circle. The constants are uniform in $n,N$, the phases, and the initial vector in $[-\delta,\delta]^N$.
\end{lemma}

\begin{proof}
Up to the first possible exit from \eqref{meta:boundary-chart}, each patch lies in $B(\zeta_i,2a_n)$. For sufficiently small $\delta$, the bootstrap assumptions give centroid separations at least $c\ell_n$ within quadrupole $n$ and at least $c|d_n-d_j|$ between quadrupoles $n$ and $j$. Since $a_n/\ell_n=\sqrt{2/A}\,\Omega_n^{-3/2}$ is uniformly small, the supports remain separated.

For a centroid $w$ at scale $n$, let $\omega_{\mathrm{other}}$ denote all vorticity except its own patch. Since $|d_n-d_j|\ge d_j/2$ for $j\ne n$,
\[ \int\frac{|\omega_{\mathrm{other}}(y)|}{|w-y|^2}\,dA(y)\le C\frac{\Gamma_n}{\ell_n^2}+C\sum_{j\ne n}\frac{\Gamma_j}{|d_n-d_j|^2}\le C(\Omega_n+\mathsf{S})\le C\Omega_n. \]
By Theorem~\ref{meta:reduction}(iv) and Lemma~\ref{core:material-derivative}, its normalized external coefficients satisfy \eqref{meta:individual-input},
\[ \sum_{k\ge1}2^k|a_k(t)|\le C\Omega_n,\qquad \sum_{k\ge1}2^k|\dot a_k(t)|\le C\Omega_n^2. \]
The derivative follows $w(t)$ in the fixed physical axes.

The external field $U$ of the quadrupole is analytic near the convex hull of the four disks $B(\zeta_i,2a_n)$, which is separated from the other supports. The spatial kernel bounds give, throughout this neighborhood,
\[ \sup|DU|\le C\sum_{j\ne n}\frac{\Gamma_j}{|d_n-d_j|^2}\le C\mathsf{S},\qquad \sup|D^2U|\le C\sum_{j\ne n}\frac{\Gamma_j}{|d_n-d_j|^3}\le C\mathsf{T}. \]
Thus Theorem~\ref{meta:reduction} applies on the stopped interval. Its boundary bounds are
\[ \|\eta(t)\|_{C^{1,1/2}}\le\delta/2,\qquad \|F(t)-e^{i(\cdot+\eta(t))}\|_{C^{1,1/2}}\le C\Omega_n^{-2}<\delta/2. \]
Moreover, \eqref{transfer:centroid-equation} gives
\[ |E_i|\le C\Omega_n\ell_n\frac{a_n^2}{\ell_n^2}\Omega_n^{-2}+Ca_n^2\Omega_n^{-2}\mathsf{T}\le C\ell_n\Omega_n^{-4}, \]
using $a_n^2/\ell_n^2=2/(A\Omega_n^3)$ and $\mathsf{T}\le1$. The strict contour bounds and Theorem~\ref{meta:reduction}(iv) exclude a first contour exit.
\end{proof}

\begin{lemma} \label{prep:moment-bounds}
On every initial time interval in $[0,2]$ satisfying the bootstrap assumptions,
\[ |\delta_n(t)|\le C\mathsf{T}\ell_n^2t e^{C\mathsf{S}t},\qquad \left|\frac{D_n(t)}{\ell_n^2}-1\right|\le C(\mathsf{S}+\Omega_n^{-7}),\qquad \left|\frac{\dot D_n(t)}{D_n(t)}\right|\le C(\mathsf{S}+\Omega_n^{-4}), \]
\[ \frac34|d_i-d_j|\le |d_i(t)-d_j(t)|\le\frac32|d_i-d_j|\qquad(i\ne j), \]
after increasing $\Lambda$ if necessary. The shape equations are
\[ \dot q_n=2\sqrt2\Omega_n\frac{\ell_n^2}{D_n}q_n+f_n,\qquad \dot p_n=-2\sqrt2\Omega_n\frac{\ell_n^2}{D_n}p_n+\widetilde f_n, \]
\[ |f_n|+|\widetilde f_n|\le C_0\Omega_n(|q_n|+|p_n|)^2+C_A\bigl(\mathsf{S}+\Omega_n^{-4}+\Omega_n\mathsf{T}^2\ell_n^2\bigr). \]
The remainders $f_n,\widetilde f_n$ include the nonlinear shape terms and the forcing by the other patches, and depend on the entire finite solution.
\end{lemma}

\begin{proof}
Write $U_{{\rm ext},n}$ for the exact velocity of the other quadrupoles and expand
\[ U_{{\rm ext},n}(d_n(t)+\zeta)=U_{{\rm ext},n}(d_n(t))+\sum_{k\ge1}a_{k,n}\overline\zeta^{\,k}. \]
The series converges on a disk containing quadrupole $n$. Its coefficients satisfy $A_{k,n}=\ell_n^{k-1}a_{k,n}$ in the notation of Theorem~\ref{meta:reduction}. In particular, the strain coefficient is $S_n=a_{1,n}$. The coefficients for an individual patch are normalized by $a_n$ instead.

The internal velocity preserves the signed first moment $M_{1,n}=4\Gamma_n\delta_n$. Since the total circulation is zero, \eqref{transfer:neutral-identities} gives
\begin{equation} \label{prep:first-moment}
\dot M_{1,n}=\sum_{k\ge1}a_{k,n}\overline{M_{k,n}}.
\end{equation}
The supports lie within $C\ell_n$ of $d_n(t)$, so $|M_{k,n}|\le C_k\Gamma_n\ell_n^k$ for $k\ge2$. The sum over $k\ge2$ is the integral of the external Taylor remainder after its affine part is subtracted, and is bounded by $C\Gamma_n\mathsf{T}\ell_n^2$. Dividing by $4\Gamma_n$ gives
\[ \dot\delta_n=a_{1,n}\overline{\delta_n}+\frac1{4\Gamma_n}\sum_{k\ge2}a_{k,n}\overline{M_{k,n}},\qquad \left|\dot\delta_n-a_{1,n}\overline{\delta_n}\right|\le C\mathsf{T}\ell_n^2,\qquad |a_{1,n}|\le C\mathsf{S}. \]
Since $\delta_n(0)=0$, Gronwall's inequality yields
\[ |\delta_n(t)|\le C\mathsf{T}\ell_n^2\int_0^t e^{C\mathsf{S}(t-s)}\,ds\le C\mathsf{T}\ell_n^2t e^{C\mathsf{S}t}. \]
This strictly improves $|\delta_n|\le\delta\ell_n$ for large $\Lambda$.

The internal point velocity of the arithmetic center is odd in $\delta_n$ and vanishes at $\delta_n=0$ by central symmetry. Differentiating the rational point-vortex velocities on the fixed noncollision neighborhood bounds it by $C\Omega_n|\delta_n|$. The four external linear Taylor terms cancel because the centroids average to $d_n(t)$. Together with Lemma~\ref{prep:patch-bounds}, this gives
\begin{equation} \label{prep:center-residual}
|\dot d_n(t)-U_{{\rm ext},n}(d_n(t))|\le r_n,\qquad r_n=C\bigl(\mathsf{T}\Omega_n\ell_n^2+\mathsf{T}\ell_n^2+\ell_n\Omega_n^{-4}\bigr).
\end{equation}
The three terms bound the internal point velocity, the external Taylor remainder, and the centroid error, respectively.

To compare the center trajectories, multiply each quadrupole's velocity by a smooth cutoff which is zero on $B(d_n(t),\epsilon_0d_n)$ and one outside $B(d_n(t),2\epsilon_0d_n)$. Choose the fixed $\epsilon_0>0$ so that the larger disks are disjoint under the bound $|d_i(t)-d_j(t)|\ge\frac12|d_i-d_j|$. The smaller disks contain their quadrupoles for large $\Lambda$. The sum $V$ of these fields satisfies
\[ V(d_n(t))=U_{{\rm ext},n}(d_n(t)),\qquad \|\nabla V\|_{L^\infty}\le C\mathsf{S}. \]
Indeed, on the cutoff annulus the velocity is bounded by $C\Gamma_n/d_n$, while the cutoff derivative is bounded by $C/d_n$. The derivative of the velocity is bounded by $C\Gamma_n/d_n^2$ wherever its cutoff is nonzero. Summing gives the stated bound. Hence
\[ |\dot d_i(t)-\dot d_j(t)|\le C\mathsf{S}|d_i(t)-d_j(t)|+r_i+r_j. \]
Since $d_n=2^{-n}$ gives $d_i+d_j\le3|d_i-d_j|$, we have $r_i+r_j\le3|d_i-d_j|\sup_n(r_n/d_n)$. The upper and lower differential inequalities for the distance imply
\[ e^{-C\mathsf{S}t}-3t\sup_n(r_n/d_n)\le\frac{|d_i(t)-d_j(t)|}{|d_i-d_j|}\le e^{C\mathsf{S}t}\bigl(1+3t\sup_n(r_n/d_n)\bigr). \]
Since $d_n^{-1}=2^n\le\Omega_n$ for $\Lambda\ge1$ and $B>2$,
\[ \frac{r_n}{d_n}\le C\bigl(\mathsf{T}\Omega_n^2\ell_n^2+\mathsf{T}\Omega_n\ell_n^2+\ell_n\Omega_n^{-3}\bigr),\qquad \mathsf{S}+\sup_n\frac{r_n}{d_n}\longrightarrow0. \]
Thus the distance ratio tends uniformly to one for $0\le t\le2$, proving the stated separation bounds.

The centroid contribution to the signed angular impulse is $2\Gamma_nD_n$, since the terms in $|\delta_n|^2$ cancel. The equal reference-disk radial moments also cancel. Theorem~\ref{meta:reduction}(iii) therefore gives
\begin{equation} \label{prep:angular-moment}
|I_n-2\Gamma_nD_n|\le C\Gamma_n\ell_n^2\Omega_n^{-7},
\end{equation}
using $C\Gamma_na_n^2\Omega_n^{-4}\le C\Gamma_n\ell_n^2\Omega_n^{-7}$. Expanding \eqref{transfer:neutral-identities} yields
\begin{equation} \label{prep:angular-derivative}
\dot I_n=2\operatorname{Re}\Big((U_{{\rm ext},n}(d_n(t))-\dot d_n(t))\overline{M_{1,n}}+\sum_{k\ge1}a_{k,n}\overline{M_{k+1,n}}\Big).
\end{equation}
The internal contribution vanishes by antisymmetry of the kernel, and differentiation of the center contributes $-2\operatorname{Re}(\dot d_n\overline{M_{1,n}})$. The moment bounds and \eqref{prep:center-residual} imply
\[ |\dot I_n|\le C\Gamma_n\ell_n^2\bigl(\mathsf{S}+\mathsf{T}\ell_n+\Omega_n\mathsf{T}^2\ell_n^2+\mathsf{T}\ell_n\Omega_n^{-4}\bigr). \]
Since $I_n(0)=2\Gamma_n\ell_n^2$, integration and \eqref{prep:angular-moment} give
\[ \left|\frac{D_n(t)}{\ell_n^2}-1\right|\le C\bigl(\mathsf{S}+\mathsf{T}\ell_n+\Omega_n\mathsf{T}^2\ell_n^2+\mathsf{T}\ell_n\Omega_n^{-4}+\Omega_n^{-7}\bigr),\qquad 0\le t\le2. \]
For its derivative, use the patch impulse derivative bound in Theorem~\ref{meta:reduction}(iii) to obtain
\[ |\dot I_n-2\Gamma_n\dot D_n|\le C\Gamma_na_n^2\Omega_n\Omega_n^{-2},\qquad \left|\frac{\dot D_n}{D_n}\right|\le C(\mathsf{S}+\Omega_n^{-4})+C\mathsf{T}\ell_n+C\Omega_n\mathsf{T}^2\ell_n^2. \]
Here internal patch rotation cancels, leaving the individual-patch external strain, bounded by $C\Omega_n$. Since $\ell_n=e^{-\Omega_n^\alpha}$ and $\mathsf{T}\le1$, the additional terms in the bounds for $D_n$ and $\dot D_n$ are at most $C\Omega_n^{-7}$ and $C\Omega_n^{-4}$, respectively. This proves both stated bounds and strictly improves the bootstrap bound on $D_n$.

Finally, \eqref{eq:shape-linearization} at scale $D_n$ gives
\[ \binom{\dot x_n}{\dot y_n}=2\sqrt2\Omega_n\frac{\ell_n^2}{D_n}\begin{pmatrix}0&1\\1&0\end{pmatrix}\binom{x_n}{y_n}+\text{remainder}. \]
Taking the sum and difference yields the equations for $q_n,p_n$, with
\[ |f_n|+|\widetilde f_n|\le C_0\Omega_n(|q_n|+|p_n|)^2+C_A\left[\mathsf{S}+\Omega_n^{-4}+\Omega_n\left|\frac{\delta_n}{\ell_n}\right|^2\right]. \]
The first term is the Taylor remainder of the point equations. Its constant $C_0$ depends only on the fixed centroid neighborhood and is independent of $A$. The second is the external strain, since common translation leaves the shape unchanged. The third follows by dividing the centroid error $C\ell_n\Omega_n^{-4}$ by $\ell_n$. For the last term, reflection through the center followed by interchange of the two members of each equal-sign pair sends $\delta_n$ to $-\delta_n$ and leaves $z_{1,n},z_{2,n}$ unchanged. Thus the internal shape equations are even in $\delta_n$, and their correction is quadratic. Substituting the bound on $\delta_n$ proves the result.
\end{proof}

\begin{lemma} \label{prep:shooting}
For every finite $N$ and prescribed phase vector $(\varphi_1, \cdots, \varphi_N) \in (\mathbb{R}/2\pi\mathbb{Z})^N$, some $(q_1(0),\ldots,q_N(0))\in[-\delta,\delta]^N$ gives a solution satisfying all bootstrap assumptions through time 2. Every such solution satisfies
\begin{equation} \label{prep:interior-tube}
|q_n(t)|+|p_n(t)|\le C\left(\frac{\mathsf{S}}{\Omega_n}+\Omega_n^{-5}\right)+C\delta e^{-c\Omega_n(2-t)},\qquad 0\le t\le2,
\end{equation}
where $c,C>0$ are independent of $n,N$. In particular,
\[ \sup_{0\le t\le1}(|q_n(t)|+|p_n(t)|)\le C(\mathsf{S}/\Omega_n+\Omega_n^{-5}). \]
\end{lemma}

\begin{proof}
On the bootstrap neighborhood,
\[ \sqrt2\Omega_n\le2\sqrt2\Omega_n\frac{\ell_n^2}{D_n(t)}\le4\sqrt2\Omega_n. \]
By Lemma~\ref{prep:moment-bounds} and $\Omega_n\mathsf{T}^2\ell_n^2\le\Omega_n^{-4}$,
\[ |f_n|+|\widetilde f_n|\le C_0\Omega_n(|q_n|+|p_n|)^2+C_A(\mathsf{S}+\Omega_n^{-4}). \]
Choose $\delta$ so that $4C_0\delta\le\sqrt2/4$ and so that the kernel estimate below can be absorbed. This choice depends only on $C_0$ and universal kernel bounds. Next fix $A$ as above, and then choose $\Lambda$ so that
\[ C_A\sup_{n\ge1}(\mathsf{S}/\Omega_n+\Omega_n^{-5})\le\tfrac{\sqrt2}{4}\delta. \]
The remainders are at most half the linear term on each shape face. Hence
\[ \begin{aligned} q_n=\delta&\ \Longrightarrow\ \dot q_n\ge\tfrac{\sqrt2}{2}\Omega_n\delta,&q_n=-\delta&\ \Longrightarrow\ \dot q_n\le-\tfrac{\sqrt2}{2}\Omega_n\delta,\\ p_n=\delta&\ \Longrightarrow\ \dot p_n\le-\tfrac{\sqrt2}{2}\Omega_n\delta,&p_n=-\delta&\ \Longrightarrow\ \dot p_n\ge\tfrac{\sqrt2}{2}\Omega_n\delta. \end{aligned} \]
Both faces $p_n=\pm\delta$ point inward. Lemma~\ref{prep:patch-bounds} strictly improves \eqref{meta:boundary-chart}, and Lemma~\ref{prep:moment-bounds} strictly improves the bounds on $\delta_n,D_n$ and $|d_i(t)-d_j(t)|$. Thus only $q_n=\pm\delta$ can be a first exit. The contours are regular and separated, and these faces lie inside the coordinate neighborhood, so Theorem~\ref{meta:reduction}(iv) gives continuation slightly beyond them.

Let $\mathbf{q}=(\mathbf{q}_1,\ldots,\mathbf{q}_N)\in[-\delta,\delta]^N$ prescribe $q_n(0)=\mathbf{q}_n$, with the other initial data fixed. Suppose no initial vector satisfies all bootstrap bounds through time 2. The preceding continuation and strict improvements imply that every trajectory first reaches $|q_n|=\delta$ for some $n$ before time 2. For interior $\mathbf{q}$, let $\tau(\mathbf{q})<2$ be this first contact time. On the boundary of the cube set $\tau(\mathbf{q})=0$.

For fixed $N$, the self-induced contour field is locally Lipschitz in $C^{1,1/2}$ by Lemma~\ref{core:analytic-contour}. Interactions between distinct patches have smooth kernels because their supports remain separated. The coupled contour equation is therefore a locally Lipschitz ordinary differential equation, and the initial disk centers depend smoothly on $\mathbf q$. Thus, near each fixed initial vector $\mathbf q$, solutions exist through $\tau(\mathbf q)$ and for a short time beyond it, and depend continuously on the initial vector in the contour norm, uniformly on this common time interval. Only continuity for each fixed $N$ is needed here. This local continuation time and modulus of continuity need not be uniform in $N$.

We claim that $\tau$ is continuous. On any nonempty interval $[0,\tau(\mathbf{q})-\varepsilon]$, all $|q_n|$ have a positive margin below $\delta$. Continuous dependence therefore prevents nearby trajectories from exiting on this interval. At $\tau(\mathbf{q})$, some coordinate $q_n$ has a strictly outward derivative, so nearby trajectories also cross before time $\tau(\mathbf{q})+\varepsilon$. These give the lower and upper bounds for continuity. If several faces are reached simultaneously, the lower bound uses all coordinates and the upper bound uses any outward coordinate. At an initial boundary point, the same upper bound gives $\tau\to0$.

Consequently,
\[ r(\mathbf{q})=\bigl(q_1(\tau(\mathbf{q});\mathbf{q}),\ldots,q_N(\tau(\mathbf{q});\mathbf{q})\bigr) \]
is a continuous map from the cube to its boundary, equal to the identity on the boundary. By Brouwer's theorem, the cube self-map $\mathbf{q}\mapsto-r(\mathbf{q})$ has a fixed point. This point is on the boundary, where $r(\mathbf{q})=\mathbf{q}$, so it must satisfy $\mathbf{q}=-\mathbf{q}=0$, a contradiction. Thus some initial vector satisfies the bootstrap bounds through time 2, with first contact at time 2 allowed. The selected initial vector need not be unique or depend continuously on the phases.

Along any such solution, set
\[ \lambda_n(t)=2\sqrt2\Omega_n\frac{\ell_n^2}{D_n(t)}. \]
Variation of constants, forward from $p_n(0)=0$ and backward from $q_n(2)$, gives
\[ p_n(t)=\int_0^t\exp\left(-\int_r^t\lambda_n(s)\,ds\right)\widetilde f_n(r)\,dr, \]
\[ q_n(t)=\exp\left(-\int_t^2\lambda_n(s)\,ds\right)q_n(2)-\int_t^2\exp\left(-\int_t^r\lambda_n(s)\,ds\right)f_n(r)\,dr. \]
In particular,
\[ q_n(0)=-\int_0^2\exp\left(-\int_0^r\lambda_n(s)\,ds\right)f_n(r)\,dr+\exp\left(-\int_0^2\lambda_n(s)\,ds\right)q_n(2). \]
The last term is at most $\delta e^{-2\sqrt2\Omega_n}$. Both $f_n$ and $D_n$ here belong to the actual coupled solution selected above.

The kernels decay at rate at least $c\Omega_n$, with $c=\sqrt2$, and have integrals at most $C/\Omega_n$. Since $(|q_n|+|p_n|)^2\le2\delta(|q_n|+|p_n|)$,
\[ \begin{aligned} |q_n(t)|+|p_n(t)|\le{}&\delta e^{-c\Omega_n(2-t)}+C_A(\mathsf{S}/\Omega_n+\Omega_n^{-5})\\ &+2C_0\delta\Omega_n\int_0^2e^{-c\Omega_n|t-r|}(|q_n(r)|+|p_n(r)|)\,dr. \end{aligned} \]
Fix $0<c'<c$. Direct integration gives
\[ \Omega_n\int_0^2e^{-c\Omega_n|t-r|}\left(\frac{\mathsf{S}}{\Omega_n}+\Omega_n^{-5}+\delta e^{-c'\Omega_n(2-r)}\right)dr\le\frac{C}{c-c'}\left(\frac{\mathsf{S}}{\Omega_n}+\Omega_n^{-5}+\delta e^{-c'\Omega_n(2-t)}\right). \]
Divide the preceding inequality for $|q_n|+|p_n|$ by the positive expression in the last parentheses and take the supremum in $t$. The initial choice of $\delta$ makes the coefficient of this supremum on the right at most $1/2$. Absorption proves \eqref{prep:interior-tube}, with $c'$ renamed $c$. On $[0,1]$, the terminal term is at most $C\delta e^{-c\Omega_n}\le C\Omega_n^{-5}$.
\end{proof}

\Needspace{6\baselineskip}
\begin{lemma} \label{prep:phase-bounds}
Every initial choice whose solution satisfies the bootstrap assumptions on $[0,2]$ obeys, for $0<T\le1$,
\[ \sup_{0\le t\le T}\left|\log\frac{I_n(t)}{2\Gamma_n\ell_n^2}\right|\le C\left(\frac{\mathsf{S}+T\mathsf{S}^2}{\Omega_n}+T\Omega_n^{-5}\right). \]
Its shape variables, angular impulse, and phase satisfy part (ii) of Proposition~\ref{prep:reference}.
\end{lemma}

\begin{proof}
By \eqref{prep:angular-moment} and the bootstrap lower bound on $D_n$,
\[ I_n\ge2\Gamma_nD_n-C\Gamma_n\ell_n^2\Omega_n^{-7}\ge\tfrac12\Gamma_n\ell_n^2>0. \]
Recall $S_n=a_{1,n}$. At the point rhombus, $z_{2,n}^2=-R_*z_{1,n}^2$, so
\[ \frac{M_{2,n}}{I_n}=\frac{z_{1,n}^2-z_{2,n}^2}{|z_{1,n}|^2-|z_{2,n}|^2}=\frac{1+R_*}{1-R_*}e^{2i\theta_n}=\sqrt2e^{2i\theta_n}. \]
The centroid terms in $M_{2,n}$ are $2\Gamma_n(z_{1,n}^2-z_{2,n}^2)$, since the terms in $\delta_n^2$ cancel. By Theorem~\ref{meta:reduction}(i), each centered second patch moment is bounded by $C\Gamma_na_n^2\Omega_n^{-2}\le C\Gamma_n\ell_n^2\Omega_n^{-5}$. Together with \eqref{prep:angular-moment}, Taylor expansion in $p_n,q_n$ gives
\[ \frac{M_{2,n}}{I_n}=\sqrt2e^{2i\theta_n}+O(|q_n|+|p_n|+\Omega_n^{-5}). \]
Thus \eqref{prep:angular-derivative} gives
\[ \left|\frac{\dot I_n}{I_n}-2\sqrt2\operatorname{Re}(S_ne^{-2i\theta_n})\right|\le C\bigl[\mathsf{S}(|q_n|+|p_n|+\Omega_n^{-5})+\mathsf{T}\ell_n+\Omega_n\mathsf{T}^2\ell_n^2+\Omega_n^{-4}\mathsf{T}\ell_n\bigr]. \]

For large $\Lambda$, $B(d_n(t),d_n/64)$ contains quadrupole $n$, and $B(d_n(t),d_n/16)$ excludes every other support. The weighted external absolute circulation is at most $C\mathsf{S}$, and $\Gamma_n/(d_n/16)^2\le C\mathsf{S}$. By \eqref{meta:group-strain} and \eqref{prep:center-residual},
\[ |\dot S_n|\le C\mathsf{S}^2+C\mathsf{T}r_n. \]
The point phase equation in Section~\ref{sec:four-point}, expanded at the rhombus with scale $D_n$, and the centroid error give
\[ \dot\theta_n=-\Omega_n\frac{\ell_n^2}{D_n}+e_{\theta,n},\qquad |e_{\theta,n}|\le C\bigl(\Omega_n(|q_n|+|p_n|)+\mathsf{S}+\Omega_n^{-4}+\Omega_n\mathsf{T}^2\ell_n^2\bigr). \]
Fix $0<T\le1$. Lemmas~\ref{prep:shooting} and~\ref{prep:moment-bounds} imply
\[ \begin{aligned} \sup_{0\le t\le T}(|q_n(t)|+|p_n(t)|)&\le C(\mathsf{S}/\Omega_n+\Omega_n^{-5}),\\ \sup_{0\le t\le T}\left|\dot D_n(t)/D_n(t)\right|+\sup_{0\le t\le T}|e_{\theta,n}(t)|&\le C(\mathsf{S}+\Omega_n^{-4}),\\ \sup_{0\le t\le T}|\dot S_n(t)|&\le C(\mathsf{S}^2+\Omega_n^{-4}),\qquad \sup_{0\le t\le T}|S_n(t)|\le C\mathsf{S}. \end{aligned} \]
Here we used $\mathsf{T}\le1$ and $\ell_n=e^{-\Omega_n^\alpha}$ to absorb the terms containing $\ell_n$.

It remains to integrate the oscillatory term in $\dot I_n/I_n$. The phase equation gives
\[ \dot{\bigl(e^{-2i\theta_n}\bigr)}=\left(\frac{2i\Omega_n\ell_n^2}{D_n}-2ie_{\theta,n}\right)e^{-2i\theta_n}. \]
Solving for $e^{-2i\theta_n}$ and integrating by parts, we obtain
\[ \begin{aligned} \int_0^t S_n(s)e^{-2i\theta_n(s)}\,ds={}&\left[\frac{D_n(s)S_n(s)e^{-2i\theta_n(s)}}{2i\Omega_n\ell_n^2}\right]_{s=0}^{s=t}\\ &-\int_0^t\frac{D_n(s)\dot S_n(s)+S_n(s)\dot D_n(s)}{2i\Omega_n\ell_n^2}e^{-2i\theta_n(s)}\,ds\\ &+\int_0^t\frac{D_n(s)S_n(s)e_{\theta,n}(s)}{\Omega_n\ell_n^2}e^{-2i\theta_n(s)}\,ds. \end{aligned} \]
Since $D_n/\ell_n^2\asymp1$, the preceding bounds yield
\[ \begin{aligned} \sup_{0\le t\le T}\left|\int_0^t S_n(s)e^{-2i\theta_n(s)}\,ds\right|&\le\frac{C}{\Omega_n}\left[\mathsf{S}+T(\mathsf{S}^2+\Omega_n^{-4})+T\mathsf{S}(\mathsf{S}+\Omega_n^{-4})\right]\\ &\le C\left[\frac{\mathsf{S}+T\mathsf{S}^2}{\Omega_n}+T\Omega_n^{-5}\right]. \end{aligned} \]
The remaining terms in $\dot I_n/I_n$ are bounded by $C(\mathsf{S}^2/\Omega_n+\Omega_n^{-5})$. Since $I_n(0)=2\Gamma_n\ell_n^2$, integration gives
\[ \sup_{0\le t\le T}\left|\log\frac{I_n(t)}{2\Gamma_n\ell_n^2}\right|\le C\left[\frac{\mathsf{S}+T\mathsf{S}^2}{\Omega_n}+T\Omega_n^{-5}\right]. \]
The right side is uniformly small. By \eqref{prep:angular-moment} and Lemma~\ref{prep:shooting},
\[ \begin{aligned} \sup_{0\le t\le T}(|q_n(t)|+|p_n(t)|)&\le C\left[\frac{\mathsf{S}+T\mathsf{S}^2}{\Omega_n}+\Omega_n^{-5}\right],\\ \sup_{0\le t\le T}\left[\left|\frac{I_n(t)}{2\Gamma_n\ell_n^2}-1\right|+\left|\frac{D_n(t)}{\ell_n^2}-1\right|+\left|\frac{\ell_n^2}{D_n(t)}-1\right|\right]&\le C\left[\frac{\mathsf{S}+T\mathsf{S}^2}{\Omega_n}+\Omega_n^{-5}\right]. \end{aligned} \]
Finally, integrating the phase equation gives
\[ \sup_{0\le t\le T}|\theta_n(t)-\theta_n(0)+\Omega_nt|\le CT\bigl(\mathsf{S}+T\mathsf{S}^2+\Omega_n^{-4}\bigr). \]
These are the estimates in part (ii).
\end{proof}

\begin{proof}[Proof of Proposition~\ref{prep:reference}]
Apply Lemma~\ref{prep:shooting} to the prescribed phases. Lemma~\ref{prep:patch-bounds} and Theorem~\ref{meta:reduction}(iv) give the exact patch solution on $[0,2]$ and its boundary estimates. Lemma~\ref{prep:moment-bounds} gives the bounds on $\delta_n,D_n$ and $|d_i(t)-d_j(t)|$. Together with Lemmas~\ref{prep:shooting} and~\ref{prep:phase-bounds}, these prove (i) and (ii).

For (iii), the other quadrupoles stay at a fixed positive distance from quadrupole $1$, uniformly in $N$. Their velocity at $d_1(t)$ is therefore bounded by $C\sum_{j\ge2}\Gamma_j$. Integrating \eqref{prep:center-residual}, we get
\[ \sup_{0\le t\le2}\frac{|d_1(t)-d_1|}{\ell_1}\le C\left(\frac{\sum_{j\ge2}\Gamma_j}{\ell_1}+\mathsf{T}\Omega_1\ell_1+\mathsf{T}\ell_1+\Omega_1^{-4}\right). \]
The last three terms tend to zero. For the first, $\Omega_j=\Omega_1B^{j-1}$ and $\ell_j=e^{-\Omega_j^\alpha}$ imply
\[ \frac{\sum_{j\ge2}\Gamma_j}{\ell_1}\le C\Omega_2\exp\bigl(-(2B^\alpha-1)\Omega_1^\alpha\bigr)\longrightarrow0. \]
All estimates apply to every initial choice that remains within the bootstrap assumptions through time 2.
\end{proof}

\section{Second solution: cascade of separating dipoles} \label{sec:second-solution}

We change only the four initial disk centers at scale $N$ in the reference solution of Section~\ref{sec:first-solution}. They are placed in a fixed configuration of separating dipoles, with the same radii and signed circulations. At every scale $n<N$, the initial patches of $\omega_N^{(1)}$ and $\omega_N^{(2)}$ agree. Both solutions then evolve under their own Biot--Savart velocities. No patch is moved or replaced at a later time.

The objective is to make scales $N-1,\ldots,1$ undergo the transition to separating dipoles successively, with scale $1$ separating before time $1/2$. The field of the modified $N$-th scale acts on every scale from time zero. Thus each scale has already moved when the preceding scale separates, and the proof must retain that earlier motion. We use $\tau_n$ for the time when scale $n$ reaches a fixed dipole separation, with $\tau_N=0$. The intended motion at scale $n<N$ is
\[ \begin{array}{c|c|c} [0,\tau_{n+1}]&[\tau_{n+1},\tau_n]&[\tau_n,1]\\ \text{near the rotating rhombus}&\text{transition and separation}&\text{separating dipoles}. \end{array} \]
The difference from the reference remains small before $\tau_{n+1}$ and grows exponentially during the first part of $[\tau_{n+1},\tau_n]$. The second configuration then follows the fixed separating trajectory. After $\tau_n$, the dipole centers separate at speed comparable to $\Omega_n\ell_n$. The second moment of scale $n+1$ supplies the strain that starts this growth. A large frequency ratio $\Omega_{n+1}/\Omega_n=B$ keeps the effects of the earlier motion and the responses of slower scales smaller than the interaction between adjacent scales.

This section explains the proof strategy and the approximate computations used in the later sections. Section~\ref{transfer:material-section} proves the estimates for one quadrupole, and Section~\ref{array:section} verifies them simultaneously for the interacting system. Figure~\ref{fig:second-dependencies} summarizes the induction.
\subsection{Initial configurations and phases} \label{second:configuration}

We choose the initial phases so that separation at scale $n+1$ produces a positive growing coefficient at scale $n$.

Take the separating trajectory of Proposition~\ref{prop:separating} with $\ell=\Omega=1$, write its time as $s$, and normalize its time and rotation by
\[ \lim_{s\to-\infty}e^{-2\sqrt2s}q(s)=1,\qquad \lim_{s\to-\infty}(\theta(s)+s)=0. \]
Choose $K$ sufficiently large for a fixed small tolerance $\varepsilon_*>0$, following the parameter order in Section~\ref{array:choices}. It is independent of the scale and of $N$. Let $\theta_\infty$ be the limiting angle of the negative-to-positive dipole vector in \eqref{eq:dipole-asymptotics}. Pair the nearby opposite signs, let $X_1,X_2$ be their pair centers, and set
\[ r_{\mathrm{out}}=|X_1(K)-X_2(K)|,\qquad \beta=\theta_\infty+K. \]
Choose $\vartheta_N$ arbitrarily and put the four initial disk centers at scale $N$ at
\[ d_N+\ell_Ne^{i(\vartheta_N-\theta_\infty)}\{\pm z_1(K),\pm z_2(K)\}. \]
Using the same pairing for the patch centroids, define the actual separation time and associated dipole angle, whenever the prescribed distance is reached, by
\[ \begin{aligned} \tau_n&=\inf\{t\ge0:|X_1(t)-X_2(t)|=r_{\mathrm{out}}\ell_n\},\\ \theta_{d,n}&=\arg z_{1,n}^{(2)}(\tau_n)+\theta_\infty-\theta(K)\pmod\pi. \end{aligned} \]
For a rotated copy of the fixed shape, this angle is its limiting dipole angle. In particular, $\tau_N=0$ and $\theta_{d,N}=\vartheta_N$. Later angle differences use their small representatives modulo $\pi$.

For $t\ge\tau_{n+1}$, the linear part of the source second moment has size $\Gamma_{n+1}^2(t-\tau_{n+1})$, inducing strain of size $\Gamma_{n+1}^2(t-\tau_{n+1})/|d_n-d_{n+1}|^4$. Formula~\eqref{eq:positive-coefficient} gives the positive coefficient $\kappa_n$ below when the phases maximize its cosine. After the forcing tail decays, the approximation $q_n^{(2)}(t)-q_n^{(1)}(t)\approx\kappa_ne^{2\sqrt2\Omega_n(t-\tau_{n+1})}$ predicts a delay $\Delta_n$ before scale $n$ reaches the fixed shape at normalized orbit time $K$.
\begin{equation} \label{second:seed-delay}
\kappa_n=\frac{\sqrt3\,\Gamma_{n+1}^2}{2\sqrt2\pi^2\Omega_n^2|d_n-d_{n+1}|^4},\qquad \Delta_n=\frac1{\Omega_n}\left(\frac{\log(1/\kappa_n)}{2\sqrt2}+K\right).
\end{equation}
Define the predicted separation times by
\[ \tau_{n,N}=\sum_{j=n}^{N-1}\Delta_j,\qquad \tau_{N,N}=0. \]

At $\tau_{n+1,N}$, the predicted receiving phase is $\varphi_n-\Omega_n\tau_{n+1,N}$ and the source angle is $\vartheta_{n+1}$. Since $d_n-d_{n+1}>0$, the following downward recursion maximizes the cosine in \eqref{eq:positive-coefficient} and assigns the next dipole angle.
\begin{equation} \label{second:phase-choice}
\begin{aligned}
\varphi_n&=\Omega_n\tau_{n+1,N}+\frac{\arctan(1/\sqrt2)-\pi}{2}-\vartheta_{n+1}\pmod\pi,\\
\vartheta_n&=\varphi_n-\Omega_n\tau_{n,N}+\beta\pmod\pi,\qquad n=N-1,\ldots,1.
\end{aligned}
\end{equation}
Choose $\varphi_N$ arbitrarily. With all phases now fixed, Proposition~\ref{prep:reference} chooses $(q_1(0),\ldots,q_N(0))$ and gives the reference solution $\omega_N^{(1)}$. For $\omega_N^{(2)}$, retain its initial disks at every scale $n<N$ and use the prescribed separating shape at scale $N$, with unchanged radii, signs, and vorticity values. Give every disk the labels $F_i(0,\theta)=e^{i\theta}$. Theorem~\ref{meta:reduction}(iv) gives local existence. The phases may depend on $N$, while all subsequent estimates are uniform in these choices.

Finally, the choice of scales \eqref{eq:scales} gives
\[ \log(1/\kappa_n)=4\Omega_{n+1}^\alpha+O(n+\log B+1),\qquad \sum_{n\ge1}\Delta_n\le C_{\alpha,B}\Lambda^{\alpha-1}. \]
For fixed $B$, choose $\Lambda$ large so that $0<\kappa_n<1$, $\Delta_n>0$, and $\sum_n\Delta_n<1/4$. The later estimates prove continuation through time 1 and control the actual errors $\tau_n-\tau_{n,N}$ and $\theta_{d,n}-\vartheta_n$.
\subsection{Growth and continued separation at one scale} \label{second:one-scale}

Suppose scale $n+1$ reaches its prescribed dipole separation at $t_0=\tau_{n+1}$, and set $s=\Omega_n(t-t_0)$. The estimates in this subsection assume initial differences of size $O(\kappa_n)$, positive weighted forcing, and the stated bounds for the other quadrupoles. The next subsection explains how the induction verifies these assumptions. All estimates are made up to the first possible failure of a shape, separation, or patch boundary bound. They then give continuation through time 1.

\paragraph{Differences from the reference.}

Before scale $j$ reaches its prescribed separation, write $Y_j$ for its dimensionless difference from the reference, including the integral $W_j$ in \eqref{transfer:core-memory} that controls boundary differences. Section~\ref{array:weights-section} gives the precise equivalent norms, which satisfy
\[ Y_j\asymp\frac{|\triangle z_{1,j}|+|\triangle z_{2,j}|}{\ell_j}+\frac{d_j|\triangle M_{1,j}|}{\Gamma_j\ell_j^2}+\frac{|\triangle d_j|}{d_j}+W_j. \]
The first term controls shape, impulse, and phase. The first moment $M_{1,j}=4\Gamma_j\delta_j$ controls the offset between the positive and negative pairs. Here $\triangle d_j$ is the difference between the moving quadrupole centers, whereas $d_j$ in the denominator is the fixed initial location. In particular,
\[ |\triangle M_{1,j}|\le C\frac{\Gamma_j\ell_j^2}{d_j}Y_j,\qquad |\triangle M_{2,j}|\le C\Gamma_j\ell_j^2Y_j. \]
These bounds convert differences of configurations into differences of the velocity acting on another scale. Small $Y_j$ keeps the second solution's quadrupole near its reference counterpart. Along the subsequent fixed trajectory, only a uniform bound on $Y_j$ is needed.

\paragraph{Why a relative estimate is needed.}

By the preceding estimate for $\log(1/\kappa_n)$, the coefficient $\kappa_n$ is smaller than every fixed negative power of $\Omega_n$. An absolute normalized centroid error $O(\Omega_n^{-5})$ therefore cannot determine the sign of this unstable coefficient. Theorem~\ref{meta:reduction}(ii) instead bounds the difference of the two centroid errors. With $F(s)=\Omega_n^{-1}\|\triangle A_n(t_0+s/\Omega_n)\|_R$, its relevant consequence is
\[ \frac{|\triangle E_n(t_0+s/\Omega_n)|}{\Omega_n\ell_n}\le C\frac{a_n^2}{\ell_n^2}\bigl[Y_n(t_0+s/\Omega_n)+F(s)\bigr],\qquad \frac{a_n^2}{\ell_n^2}=O(\Omega_n^{-3}). \]
There is no additive error independent of the differences. The integral controlling the boundary difference retains its actual value at $t_0$. From \eqref{meta:history-restart},
\[ W_n(t_0+s/\Omega_n)\le e^{\sigma_cs}W_n(t_0)+C\int_0^s e^{\sigma_c(s-r)}\bigl[Y_n(t_0+r/\Omega_n)+F(r)\bigr]\,dr,\qquad \sigma_c<2\sqrt2. \]
Thus $W_n(t_0)=O(\kappa_n)$ and inputs bounded by $C\kappa_ne^{2\sqrt2s}$ give the same growth bound for $W_n$. Over the required interval $S=O(\log(1/\kappa_n))$, the resulting relative error in the unstable coefficient is
\[ \frac1{\kappa_n}\int_0^S e^{-2\sqrt2s}\frac{|\triangle E_n(t_0+s/\Omega_n)|}{\Omega_n\ell_n}\,ds\le C\Omega_n^{-3}\log(1/\kappa_n)=O_{\alpha,B}(\Omega_n^{\alpha-3}). \]
The induction proves the growth bound and this error bound together. It does not assume an absolute patch error smaller than $\kappa_n$.

\paragraph{From positive forcing to the separating trajectory.}

Section~\ref{transfer:linearization} uses shape coordinates $(\widehat p,\widehat q)$ in which the separating point-vortex trajectory satisfies $\widehat p=0$ and $\widehat q_s=2\sqrt2\widehat q$, with $\widehat q/q\to1$ at the rhombus. The nonlinear remainder in the unstable equation vanishes when $\widehat p=0$.

Put $Q_n(t)=\triangle\widehat q_n(t)$. The common initial patches give $Q_n(0)=0$, regardless of the unknown values $q_n^{(1)}(0)$ chosen by shooting. The next subsection bounds the motion before $t_0$ by $|Q_n(t_0)|/\kappa_n\le CM_{\mathrm{hist}}B^{-2}+o_{\Lambda\to\infty}(1)$, with $M_{\mathrm{hist}}$ fixed before $B,\Lambda$. In the normalized equation for $Q_n$, let $\kappa_n f_q$ collect the forcing from completed scales $j>n$, their subsequent influence through scales $i<n$, and the slower-scale differences already present at $t_0$. The additional forcing produced when $n$ perturbs the slower scales after $t_0$ is estimated separately. The leading calculation is
\[ \begin{aligned} \frac d{ds}Q_n(t_0+s/\Omega_n)&=2\sqrt2Q_n(t_0+s/\Omega_n)+\kappa_nf_q(s)+\text{relative error},\\ \mathcal A_n&=\frac{Q_n(t_0)}{\kappa_n}+\int_0^\infty e^{-2\sqrt2s}f_q(s)\,ds. \end{aligned} \]
Here $\mathcal A_n$ is the normalized coefficient of unstable growth. The forcing may change sign. Under the completed-scale bounds below, the chosen phases give $\int_0^\infty e^{-2\sqrt2s}f_q(s)\,ds=1+O(\varepsilon_*)$. Together with the small value $Q_n(t_0)/\kappa_n$, this determines the sign independently of the common initial $q_n(0)$. On a stopped interval, use the actual forcing up to the stop and the explicit ideal source afterwards. Its tail is subtracted in the variation-of-constants formula, so this notation assumes no future physical motion.

If the accumulated relative errors are at most $\varepsilon_*$, the calculation of Section~\ref{transfer:material-block} gives, after the initial forcing tail has decayed,
\[ Q_n(t_0+s/\Omega_n)=\kappa_ne^{2\sqrt2s}\bigl[\mathcal A_n+O(\varepsilon_*)+O((1+s)^3e^{-\sqrt2s})\bigr]. \]
Fix a small $q_*>0$, independently of $n,N$, and let $T_n$ be the first physical time when $Q_n(T_n)=q_*$. If $\mathcal A_n\ge3/4$ and the other bounds do not fail first, then
\[ \Omega_n(T_n-t_0)=\frac1{2\sqrt2}\log\frac{q_*}{\kappa_n\mathcal A_n}+O(\varepsilon_*). \]
The stable shape difference decays under its linear equation, while impulse and phase have at most polynomial unforced growth. Their forced estimates give, at $T_n$,
\[ |\widehat p_n^{(2)}|+\left|\frac{I_n^{(2)}}{2\Gamma_n\ell_n^2}-1\right|+\frac{|\delta_n^{(2)}|}{\ell_n}\le C\varepsilon_*. \]
After the phase correction in Section~\ref{transfer:linearization}, the phase difference from the reference is also $O(\varepsilon_*)$. The reference shape is much smaller than $q_*$. Thus, after matching its phase, the second configuration is close to the positive separating trajectory at $\widehat q=q_*$. The remaining normalized point-vortex duration is the fixed number $K-\log q_*/(2\sqrt2)$. The vector field is Lipschitz on this collision-free segment, so the normalized centroid error stays $O(\varepsilon_*)$. The normalized pair-center distance crosses $r_{\mathrm{out}}$ with derivative at least one, so the crossing time also changes by $O(\varepsilon_*)$ in normalized time. Adding the two durations and comparing phases yields
\begin{equation} \label{second:local-passage}
\begin{aligned}
\tau_n-\tau_{n+1}&=\frac1{\Omega_n}\left[\frac{\log(1/(\kappa_n\mathcal A_n))}{2\sqrt2}+K+O(\varepsilon_*)\right],\\
\theta_{d,n}&=\theta_n^{(1)}(\tau_n)+\beta+O(\varepsilon_*).
\end{aligned}
\end{equation}

\paragraph{Why the dipoles keep separating.}

Each opposite-sign pair has length close to $\ell_n$ and translates with speed close to $\Omega_n\ell_n$. The two leading velocities are opposite. At pair-center distance $R$, the other pair changes its relative vector at rate $O(\Gamma_n\ell_n^2/R^3)$. The fractional change accumulated along the expected separation is bounded by
\[ \int_0^{1-\tau_n}\frac{\Gamma_n\ell_n}{[\ell_n(r_{\mathrm{out}}+\Omega_nt)]^3}\,dt\le Cr_{\mathrm{out}}^{-2}. \]
The external strain contributes at most $C\mathsf S$ to the fractional rate of change of the dipole vector. The patch correction contributes at most $C\Omega_n(a_n^2/\ell_n^2)\Omega_n^{-2}$, since the normalized boundary deviation is $O(\Omega_n^{-2})$. Over the whole remaining physical interval,
\[ \int_{\tau_n}^1\left(\mathsf S+\Omega_n\frac{a_n^2}{\ell_n^2}\Omega_n^{-2}\right)dt\le\mathsf S+C\Omega_n^{-4}\ll\varepsilon_*. \]
Here $\mathsf S$ bounds the strain of the other quadrupoles and tends to zero with $\Lambda^{-1}$. Thus the pair lengths and directions remain close to their initial values. Section~\ref{transfer:outgoing-material} makes this a bootstrap estimate: the relative pair-center velocity has projection at least $\Omega_n\ell_n$ in the separating direction, and
\[ |X_1(t)-X_2(t)|\asymp\ell_n\bigl(r_{\mathrm{out}}+\Omega_n(t-\tau_n)\bigr),\qquad \tau_n\le t\le1. \]
The same control gives
\[ \dot M_{2,n}^{(2)}(t)=-\frac{2i}{\pi}\Gamma_n^2e^{2i\theta_{d,n}}+O(\varepsilon_*\Gamma_n^2). \]
This is the force that scale $n$ supplies to scale $n-1$. These conclusions hold up to the common stopping time until the geometric and boundary estimates give continuation.
\subsection{The cascade by induction} \label{second:cascade-estimates}

At $t=\tau_{n+1}$, assume the stages $j>n$ are complete. Each completed quadrupole supplies the separation and moment estimates above, a time and angle close to their prescribed values, and the bound
\[ Y_j(t)\le M_{\mathrm{hist}}e^{c\Omega_j(t-\tau_j)},\qquad 0\le t\le\tau_j,\qquad 0<c<2\sqrt2. \]
The constant $M_{\mathrm{hist}}$ is independent of $j,N$. During $[\tau_{n+1},\tau_n]$, we estimate $n$ and all $i<n$ together, while the completed scales $j>n$ keep separating and retain their bounds on $[0,\tau_j]$. At the first stage $n=N-1$, $t_0=0$ and all differences at scales $j<N$ vanish. The bounds for scale $N$ follow from its prescribed configuration.

\paragraph{Force from faster scales.}

The second moment of $n+1$ is its value at $\tau_{n+1}$, of size $O(\Gamma_{n+1}\ell_{n+1}^2)$, plus an approximately linear term in $t-\tau_{n+1}$ with coefficient of size $\Gamma_{n+1}^2$. Writing $r=t-\tau_{n+1}$, the constant and linear terms have weights $\int_0^\infty e^{-2\sqrt2\Omega_nr}\,dr=(2\sqrt2\Omega_n)^{-1}$ and $\int_0^\infty re^{-2\sqrt2\Omega_nr}\,dr=(8\Omega_n^2)^{-1}$. Their relative size is of order $B^{-1}$, since
\[ \frac{\Gamma_{n+1}\ell_{n+1}^2/\Omega_n}{\Gamma_{n+1}^2/\Omega_n^2}=\frac1{2\pi B}. \]
Both the initial second moment $M_{2,n+1}^{(2)}(\tau_{n+1})$ and the reference second moment $M_{2,n+1}^{(1)}(t)$ are bounded by $C\Gamma_{n+1}\ell_{n+1}^2$. Their contributions to the weighted forcing coefficient are therefore at most $C\kappa_n/B$, giving an $O(B^{-1})$ relative error.

The earlier motion has an $O(B^{-2})$ relative error. Its effective duration in the weighted integral is $O(\Omega_{n+1}^{-1})$. For $B>4\sqrt2/c$,
\[ \frac{\Gamma_{n+1}\ell_{n+1}^2}{|d_n-d_{n+1}|^4}\int_{-\infty}^0e^{(c\Omega_{n+1}-2\sqrt2\Omega_n)t}\,dt\le\frac{C\kappa_n}{B(cB-2\sqrt2)}\le C\kappa_nB^{-2}. \]
The negative-time integral is an upper bound after shifting $\tau_{n+1}$ to zero. Applying the weighted comparison of Section~\ref{array:induction-section}, Step 3, only up to $\tau_{n+1}$ gives
\[ \sup_{0\le t\le \tau_{n+1}}e^{2\sqrt2\Omega_n(\tau_{n+1}-t)}Y_n(t)\le CM_{\mathrm{hist}}\kappa_nB^{-2}+C(1+M_{\mathrm{hist}})\kappa_ne^{-c\Omega_{n+2}^\alpha}. \]
This controls the whole earlier motion, not just $Y_n(\tau_{n+1})$, and includes $W_n(\tau_{n+1})$. Since the exponential weight is at least one, it also keeps scale $n$ near its rhombus throughout $[0,\tau_{n+1}]$.

For clarity, the smallness of the older sources follows from a comparison of two exponents. The completed delays give
\[ \Omega_n\sum_{k=n+1}^{j-1}\frac{\log(1/\kappa_k)}{\Omega_k}\le\frac{5\Omega_{n+1}^\alpha}{B^{1-\alpha}-1},\qquad \log\frac{\Gamma_j^2}{\Gamma_{n+1}^2}=-4(\Omega_j^\alpha-\Omega_{n+1}^\alpha)+2(j-n-1)\log B. \]
For $j\ge n+2$, the negative term is of size $\Omega_j^\alpha$, while the extra amplification is smaller by a factor of order $B^{-1}$. Taking $B$ large and then $\Lambda$ large absorbs the spatial factors, which grow only geometrically in $j$. Section~\ref{array:sources-section} bounds their total contribution, divided by $\kappa_n$, by $C(1+M_{\mathrm{hist}})e^{-c\Omega_{n+2}^\alpha}$.

\paragraph{Errors $\tau_n - \tau_{n,N}, \theta_{d,n} - \vartheta_n$ in time and angle.}

Put $h_j=\Omega_j(\tau_j-\tau_{j,N})$. The already completed scale $n+1$ supplies $|h_{n+1}|\le C_h\varepsilon_*$ and $\theta_{d,n+1}-\vartheta_{n+1}=-h_{n+1}+O(\varepsilon_*)$. At $\tau_{n+1}$ the reference phase is
\[ \theta_n^{(1)}(\tau_{n+1})=\varphi_n-\Omega_n\tau_{n+1,N}-B^{-1}h_{n+1}+O(\varepsilon_*). \]
The initial phases maximize the cosine in \eqref{eq:positive-coefficient}. Substitution of these actual phases gives
\[ \mathcal A_n=\cos\bigl(2(1+B^{-1})h_{n+1}\bigr)+O(\varepsilon_*)=1+O(h_{n+1}^2+\varepsilon_*). \]
The bounded moment and retained past contribute $O(B^{-1}+M_{\mathrm{hist}}B^{-2})$, and older sources give $o_{\Lambda\to\infty}(1)$. For the nearest source, the slower response contributes $O((\Omega_nd_{n+1})^{-1})$. The convolution below gives $\Omega_n^{-1}$, and the first-moment normalization contributes $d_{n+1}^{-1}$. This is uniformly small since $\Omega_nd_{n+1}\ge\Lambda B/4$ for $B>2$. For older sources, $d_j^{-1}$ remains inside their exponentially small sum. The parameter choices make these errors $O(\varepsilon_*)$, so $\mathcal A_n\ge3/4$. This positivity uses only the completed stages. Once $n$ separates, subtracting the prescribed delay from \eqref{second:local-passage} gives
\[ h_n=B^{-1}h_{n+1}-\frac{\log\mathcal A_n}{2\sqrt2}+O(\varepsilon_*),\qquad |h_n|\le B^{-1}|h_{n+1}|+C|h_{n+1}|^2+C\varepsilon_*. \]
Fix $C_h$ and then $\varepsilon_*$ using the universal constants in this recurrence, as specified in Section~\ref{array:choices}. For large $B$, the recurrence improves $|h_n|\le C_h\varepsilon_*$. Starting from $h_N=0$, the same timing and angle bounds are available for the next coefficient.

\paragraph{The response of slower scales.}

The motion of $n$ also forces every $i<n$, which then acts back on $n$. Let $\Xi$ bound these interaction maps in the weighted comparison norm of Section~\ref{array:weights-section}. Then $\Xi\to0$ as $\Lambda\to\infty$. The slower subsystem has growth rate at most $C\Omega_{n-1}+C\Xi<\sqrt2\Omega_n$ for large $B$. Hence
\[ \int_{\tau_{n+1}}^t e^{\sqrt2\Omega_n(t-r)}e^{2\sqrt2\Omega_n(r-\tau_{n+1})}\,dr\le\frac1{\sqrt2\Omega_n}e^{2\sqrt2\Omega_n(t-\tau_{n+1})}. \]
On the common stopped interval, assume $Y_n(t)\le C\kappa_ne^{2\sqrt2\Omega_n(t-\tau_{n+1})}$ with a sufficiently large fixed constant. The part generated by scale $n$ then satisfies, schematically,
\[ \max_{i<n}\left(\frac{d_i}{d_n}\right)^4Y_i(t)\le\text{inherited and source-driven terms}+\frac{C\Xi}{\Omega_n}\kappa_ne^{2\sqrt2\Omega_n(t-\tau_{n+1})}. \]
The product $\kappa_ne^{2\sqrt2\Omega_n(T_n-\tau_{n+1})}$ is $O(q_*)$, and the remaining transition takes bounded normalized time. Since $d_i>d_n$, this estimate also controls each $Y_i$. Acting back on $n$ costs another $\Xi$, and converting to normalized time gives another $\Omega_n^{-1}$. Lemma~\ref{array:follower-convolution} bounds the resulting feedback norm by $C\Xi^2/\Omega_n^2$, including all mutual interactions among slower scales. Their actual states at $\tau_{n+1}$ and responses to completed sources remain in $\kappa_nf_q$. The feedback contains only their response to the subsequent motion of $n$.

The estimates for the stable shape and the remaining coordinates give the following schematic bound on the common stopped interval, for $S\le C\log(1/\kappa_n)$:
\[ \begin{aligned} \sup_{0\le s\le S}\frac{e^{-2\sqrt2s}Y_n(\tau_{n+1}+s/\Omega_n)}{\kappa_n}\le{}&C+C\left(\frac{\mathsf S}{\Omega_n}+\Omega_n^{-3}+\frac{\Xi^2}{\Omega_n^2}\right)\log(1/\kappa_n)\\ &\quad\times\sup_{0\le s\le S}\frac{e^{-2\sqrt2s}Y_n(\tau_{n+1}+s/\Omega_n)}{\kappa_n}. \end{aligned} \]
The three small terms come from the reference and external fields, the relative patch error, and the feedback through slower scales. The other local terms are absorbed by the fixed small shape threshold and the scale choices. For large $\Lambda$, the multiplier of the last supremum is less than $1/2$. Absorption then improves this bootstrap bound. Including the slower scales' initial values and forcing from $j>n$, Section~\ref{array:induction-section}, Step 4, obtains $\sup_{i<n,\,\tau_{n+1}\le t\le\tau_n}Y_i(t)=o_{\Lambda\to\infty}(1)$, uniformly in $n,N$. They therefore cannot leave their rhombus neighborhoods first. These strict improvements prove the current and slower bounds together. For $n=1$, there are no slower scales.

\paragraph{The estimates passed to the next stage.}

The bound on the differences, including the weighted past before $\tau_{n+1}$, gives
\[ Y_n(t)\le C\kappa_ne^{2\sqrt2\Omega_n(t-\tau_{n+1})}\quad(0\le t\le T_n),\qquad \kappa_ne^{2\sqrt2\Omega_n(T_n-\tau_{n+1})}=\frac{q_*}{\mathcal A_n}(1+O(\varepsilon_*)). \]
Combining these bounds controls every earlier time by its distance from $T_n$. The interval from $T_n$ to $\tau_n$ has bounded length in normalized time and bounded linear evolution. Hence
\[ Y_n(t)\le M_{\mathrm{hist}}e^{c\Omega_n(t-\tau_n)},\qquad 0\le t\le\tau_n. \]
To keep the same constant at every stage, the proof separates the leading source from the terms depending on the previously assumed $M_{\mathrm{hist}}$. The coefficient in the new bound is at most
\[ C+C\bigl(B^{-1}+B^{-2}+o_{\Lambda\to\infty}(1)\bigr)M_{\mathrm{hist}}\le\tfrac34M_{\mathrm{hist}}. \]
The constants on the left are fixed by the one-quadrupole estimates before $B,\Lambda$ are chosen. Section~\ref{array:choices} first fixes $M_{\mathrm{hist}}$ larger than the leading constant, then chooses $B$ and $\Lambda$ to obtain this strict improvement.

The slower-scale estimate supplies the actual state of the next scale. The main factor is
\[ \frac{\Gamma_n\ell_n^2}{\Omega_n|d_{n-1}-d_n|^4}\asymp\kappa_{n-1}\left(\frac{\Omega_{n-1}}{\Omega_n}\right)^2=\kappa_{n-1}B^{-2}. \]
Consequently, for $n\ge2$,
\[ Y_{n-1}(\tau_n)\le CM_{\mathrm{hist}}\kappa_{n-1}B^{-2}+C(1+M_{\mathrm{hist}})\kappa_{n-1}e^{-c\Omega_{n+1}^\alpha}. \]
This includes $W_{n-1}(\tau_n)$, so no state or boundary label is restarted. Scale $n$ now supplies its earlier-motion bound, its time and angle estimates, and its continuing dipole and moment estimates. Every $i<n$ retains its improved rhombus bound. These are the data needed for the next step.

Finally, the geometric estimates give
\[ \operatorname{diam}(\text{quadrupole }j)\le C\ell_j(1+\Omega_j),\qquad |d_i^{(k)}(t)-d_j^{(k)}(t)|\ge\tfrac12|d_i-d_j|. \]
Since $\ell_j(1+\Omega_j)\ll d_j$, quadrupoles remain separated. Their minimum internal distances stay bounded below by $c\ell_j$, while the lengths within each dipole remain comparable to $\ell_j$, and Theorem~\ref{meta:reduction}(iv) keeps each normalized patch boundary within $C\Omega_j^{-2}$ of a circle. These bounds exclude the remaining failures on the stopped interval. Descending induction therefore reaches scale $1$ and continues every dipole through time 1, with constants independent of $N$. Moreover,
\[ \tau_1\le\sum_{n\ge1}\Delta_n+\frac{C_h\varepsilon_*}{\Omega_1}<\frac12. \]

\subsection{Conclusion} \label{second:observation-section}

\begin{proposition} \label{second:finite-solutions}
For the fixed choices above and every $N\ge2$, the prescribed initial configurations give Euler solutions $\omega_N^{(1)},\omega_N^{(2)}$ on $[0,1]$, with $u_N^{(k)}=K*\omega_N^{(k)}$. Their initial patches agree at every scale $n<N$, and the first solution satisfies Proposition~\ref{prep:reference}. In the second solution, scale $N$ starts in the prescribed separating configuration and the remaining scales reach that separation in order, with
\[ 0=\tau_N<\tau_{N-1}<\cdots<\tau_1<\tfrac12,\qquad |\tau_n-\tau_{n,N}|\le C_h\varepsilon_*/\Omega_n. \]
All patches remain disjoint. Each normalized boundary has $C^{1,1/2}$ distance at most $C\Omega_n^{-2}$ from a parameterized circle. Distances between quadrupoles remain comparable to their initial distances, and the supports lie in one bounded set independent of $N$.
\end{proposition}

Proposition~\ref{array:finite-induction}, proved in Section~\ref{array:induction-section}, establishes this statement from the reference construction and the initial data above. By Section~\ref{array:observation-section}, at time 1 the first solution's outer patches lie within $C\ell_1$ of $d_1$, and the second solution's outer patches lie at distances comparable to $\Omega_1\ell_1$. In each solution, the outer center moves by $o(\ell_1)$, uniformly in $N$. A fixed annulus centered at $d_1$ contains all four outer patches of the second solution and no other patches of either solution at time $1$. Choose a smooth $0\le\chi\le1$ supported in this annulus and equal to one on those four patches, uniformly in $N$. Then
\begin{equation} \label{second:fixed-observation}
\int\chi|\omega_N^{(1)}(1)|\,dx=0,\qquad \int\chi|\omega_N^{(2)}(1)|\,dx=4\Gamma_1.
\end{equation}
Since $A,B,\Lambda$ are fixed, $4\Gamma_1>0$ is independent of $N$. The initial difference, confined to scale $N$, tends to zero in every finite $L^p$ by \eqref{eq:initial-difference}. Section~\ref{sec:limit} extracts a common subsequence whose initial data converge to the same $\omega_0$ and whose vorticities converge strongly in $C([0,1];L^1)$. Thus $\omega^{(1)}(0)=\omega^{(2)}(0)$ and $\|\omega^{(2)}(1)-\omega^{(1)}(1)\|_{L^1}\ge4\Gamma_1>0$.

\begin{figure}[p]
\centering
\normalsize
\setlength{\abovedisplayskip}{5pt}
\setlength{\belowdisplayskip}{5pt}
\setlength{\abovedisplayshortskip}{4pt}
\setlength{\belowdisplayshortskip}{4pt}
\begin{tikzpicture}[x=1cm,y=1cm,>=Stealth,every node/.style={align=flush center,inner sep=2pt},textblock/.style={anchor=north,text width=14.7cm},half/.style={anchor=north,text width=6.9cm},flow/.style={->,semithick}]
\node[font=\bfseries] at (7.4,0.45) {One induction step at scale $n<N$};
\draw[flow] (0,-0.3) -- (15,-0.3);
\foreach \x in {0,3.8,8.2,10.4,14.8} {\draw (\x,-0.2) -- (\x,-0.4);}
\node[above] at (0,-0.2) {$0$};
\node[above] at (3.8,-0.2) {$\tau_{n+1}$};
\node[above] at (8.2,-0.2) {$T_n$};
\node[above] at (10.4,-0.2) {$\tau_n$};
\node[above] at (14.8,-0.2) {$1$};
\node[anchor=north,text width=3.4cm] at (1.9,-0.6) {Near the rhombus};
\node[anchor=north,text width=4cm] at (6,-0.6) {Unstable growth};
\node[anchor=north,text width=2cm] at (9.3,-0.6) {Transition};
\node[anchor=north,text width=4cm] at (12.6,-0.6) {Separating dipoles};
\node[anchor=north] at (7.4,-1.15) {Choice of $T_n, \tau_n$:~~$Q_n(T_n)=q_*,~~ |X_1(\tau_n)-X_2(\tau_n)|=r_{\mathrm{out}}\ell_n$};
\draw[black!30] (0,-1.75) -- (14.8,-1.75);

\node[textblock] (initial) at (7.4,-2) {\textbf{At $\tau_{n+1}$}\quad{\S\ref{array:induction-section}, Steps 1, 3}
\[ \sup_{0\le t\le\tau_{n+1}}e^{2\sqrt2\Omega_n(\tau_{n+1}-t)}Y_n(t)\le CM_{\mathrm{hist}}\kappa_nB^{-2}. \]
Scales $j>n$ are already separating.};
\draw[flow] ([yshift=-0.05cm]initial.south) -- ++(0,-0.5);

\node[textblock,yshift=-0.75cm] (forcing) at (initial.south) {\textbf{Forcing from scale $n+1$}\quad{\S\ref{transfer:outgoing-material}, \S\ref{array:sources-section}--\ref{array:phases-section}}
\[ \dot M_{2,n+1}^{(2)}=-\frac{2i}{\pi}\Gamma_{n+1}^2e^{2i\theta_{d,n+1}}+O(\varepsilon_*\Gamma_{n+1}^2). \]};
\draw[flow] ([yshift=-0.05cm]forcing.south) -- ++(0,-0.5);

\node[textblock,yshift=-0.75cm] (coupled) at (forcing.south) {\textbf{Simultaneous estimates}\quad{\S\ref{array:induction-section}, Steps 3--4}};
\node[half,xshift=-3.95cm,yshift=-0.2cm,minimum height=2.25cm] (growth) at (coupled.south) {\textbf{Scale $n$}\quad{\S\ref{transfer:linearization}--\ref{transfer:material-block}}
\[ Y_n(t)\le C\kappa_ne^{2\sqrt2\Omega_n(t-\tau_{n+1})}. \]};
\node[half,xshift=3.95cm,yshift=-0.2cm,minimum height=2.25cm] (slower) at (coupled.south) {\textbf{Scales $i<n$}\quad{\S\ref{array:weights-section}}
\\[5pt]{ Lemma~\ref{array:follower-convolution} and \eqref{array:follower-current}}};
\draw[<->,semithick] ([xshift=0.08cm]growth.east) -- ([xshift=-0.08cm]slower.west);
\draw[flow] ([xshift=-3.95cm,yshift=-0.05cm]slower.south) -- ++(0,-0.5);

\node[textblock,xshift=-3.95cm,yshift=-0.75cm] (transition) at (slower.south) {\textbf{From $T_n$ to $\tau_n$}\quad{\S\ref{transfer:material-block}, \S\ref{array:induction-section}, Step 5}\\[5pt]Follow the point-vortex trajectory for $O(\Omega_n^{-1})$.};
\draw[flow] ([yshift=-0.05cm]transition.south) -- ++(0,-0.5);

\node[textblock,yshift=-0.75cm] (conclusions) at (transition.south) {\textbf{At $\tau_n$}\quad{\S\ref{array:induction-section}, Step 6}
\[ Y_n(t)\le\tfrac34M_{\mathrm{hist}}e^{c\Omega_n(t-\tau_n)},\qquad 0\le t\le\tau_n. \]
\[ \Omega_n|\tau_n-\tau_{n,N}|\le\tfrac34C_h\varepsilon_*,\qquad \theta_{d,n}-\vartheta_n=-\Omega_n(\tau_n-\tau_{n,N})+O(\varepsilon_*). \]
For the next scale $n-1$,
\[ Y_{n-1}(\tau_n)\le CM_{\mathrm{hist}} \kappa_{n-1} B^{-2}+C(1+M_{\mathrm{hist}}) \kappa_{n-1} e^{-c\Omega_{n+1}^\alpha}\qquad(n\ge2). \]
};
\draw[flow] ([yshift=-0.05cm]conclusions.south) -- ++(0,-0.5);

\node[textblock,yshift=-0.75cm] (continue) at (conclusions.south) {\textbf{Separation continues}\quad{\S\ref{transfer:outgoing-material}}
\[ |X_1(t)-X_2(t)|\asymp\ell_n\bigl(r_{\mathrm{out}}+\Omega_n(t-\tau_n)\bigr),\qquad \tau_n\le t\le1. \]
Repeat from $\tau_N=0$ down to scale $1$, with $\tau_1<1/2$.};
\end{tikzpicture}
\caption{The bounds at $\tau_n$ supply the next induction step. Patch boundaries and geometry are controlled by Theorem~\ref{meta:reduction}(iv) and \S\ref{array:choices}. The time axis is schematic.}
\label{fig:second-dependencies}
\end{figure}
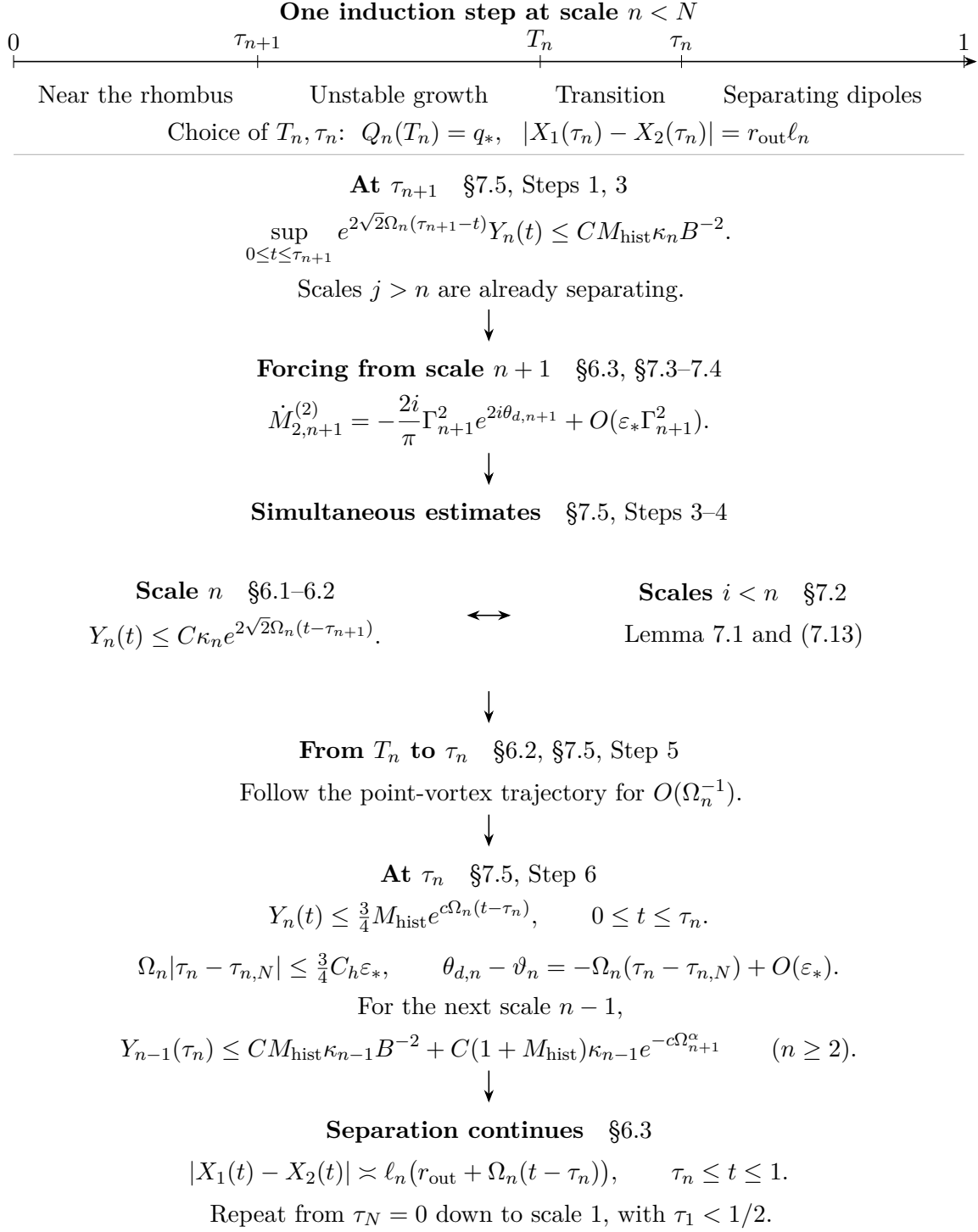
\clearpage

\section{Transition in quadrupole configurations} \label{transfer:material-section}

We compare one quadrupole with the prepared reference, starting at a physical time $t_0\ge0$. In Section~\ref{array:section}, $t_0=\tau_{n+1}$ is the time when quadrupole $n+1$ reaches the prescribed dipole separation. Both solutions have already evolved on $[0,t_0]$, and we retain their centroid, moment, and boundary differences at $t_0$. We first estimate the growth of the unstable shape variable, then the separation and second moment of the resulting dipoles. Section~\ref{array:section} verifies the external-field bounds used here.

Suppress the quadrupole index and write $\Gamma=2\pi\Omega\ell^2$, $a^2/\ell^2=2/(A\Omega^3)$, and $\rho=C_A\Omega^{-2}$. The external velocity $U$ is generated by all other quadrupoles. We use the centroid and moment equations \eqref{transfer:centroid-equation} and \eqref{transfer:neutral-identities}. The quantity $W$ is the integral in \eqref{transfer:core-memory}, retaining its value at $t_0$ as in \eqref{meta:history-restart}. To estimate $D=|z_1|^2-|z_2|^2$, we evolve the patch angular impulse $I$. Equation~\eqref{transfer:action-correction} gives
\[ \left|\frac D{\ell^2}-\frac I{2\Gamma\ell^2}\right|\le C\frac{a^2}{\ell^2}\rho^2,\qquad \left|\triangle\left(\frac D{\ell^2}-\frac I{2\Gamma\ell^2}\right)\right|\le C\frac{a^2}{\ell^2}\rho W. \]
Thus $I/(2\Gamma\ell^2)$ determines $D/\ell^2$ up to the displayed patch error.

\subsection{The exact unstable rate} \label{transfer:linearization}

Consider first the four point vortices of Section~\ref{sec:four-point}, at $\pm z_1,\pm z_2$ with circulations $\Gamma,\Gamma,-\Gamma,-\Gamma$. Their angular impulse is $2\Gamma D$, where $D=|z_1|^2-|z_2|^2$ is conserved. The orbit in Proposition~\ref{prop:separating} has $D=\ell^2$. In normalized time $s=\Omega t$, the shape variables $(p,q)$ have linear rates $-2\sqrt2,2\sqrt2$ at the rhombus.

We first express the separating orbit as $p=h(q)$. Put $r=\sqrt R$ and $r_*=\sqrt2-1$. The sine formula in the proof of Proposition~\ref{prop:separating} is analytic near $r_*$. The branch $\phi(r)=\pi-\arcsin(\sin\phi(r))$ gives
\[ \left.\frac{d\sin\phi}{dr}\right|_{r=r_*}=-\frac{2\sqrt2}{r_*},\qquad \left.\frac{dq}{dr}\right|_{r=r_*}=\frac2{r_*},\qquad \left.\frac{dp}{dr}\right|_{r=r_*}=0. \]
Consequently $p=h(q)$ near zero, with $h$ analytic and $h(0)=h'(0)=0$. Along the separating trajectory, $q>0$ and
\[ p_s-h'(q)q_s=0\qquad\text{when }p=h(q). \]
After substituting $p=h(q)$, the left side is an analytic function of $q$. Since it vanishes for all small positive $q$, it also vanishes for all small negative $q$. Thus a solution starting on this graph stays on it, for either sign of $q$.

On the graph write $q_s=f(q)=2\sqrt2q+O(q^2)$ and define
\[ \widehat q=q\exp\left(\int_0^q\left[\frac{2\sqrt2}{f(r)}-\frac1r\right]dr\right),\qquad \widehat p=p-h(q). \]
The integrand is analytic at zero, and the derivative of the coordinate change is the identity there. On $\widehat p=0$,
\[ \frac d{ds}\log|\widehat q|=\left(\frac1q+\frac{2\sqrt2}{f(q)}-\frac1q\right)f(q)=2\sqrt2. \]
Hence $\widehat q_s=2\sqrt2\widehat q$ exactly on the separating orbit. Away from this graph, analyticity and the linearization give
\[ \widehat q_s=2\sqrt2\widehat q+\widehat pB_q,\qquad \widehat p_s=-2\sqrt2\widehat p+\widehat pB_p,\qquad |B_q|+|B_p|\le C(|\widehat q|+|\widehat p|). \]
Here $B_q,B_p$ are analytic functions of $(\widehat q,\widehat p)$. For example, subtracting the value of $\widehat q_s$ at $\widehat p=0$ gives $\widehat p\int_0^1\partial_{\widehat p}(\widehat q_s)(\widehat q,r\widehat p)\,dr$, whose coefficient vanishes at the origin.

The adjusted phase \eqref{eq:adjusted-phase} satisfies $\Theta_s=-1+\psi(\widehat q,\widehat p)$, where $\psi=O((|\widehat q|+|\widehat p|)^2)$. Set
\[ \widehat\Theta=\Theta-\int_0^{\widehat q}\frac{\psi(r,0)}{2\sqrt2r}\,dr. \]
Differentiation removes the phase error on $\widehat p=0$ and gives $\widehat\Theta_s=-1+\widehat pB_\Theta$, with
\[ B_\Theta=\int_0^1\partial_{\widehat p}\psi(\widehat q,r\widehat p)\,dr-\frac{\psi(\widehat q,0)}{2\sqrt2\widehat q}B_q,\qquad |B_\Theta|\le C(|\widehat q|+|\widehat p|). \]
The quotient is analytic at zero.

For general $D>0$ near $\ell^2$, the common frequency in \eqref{eq:shape-ode} is $\Gamma/(2\pi D)=\Omega\ell^2/D$, in place of $\Omega$. Keeping $s=\Omega t$ gives
\[ \widehat q_s=\frac{\ell^2}{D}\bigl(2\sqrt2\widehat q+\widehat pB_q\bigr),\qquad \widehat p_s=\frac{\ell^2}{D}\bigl(-2\sqrt2\widehat p+\widehat pB_p\bigr). \]
The phase satisfies $\widehat\Theta_s=(\ell^2/D)(-1+\widehat pB_\Theta)$, and $D_s=0$. Thus the nonlinear shape and phase terms vanish when $\widehat p=0$, with no term depending quadratically on $\widehat q$ alone. The next subsection estimates both the stable difference and $\ell^2/D$ during the long growth interval.

For the orbit normalized in Section~\ref{second:configuration}, $\widehat q/q\to1$ and $\widehat\Theta-\theta\to0$ at the rhombus. Its equations therefore give $\widehat q(s)=e^{2\sqrt2s}$ and $\widehat\Theta(s)=-s$ as long as these local coordinates are defined. We retain $K,\theta_\infty,\beta=\theta_\infty+K$ from that section.

\subsection{Departure and dipole separation} \label{transfer:memory-saddle} \label{transfer:material-block}

We now compare two four-patch motions, the first being the reference in Proposition~\ref{prep:reference}. At the original time $t=0$, their radii and signed circulations agree and $F_i^{(1)}(0,\theta)=F_i^{(2)}(0,\theta)=e^{i\theta}$. Their centroids may differ. Starting at $t_0$, set $s=\Omega(t-t_0)$, $\iota^{(j)}=I^{(j)}/(2\Gamma\ell^2)$, $Q=\triangle\widehat q$, and $P=\triangle\widehat p$. Write
\[ \nu=\left(\triangle\iota,\triangle\widehat\Theta,\operatorname{Re}\frac{\triangle\delta}{\ell},\operatorname{Im}\frac{\triangle\delta}{\ell}\right),\qquad |Z|=|Q|+|P|+|\nu|. \]
We use the sum of the absolute values of the components for $|\nu|$. We write $W(s)$ for the restarted integral \eqref{meta:history-restart}, so $W(0)$ retains the history up to physical time $t_0$. The normalized external input is $F(s)=\Omega^{-1}\|\triangle A(t_0+s/\Omega)\|_R$.

At the rhombus with $D=\ell^2$, the point-vortex linearization is
\[ Q_s=2\sqrt2Q,\qquad P_s=-2\sqrt2P,\qquad (\triangle\iota)_s=0,\qquad (\triangle\widehat\Theta)_s=\triangle\iota,\qquad (\triangle\delta)_s=0. \]
The phase equation follows from $\partial_\iota(-1/\iota)=1$ at $\iota=1$. Thus the linear operator on $\nu$ is simply
\[ \mathcal N\nu=(0,\triangle\iota,0,0),\qquad e^{s\mathcal N}=\mathrm{Id}+s\mathcal N,\qquad \|e^{s\mathcal N}\|\le C_\sigma e^{\sigma s/2}\quad(s\ge0), \]
for any fixed $0<\sigma<\sqrt2/2$. The unstable shape difference grows exponentially. The other unforced differences decay, stay constant, or grow only linearly.

\paragraph{External input and initial differences.}

We specify the bounds that Section~\ref{array:section} verifies. The direct external input includes completed sources and the slower quadrupoles' responses to those sources and to their own values at $t_0$. Evaluate this input in the reference centroid and moment equations, including the receiving-patch averages, and denote its normalized contributions by $\kappa f_q,\kappa f_p,\kappa f_\nu$. Here $0<\kappa\ll1$ measures the initial perturbation and forcing. Assume
\[ |Q(0)|+|P(0)|+|\nu(0)|+W(0)\le C\kappa, \]
\[ |f_q(s)|+|f_p(s)|+|f_\nu(s)|\le C(1+s)^m e^{\sigma s/2},\qquad \mathcal A=\frac{Q(0)}\kappa+\int_0^\infty e^{-2\sqrt2s}f_q(s)\,ds\ge A_*>0. \]
The integer $m\ge0$, $A_*$, and the input constants are fixed independently of $\kappa$. The displayed bound also gives $\mathcal A\le C$. For a stopped motion, $\mathcal A\ge A_*$ refers to an extension of $f_q$ satisfying the same bound. Its tail after $s$ has absolute value at most $C(1+s)^m e^{-(2\sqrt2-\sigma/2)s}$, which becomes a small power of $\kappa$ at the departure time proved below.

The slower quadrupoles also respond to this quadrupole's motion after $t_0$. Bound this returned input by a positive operator
\[ \begin{gathered} (\mathcal RY)(s)=r_0(s)Y(s)+\int_0^s r(s,u)Y(u)\,du,\qquad r_0,r\ge0,\\ \sup_s\left[r_0(s)+\int_0^s r(s,u)e^{-2\sqrt2(s-u)}\,du\right]\le\eta_{\mathrm{return}}. \end{gathered} \]
Its bound includes the fixed projections onto the centroid and moment equations. With $F_0(s)=C\kappa(1+s)^m e^{\sigma s/2}$, assume
\[ F\le F_0+\mathcal R(|Z|+W),\qquad \sigma_c\le\sigma/2. \]
The inherited values at $t_0$ belong to the initial state and direct input. The operator $\mathcal R$ contains only the response to subsequent motion.

Choose $T\le1$ so that the available interval contains
\[ t_0+\frac1\Omega\left[\frac1{2\sqrt2}\log\frac4{A_*\kappa}+K+1\right]\le T\le1. \]
Up to the displayed time or the first bootstrap exit, assume the geometric, external, and comparison hypotheses of Theorem~\ref{meta:reduction} on the original interval. The comparison neighborhoods are a fixed small neighborhood of the rhombus and a fixed collision-free neighborhood of the subsequent point-vortex trajectory to pair-center distance $r_{\mathrm{out}}\ell$. Proposition~\ref{prep:reference}(ii) gives
\[ |q^{(1)}|+|p^{(1)}|+|\iota^{(1)}-1|+\left|\frac{\ell^2}{D^{(1)}}-1\right|\le C\left(\frac{\mathsf S+T\mathsf S^2}{\Omega}+\Omega^{-5}\right),\qquad \frac{|\delta^{(1)}|}{\ell}\le C\mathsf T\ell T. \]
We collect the small coefficients in
\begin{equation} \label{transfer:material-error}
\eta=C\left(\frac{\mathsf S+T\mathsf S^2}{\Omega}+\Omega^{-3}+\mathsf T\ell T+\eta_{\mathrm{return}}\right).
\end{equation}
In particular, $a^2/\ell^2\le\eta$ and the reference errors above are $O(\eta)$. We choose a fixed small $q_*>0$ from the input bounds, then require $\kappa$ and $\eta\log(1/\kappa)$ to be sufficiently small. We prove the separation-time formula \eqref{transfer:memory-exit} and the associated phase estimate before any geometric or external-field failure. Section~\ref{array:section} excludes those failures for the coupled construction.

\paragraph{Difference equations.}

Subtracting the two motions defines the remainders by
\[ Q_s=2\sqrt2Q+\kappa f_q+r_q,\qquad P_s=-2\sqrt2P+\kappa f_p+r_p,\qquad \nu_s=\mathcal N\nu+\kappa f_\nu+r_\nu. \]
The coordinate change removes the nonlinear terms on the separating orbit. After the terms with small reference coefficients are included in $C\eta|Z|$, every remaining quadratic difference contains $P$ or a component of $\nu$. We derive the bounds
\begin{equation} \label{transfer:ode-remainders}
\begin{aligned}
|r_q|&\le C\left[\eta|Z|+|P|(|Q|+|P|)+|\nu||Q|+|\nu|^2+\frac{a^2}{\ell^2}W+F_0|Z|\right]+\mathcal R(|Z|+W),\\
|r_p|+|r_\nu|&\le C\left[\eta|Z|+(|P|+|\nu|)|Z|+\frac{a^2}{\ell^2}W+F_0|Z|\right]+\mathcal R(|Z|+W).
\end{aligned}
\end{equation}

The impulses $\iota^{(j)}$ stay bounded away from zero in the chosen neighborhood. By the patch impulse bounds above and $x^{-1}-y^{-1}=(y-x)/(xy)$,
\[ \left|\triangle\left(\frac{\ell^2}{D}-\frac1\iota\right)\right|\le C\frac{a^2}{\ell^2}\rho W+C\frac{a^2}{\ell^2}\rho^2|\triangle\iota|. \]
Replacing $\ell^2/D$ by $1/\iota$ therefore costs at most $C(a^2/\ell^2)W+C\eta|Z|$. This includes the difference of the shape factors multiplying the absolute reference correction $O((a^2/\ell^2)\rho^2)$. The normalized centered centroid differences are bounded by $C|Z|+C(a^2/\ell^2)\rho W$, so the same replacement is allowed in the input to $W$.

For the point-vortex system with no external field, conservation of $M_1=4\Gamma\delta$ gives $\delta_s=0$. The equations for $z_1$ and $z_2$ are even in $\delta$, since the cross-pair denominators occur in pairs $2\overline\delta\pm\overline z_1\pm\overline z_2$. The internal arithmetic-center velocity is odd in $\delta$ and bounded by $C\Omega|\delta|$. Its product with $M_1$ in \eqref{transfer:neutral-identities} is therefore quadratic. Taylor's formula bounds the normalized differences of these corrections by
\[ C\left[\frac{|\delta^{(1)}|}{\ell}\frac{|\triangle\delta|}{\ell}+\left|\frac{\triangle\delta}{\ell}\right|^2+\left(\frac{|\delta^{(1)}|}{\ell}\right)^2|Z|\right]\le C\eta|Z|+C\left|\frac{\triangle\delta}{\ell}\right|^2. \]
The moving-center term in the material impulse equation is $-2\operatorname{Re}(\dot d\overline{M_1})$. Its additional centroid-error contribution has relative bound $C\eta|Z|+C(a^2/\ell^2)W$. Common translation cancels from the centered internal equations and is estimated separately in Section~\ref{array:section}.

The reference shape variables, $\iota^{(1)}-1$, and $\delta^{(1)}/\ell$ are $O(\eta)$. The coordinate derivatives are bounded uniformly in the reference phase. For the inverse-impulse terms,
\[ \begin{aligned} \frac{\widehat q^{(2)}}{\iota^{(2)}}-\frac{\widehat q^{(1)}}{\iota^{(1)}}-Q&=\left(\frac1{\iota^{(2)}}-1\right)Q-\frac{\triangle\iota\,\widehat q^{(1)}}{\iota^{(2)}\iota^{(1)}},\\ \frac{\widehat p^{(2)}}{\iota^{(2)}}-\frac{\widehat p^{(1)}}{\iota^{(1)}}-P&=\left(\frac1{\iota^{(2)}}-1\right)P-\frac{\triangle\iota\,\widehat p^{(1)}}{\iota^{(2)}\iota^{(1)}},\\ -\frac1{\iota^{(2)}}+\frac1{\iota^{(1)}}-\triangle\iota&=\triangle\iota\left(\frac1{\iota^{(2)}\iota^{(1)}}-1\right). \end{aligned} \]
These expressions are bounded, respectively, by $C\eta|Z|+C|\nu||Q|$, $C\eta|Z|+C|\nu||P|$, and $C\eta|Z|+C|\nu|^2$. Since each $B_q,B_p,B_\Theta$ vanishes linearly and has bounded first derivatives, the difference of $\widehat pB/\iota$ is bounded by $C\eta|Z|+C|P|(|Q|+|P|)$. Together with the offset estimate, these give the quadratic terms in \eqref{transfer:ode-remainders}.

We next estimate the external terms. Pull back each external field about its own center. For a centered field $U$, put
\[ E_{\mathrm{ext},i}^{(j)}[U]=\frac1{\Gamma_i}\int\bigl[U(z-d^{(j)})-U(\zeta_i^{(j)}-d^{(j)})\bigr]\omega_i^{(j)}(z)\,dA(z). \]
Here $\Gamma_i$ is the signed circulation, and $E_{\mathrm{ext}}^{(j)}[U]$ denotes the vector of the four averages. By linearity,
\[ \triangle E_{\mathrm{ext}}=E_{\mathrm{ext}}^{(1)}[\triangle U]+(E_{\mathrm{ext}}^{(2)}-E_{\mathrm{ext}}^{(1)})[U^{(2)}]. \]
The averaging calculation in Theorem~\ref{meta:reduction}(i),(ii), applied to the fixed field $U^{(2)}$, gives
\[ \frac{|E_{\mathrm{ext}}^{(1)}[\triangle U]|}{\Omega\ell}\le C\frac{a^2}{\ell^2}\rho F,\qquad \frac{|(E_{\mathrm{ext}}^{(2)}-E_{\mathrm{ext}}^{(1)})[U^{(2)}]|}{\Omega\ell}\le C\frac{a^2}{\ell^2}\frac{\mathsf S}{\Omega}(\rho|Z|+W). \]
The second estimate uses $\|A^{(2)}\|_R\le C\mathsf S$ without separate absolute bounds for the direct and returned inputs. The same decomposition applies to the external moment integrals. Indeed, the first integral in the equation for $\dot I$, divided by $2\Gamma\ell^2\Omega$, is
\[ \operatorname{Re}\sum_{k\ge1}\frac{A_k}{\Omega}\frac{\overline{M_{k+1}}}{\Gamma\ell^{k+1}}, \]
and the equation for $\dot M_1/(\Gamma\ell\Omega)$ has the corresponding series with $M_k/(\Gamma\ell^k)$, where its $k=1$ term is $M_1/(\Gamma\ell)$. The relative moment bounds and the convergent coefficient series bound the changes at fixed $U^{(2)}$ by $C(\mathsf S/\Omega)[|Z|+(a^2/\ell^2)W]$. The external contribution to $(U(d)-\dot d)\overline{M_1}$ obeys the same bound, by the four centroid averages and $M_1=4\Gamma\delta$.

Relative to evaluation at the centroids, the receiving-patch averages enlarge the direct-input bounds by at most $1+C(a^2/\ell^2)\rho$ and contribute at most $C(a^2/\ell^2)\rho\kappa$ to $\kappa\mathcal A$. They enlarge the positive return bound by the same factor. Evaluating the direct input at the second state instead of the reference costs $CF_0|Z|$, and the analogous returned terms are absorbed into $\mathcal R(|Z|+W)$ on the fixed neighborhood. When converting the centroid errors to coordinate errors, the difference of the coordinate derivatives also multiplies the absolute reference error. Theorem~\ref{meta:reduction}(i) and the centered-coordinate bound above estimate this term by $C(a^2/\ell^2)\rho[|Z|+(a^2/\ell^2)\rho W]$, which is included in $C\eta|Z|+C(a^2/\ell^2)W$. Together with \eqref{meta:relative-error}, these estimates prove \eqref{transfer:ode-remainders}, with no term $CQ^2$ independent of $P,\nu$.

\paragraph{Growth of the unstable shape difference.}

Duhamel's formula gives
\[ \begin{aligned} Q(s)&=e^{2\sqrt2 s}\left[Q(0)+\int_0^s e^{-2\sqrt2 r}(\kappa f_q+r_q)(r)\,dr\right],\\ P(s)&=e^{-2\sqrt2 s}P(0)+\int_0^s e^{-2\sqrt2(s-r)}(\kappa f_p+r_p)(r)\,dr,\\ \nu(s)&=e^{s\mathcal N}\nu(0)+\int_0^s e^{(s-r)\mathcal N}(\kappa f_\nu+r_\nu)(r)\,dr. \end{aligned} \]
Put $E(s)=\kappa e^{2\sqrt2 s}$. Stop at $Q=\pm q_*$, another local boundary, or
\[ T_\kappa=\frac1{2\sqrt2}\log\frac{4q_*}{A_*\kappa}. \]
Then $E\le4q_*/A_*$. We prove
\[ |Q|\le CE,\qquad |P|+|\nu|\le C\eta E+C\kappa e^{\sigma s},\qquad W\le CE. \]
Choose the constants for $|Q|$, the $\kappa e^{\sigma s}$ term, $W$, and the $\eta E$ term in this order, then reduce $\eta$. The convolution estimates are
\[ \begin{aligned} \int_0^s e^{-2\sqrt2(s-r)}E(r)\,dr&\le\frac{E(s)}{4\sqrt2},& \int_0^s e^{\sigma_c(s-r)}E(r)\,dr&\le\frac{E(s)}{2\sqrt2-\sigma_c},\\ \int_0^s C_\sigma e^{\sigma(s-r)/2}E(r)\,dr&\le\frac{C_\sigma E(s)}{2\sqrt2-\sigma/2}. \end{aligned} \]
The initial stable and neutral terms and their direct-forcing integrals are bounded by $C\kappa e^{\sigma s}$, since $(1+s)^{m+1}e^{\sigma s/2}\le Ce^{\sigma s}$. Also $\kappa e^{\sigma s}\le E$ and $\mathcal RE\le\eta E$. Under the displayed bounds,
\[ \mathcal R(|Z|+W)\le C\eta E,\qquad \frac{a^2}{\ell^2}W\le C\eta E. \]

Since $|Z|\le Cq_*$, variation of constants gives
\[ |P(s)|+|\nu(s)|\le C\kappa(1+s)^{m+1}e^{\sigma s/2}+C\int_0^s e^{\sigma(s-r)/2}\left[\eta E(r)+q_*(|P(r)|+|\nu(r)|)\right]dr. \]
Here $F_0|Z|\le Cq_*F_0$ is included in the first term. Choose $q_*$ so that the rate added by Gronwall is less than $\sigma/2$. The preceding convolution estimates then give $|P|+|\nu|\le C\eta E+C\kappa e^{\sigma s}$. Equation~\eqref{meta:history-restart} gives
\[ \sup_{0\le r\le s}\frac{W(r)}{E(r)}\le C+C\eta\sup_{0\le r\le s}\frac{W(r)}{E(r)}. \]
The coefficient of $\eta E$ was fixed after the bound for $W$, so its contribution is made small by reducing $\eta$. The remaining constant depends only on the chosen bounds for $Q$, the $\kappa e^{\sigma s}$ term, and $W(0)/\kappa$. Absorbing the last term proves $W\le CE$.

Every quadratic term in $r_q$ contains $P$ or $\nu$. Hence
\[ |r_q|\le C\eta E+CE(|P|+|\nu|+F_0),\qquad \int_0^s(|P|+|\nu|+F_0)\,dr\le C\eta+C\int_0^s\kappa e^{\sigma r}\,dr. \]
Dividing the formula for $Q$ by $E$ gives
\[ \frac{Q(s)}{\kappa e^{2\sqrt2 s}}=\mathcal A+O\left(\eta(1+s)+\int_0^s\kappa e^{\sigma r}\,dr\right)+O\left((1+s)^m e^{-(2\sqrt2-\sigma/2)s}\right). \]
Since $\sigma<\sqrt2/2$, for $s\le T_\kappa$ we have
\[ \kappa e^{\sigma s}+\int_0^s\kappa e^{\sigma r}\,dr\le C\kappa^{1-\sigma/(2\sqrt2)}\le C\kappa^{3/4}. \]
The partial direct-forcing integral is uniformly bounded. For $\eta\log(1/\kappa)$ and $\kappa$ small, the last two estimates improve the bound for $Q$. Taking $q_*$ smaller keeps all other coordinates inside their neighborhood. This closes the continuity argument up to $Q=\pm q_*$ or $T_\kappa$.

Choose a fixed $s_0$ so that the forcing tail is at most $A_*/4$ for $s\ge s_0$. For small $\kappa$, $|Q|\le C\kappa e^{2\sqrt2 s_0}<q_*$ on $[0,s_0]$. For $s\ge s_0$, $Q/E\ge A_*/2$, so the value $Q=-q_*$ cannot be reached. Reaching $T_\kappa$ first would give $Q(T_\kappa)\ge2q_*$. Thus, unless a geometric or external-field bound fails first, $Q=q_*$ is reached. Its time satisfies $s_*=\frac1{2\sqrt2}\log(1/\kappa)+O(1)$. At this time the forcing tail is $O(\kappa^{3/4})$. Taking a logarithm in the formula for $Q/E$ and using the stable and neutral estimates gives
\[ s_*=\frac1{2\sqrt2}\log\frac{q_*}{\kappa\mathcal A}+O\bigl(\eta\log(1/\kappa)+\kappa^{3/4}\bigr),\qquad |P(s_*)|+|\nu(s_*)|\le C(\eta+\kappa^{3/4}). \]

On this interval $|Z|+F+W\le C\kappa e^{2\sqrt2s}$. Since the relative centroid error contains $(a^2/\ell^2)W$, its accumulated contribution to the unstable coefficient, divided by $\kappa$, is at most
\[ \frac C\kappa\int_0^{s_*}e^{-2\sqrt2s}\frac{a^2}{\ell^2}W(s)\,ds\le C\Omega^{-3}\log(1/\kappa). \]

\paragraph{Continuation to separating dipoles.}

At $s=s_*$, the second motion has $\widehat q^{(2)}=q_*+O(\eta)$, while its errors in $\widehat p$, $\iota$, and $\delta/\ell$ are $O(\eta+\kappa^{3/4})$. Moreover, $\triangle\widehat\Theta=O(\eta+\kappa^{3/4})$ and $\widehat\Theta^{(1)}=\theta^{(1)}+O(\eta)$. The normalized point-vortex orbit reaches $\widehat q=q_*$ at time $\log q_*/(2\sqrt2)$, with $\widehat\Theta=-\log q_*/(2\sqrt2)$. Rotate it by
\[ \theta^{(1)}(s_*)+\frac{\log q_*}{2\sqrt2}. \]
Its centered configuration then agrees with the second motion up to $O(\eta+\kappa^{3/4})$. The remaining point-vortex duration is $K-\log q_*/(2\sqrt2)$, a fixed number independent of $\kappa$.

Let $e(r)$ be the sum of the four centered centroid errors divided by $\ell$, where $r$ is the elapsed normalized time after $s_*$. On the fixed collision-free neighborhood, the point-vortex field is Lipschitz. The centroid error estimate of Theorem~\ref{meta:reduction}(i) therefore gives, up to the first neighborhood exit,
\[ e(r)\le C(\eta+\kappa^{3/4})+C_{\mathrm{sec}}\int_0^r\left[e(u)+\frac{\mathsf S}{\Omega}+\frac{a^2}{\ell^2}\rho\right]du. \]
Gronwall gives $e\le C_{\mathrm{sec}}(\eta+\kappa^{3/4})$ on this fixed interval. For small parameters this excludes the neighborhood exit, and Theorem~\ref{meta:reduction}(iv) continues the patch boundary bounds. At the prescribed separation, Proposition~\ref{prop:separating} gives
\[ \frac d{ds}\frac{|X_1-X_2|}{\ell}=2+O(r_{\mathrm{out}}^{-2})\ge1. \]
This positive lower bound persists for the nearby patch motion. Thus its crossing time and angle differ by $O(\eta+\kappa^{3/4})$ from those of the rotated point-vortex orbit. Adding the two durations gives
\begin{equation} \label{transfer:memory-exit}
s_{\mathrm{out}}=\frac1{2\sqrt2}\log\frac1{\kappa\mathcal A}+K+O\bigl(\eta\log(1/\kappa)+\kappa^{3/4}\bigr).
\end{equation}
The physical separation time is $t_0+s_{\mathrm{out}}/\Omega$. It lies before the assumed upper time, since the local duration is less than $T_\kappa$ and
\[ T_\kappa+K-\frac{\log q_*}{2\sqrt2}+1=\frac1{2\sqrt2}\log\frac4{A_*\kappa}+K+1. \]

The reference satisfies $\theta_s^{(1)}=-1+O(\eta)$, so
\[ \theta^{(1)}(s_{\mathrm{out}})=\theta^{(1)}(s_*)-K+\frac{\log q_*}{2\sqrt2}+O(\eta+\kappa^{3/4}). \]
The associated limiting angle $\theta_d$ is defined in Section~\ref{second:configuration} by adding $\theta_\infty-\theta(K)$ to the angle of $z_1^{(2)}$ at crossing. For the rotated orbit it equals $\theta^{(1)}(s_*)+\log q_*/(2\sqrt2)+\theta_\infty$. Since $\beta=\theta_\infty+K$, the patch motion satisfies
\[ \theta_d=\theta^{(1)}(s_{\mathrm{out}})+\beta+O\bigl(\eta\log(1/\kappa)+\kappa^{3/4}\bigr). \]
These estimates retain the actual state at $t_0$ and the actual reference phase at the separation time. All patch errors in the growth calculation vanish with the difference of the two motions.

\subsection{Separating dipoles and their second moment} \label{transfer:outgoing-material}

Shift physical time so that $t=0$ is the separation time obtained in Section~\ref{transfer:material-block}. The interval $[0,T]$ below is contained in the original interval $[0,2]$. Continue the patch boundary estimates from the original disks, retaining the actual contours at the new time zero.

Pair each positive patch with the nearby negative patch. Let $X_j$ be the center of pair $j$ and $v_j$ its positive centroid minus its negative centroid. Write
\[ d=\frac{X_1+X_2}{2},\qquad R=|X_1-X_2|,\qquad M_1=\Gamma(v_1+v_2). \]
Assume that at $t=0$ the four centroids are within $\xi\ell$ of a translated and rotated copy of the prescribed point-vortex configuration at distance $r_{\mathrm{out}}\ell$. Let $\vartheta$ be the limiting dipole angle of that copy. On a neighborhood of the quadrupole's convex hull, assume
\[ \|DU\|_{L^\infty}\le C\mathsf S,\qquad \|D^2U\|_{L^\infty}\le C\mathsf T. \]
Assume also the geometric criterion of Theorem~\ref{meta:reduction}(iv) from the original initial time whenever the distances between centroids within the quadrupole are at least $c\ell$. Section~\ref{array:section} verifies these hypotheses. We require
\[ \xi+r_{\mathrm{out}}^{-2}+\mathsf S(T+r_{\mathrm{out}}/\Omega)+\Omega\frac{a^2}{\ell^2}\rho T\ll1. \]
As before, $a/\ell$ and $(a^2/\ell^2)\rho$ are small by the choice of $\Omega$.

\paragraph{Continued separation.}

A point dipole with separation vector $v$ has velocity $-i\Gamma/(2\pi\overline v)$, and $v$ is constant. Proposition~\ref{prop:separating} and the initial error give
\[ \frac{|v_1(0)-\ell e^{i\vartheta}|+|v_2(0)+\ell e^{i\vartheta}|}{\ell}\le C(\xi+r_{\mathrm{out}}^{-2}). \]
Thus the two self-induced velocities are close to the opposite vectors $\mp i\Omega\ell e^{i\vartheta}$. We prove that this remains so and that the pair centers keep separating. Until a first possible failure, impose
\[ \frac\ell2\le|v_j(t)|\le2\ell,\qquad \frac\ell4(r_{\mathrm{out}}+\Omega t)\le R(t)\le4\ell(r_{\mathrm{out}}+\Omega t). \]

Between the two pairs, $|DK|\le C/R^2$ and $|D^2K|\le C/R^3$. Subtracting the kernels of the two opposite circulations therefore gives a pair-center velocity $O(\Gamma\ell/R^2)$ and a within-pair relative velocity $O(\Gamma\ell^2/R^3)$. These bounds are uniform for $R\ge4\ell$. Coupling the two normalized area measures in a pair bounds the difference of their external velocity averages by $C\mathsf S(|v_j|+a)\le C\mathsf S\ell$. Theorem~\ref{meta:reduction}(i) bounds the internal patch correction by $C\Omega\ell(a^2/\ell^2)\rho$. Consequently,
\[ \frac{|\dot v_j(t)|}{\ell}\le C\frac{\Gamma\ell}{R(t)^3}+C\mathsf S+C\Omega\frac{a^2}{\ell^2}\rho. \]
The first term is integrable uniformly in time, since
\[ \int_0^\infty\frac{\Gamma\ell}{[\ell(r_{\mathrm{out}}+\Omega t)/4]^3}\,dt=\frac{64\pi}{r_{\mathrm{out}}^2}. \]
It follows that
\[ \frac{|v_1(t)-\ell e^{i\vartheta}|+|v_2(t)+\ell e^{i\vartheta}|}{\ell}\le C\left[\xi+r_{\mathrm{out}}^{-2}+\mathsf S t+\Omega\frac{a^2}{\ell^2}\rho t\right]. \]
This improves the bounds for $|v_j|$ and controls both dipole directions.

For sufficiently large $r_{\mathrm{out}}$ and small $\xi$, the point-vortex asymptotics also give
\[ \operatorname{Re}\bigl(ie^{-i\vartheta}(X_1(0)-X_2(0))\bigr)\ge\frac34\ell r_{\mathrm{out}},\qquad R(0)\le\frac54\ell r_{\mathrm{out}}. \]
The self-induced velocity difference has projection at least $3\Omega\ell/2$ in the direction $-ie^{i\vartheta}$. Using the preceding bound for $v_j$ and the bootstrap upper bound for $R$,
\[ \begin{aligned} \frac d{dt}\operatorname{Re}\bigl(ie^{-i\vartheta}(X_1-X_2)\bigr)&\ge\frac32\Omega\ell-C\mathsf S R-C\frac{\Gamma\ell}{R^2}-C\Omega\ell\frac{a^2}{\ell^2}\rho\ge\Omega\ell,\\ |\dot X_1-\dot X_2|&\le3\Omega\ell. \end{aligned} \]
Indeed, after division by $\Omega\ell$, the three errors are bounded by $C\mathsf S(T+r_{\mathrm{out}}/\Omega)$, $Cr_{\mathrm{out}}^{-2}$, and $C(a^2/\ell^2)\rho$. Integrating yields
\[ \frac34\ell(r_{\mathrm{out}}+\Omega t)\le R(t)\le3\ell(r_{\mathrm{out}}+\Omega t). \]
This improves the remaining bootstrap bounds. All four patches stay separated by at least $c\ell$, so Theorem~\ref{meta:reduction}(iv) continues the boundary estimates. Hence the bounds for $v_j$ and $R$ hold on $[0,T]$.

\paragraph{The complex second moment.}

In the exact identity \eqref{meta:second-moment}, the same-patch terms are $O(\Gamma^2\rho)$. Replacing different-patch integrals by their centroid expressions costs $O(\Gamma^2a^2/\ell^2)$, since the linear Taylor terms vanish at the centroids and the second derivatives of $z/\bar z$ are bounded by $C|z|^{-2}$. Only the minimum pair distance $c\ell$ is needed for this bound.

The centroid second moment is $2\Gamma[(X_1-d)v_1+(X_2-d)v_2]$. Differentiating it using only the two self-induced dipole velocities, including their contribution to $\dot d$, gives
\[ -\frac{i\Gamma^2}{2\pi}\left(\frac1{\overline v_1}-\frac1{\overline v_2}\right)(v_1-v_2). \]
For $v_1=\ell e^{i\vartheta}$ and $v_2=-\ell e^{i\vartheta}$ this is $-(2i/\pi)\Gamma^2e^{2i\vartheta}$. The bound for $v_j$ changes it by at most
\[ C\Gamma^2\left[\xi+r_{\mathrm{out}}^{-2}+\mathsf S t+\Omega\frac{a^2}{\ell^2}\rho t\right]. \]
For the cross-pair terms, taking one difference in each opposite-sign pair gives $C\Gamma^2\ell^2/R^2\le C\Gamma^2r_{\mathrm{out}}^{-2}$. The corresponding moving-center term has the same bound, since its velocity is $O(\Gamma\ell/R^2)$ and $|M_1|\le C\Gamma\ell$.

For the external contribution, the derivative of $(z-d)(U(z)-U(d))$ is at most $C\mathsf S R$. Pairing opposite-sign material points at distance $O(\ell)$ therefore gives $C\Gamma\mathsf S R\ell$. The external part of $\dot d-U(d)$ is at most $C\mathsf S R$, so its product with $M_1$ has the same bound. Finally, the internal centroid error in the moving-center term $-2\dot dM_1$ is at most $C\Gamma\ell\cdot\Omega\ell(a^2/\ell^2)\rho$. Using $\Gamma=2\pi\Omega\ell^2$ and the upper bound for $R$, these errors satisfy
\[ \frac{C\Gamma\mathsf S R\ell}{\Gamma^2}\le C\mathsf S(t+r_{\mathrm{out}}/\Omega),\qquad \frac{C\Gamma\ell\cdot\Omega\ell(a^2/\ell^2)\rho}{\Gamma^2}\le C\frac{a^2}{\ell^2}\rho. \]
Since $\rho\le1$, summing gives
\begin{equation} \label{transfer:material-ramp}
\dot M_2(t)=-\frac{2i}{\pi}\Gamma^2e^{2i\vartheta}+O\left(\Gamma^2\left[\xi+r_{\mathrm{out}}^{-2}+\mathsf S(t+r_{\mathrm{out}}/\Omega)+\Omega\frac{a^2}{\ell^2}\rho t+\rho+\frac{a^2}{\ell^2}\right]\right).
\end{equation}
The constants in the separation and second-moment bounds are uniform for all sufficiently large $r_{\mathrm{out}}$ and independent of $\Gamma,\ell,\Omega$ and the number of other quadrupoles. In particular, the coefficient of $r_{\mathrm{out}}^{-2}$ is uniform. For a quadrupole produced by Section~\ref{transfer:material-block}, insert $\xi\le C_{\mathrm{sec}}[\eta\log(1/\kappa)+\kappa^{3/4}]$, whose constant contains the dependence on the preceding fixed trajectory segment.

\paragraph{The first and higher moments.}

For the inherited first moment, assume the same $\mathsf S,\mathsf T$ bounds also hold before separation, from the original initial time. In the construction, $M_1$ is initially zero and the quadrupole stays within $C_{\mathrm{sec}}\ell$ of its center for a physical time at most 1 before separation. Equation~\eqref{transfer:neutral-identities} gives $|\dot M_1|\le C\mathsf S|M_1|+C_{\mathrm{sec}}\Gamma\mathsf T\ell^2$, hence $|M_1(0)|\le C_{\mathrm{sec}}\Gamma\mathsf T\ell^2$ at the shifted time zero. For the initially separated top quadrupole, $M_1(0)=0$ directly.

During separation, subtract the affine external field at $d$. The derivative of the remainder is at most $C\mathsf T R$. Coupling opposite-sign patches again pairs points at distance $O(\ell)$, so
\[ \begin{aligned} |\dot M_1|&\le C\mathsf S|M_1|+C\Gamma\mathsf T R\ell,\\ |M_1(t)|&\le e^{C\mathsf S t}\left[|M_1(0)|+C\Gamma\mathsf T\ell^2(r_{\mathrm{out}}t+\Omega t^2)\right]. \end{aligned} \]
The second inequality follows by integration and the bound for $R$. It also controls the pair offset, since $M_1=4\Gamma\delta$.

Fix $d_{\mathrm{obs}}>0$ so that the observation distance is at least $d_{\mathrm{obs}}\ge2R(t)$ throughout $0\le t\le T$. The first-moment contribution to the external strain is then at most $C|M_1|/d_{\mathrm{obs}}^3$. Let $0<\Omega_r\le\Omega$ be the frequency of the receiving quadrupole. Integrating against $e^{-2\sqrt2\Omega_rt}$ and dividing by $\Gamma^2/(\Omega_r^2d_{\mathrm{obs}}^4)$ gives
\[ \frac{\Omega_r^2d_{\mathrm{obs}}}{\Gamma^2}\int_0^T e^{-2\sqrt2\Omega_rt}|M_1(t)|\,dt\le C\left[\frac{\Omega_r}{\Omega}\mathsf T d_{\mathrm{obs}}+\frac{\mathsf T d_{\mathrm{obs}}}{\Omega_r}\right]. \]
Here $C$ may depend on the fixed $r_{\mathrm{out}}$ and the preceding segment. Indeed, $e^{C\mathsf S t}$ is bounded on this physical interval, and the three integrals are bounded by $k!/(2\sqrt2\Omega_r)^{k+1}$ for $k=0,1,2$.

For $k\ge2$, all material points satisfy $|z-d|\le R$ when $r_{\mathrm{out}}$ is large. Pairing opposite signs and using $|(z-d)^k-(w-d)^k|\le kR^{k-1}|z-w|$ gives
\[ |M_k|\le Ck\Gamma\ell R^{k-1},\qquad \sum_{k\ge3}\frac{(k+1)|M_k|}{d_{\mathrm{obs}}^{k+2}}\le\frac{C\Gamma\ell R^2}{d_{\mathrm{obs}}^5}\quad\text{if }R/d_{\mathrm{obs}}\le1/2. \]
Thus \eqref{transfer:material-ramp} supplies the leading second moment, and the last two estimates bound the first and higher moments in the far field. The reference second moment contributes a bounded remainder of size $C\Gamma\ell^2$, whose absolute integral against $e^{-2\sqrt2\Omega_rt}$ is at most $C\Gamma\ell^2/\Omega_r$. All time integrals are restricted to the actual interval with the stated separation from the receiving quadrupole. If the ideal linear moment is integrated to infinity, its explicit exponentially small tail is subtracted as in Section~\ref{sec:four-point}.

\section{Proof of Proposition \ref{second:finite-solutions}} \label{array:section}

We apply the estimates of Section~\ref{transfer:material-block} successively to $n=N-1,\ldots,1$. At stage $n$, the quadrupoles $j>n$ have separated and those with $i<n$ remain near their reference rhombi. All estimates are made before the first geometric, patch boundary, or coordinate exit.

\subsection{Parameter choices and continuation} \label{array:choices}

Use the scales \eqref{eq:scales}. All constants are chosen independently of the number $N$ of quadrupoles, in the order shown below.
\begin{center}
\begin{tikzpicture}[>=Stealth,every node/.style={align=center},explain/.style={text width=1.85cm,anchor=north}]
\node (patch) at (0,0) {$A$};
\node (input) at (2.2,0) {$\varepsilon_*,M_0$};
\node (local) at (4.4,0) {$q_*,\kappa_0$};
\node (dipole) at (6.6,0) {$r_{\mathrm{out}},K$};
\node (history) at (8.8,0) {$M_{\mathrm{hist}}$};
\node (ratio) at (11,0) {$B$};
\node (size) at (13.2,0) {$\Lambda$};
\draw[->] (patch) -- (input);
\draw[->] (input) -- (local);
\draw[->] (local) -- (dipole);
\draw[->] (dipole) -- (history);
\draw[->] (history) -- (ratio);
\draw[->] (ratio) -- (size);
\node[explain] at (0,-0.4) {patch\\boundaries};
\node[explain] at (2.2,-0.4) {local\\input bounds};
\node[explain] at (4.4,-0.4) {departure\\from the rhombus};
\node[explain] at (6.6,-0.4) {dipole\\separation};
\node[explain] at (8.8,-0.4) {bound on\\earlier motion};
\node[explain] at (11,-0.4) {frequency\\ratio};
\node[explain] at (13.2,-0.4) {small\\error terms};
\end{tikzpicture}
\end{center}
An increase of $B$ or $\Lambda$ does not change any preceding choice.

\paragraph{Patch boundaries.}

Proposition~\ref{prop:separating} gives a minimum point-vortex distance $c_0\ell$ along the entire separating trajectory. Fix a perturbation margin $c_0\ell/4$ and the resulting individual-patch external bound $C_0\Omega$. Choose $A$ so that Theorem~\ref{meta:reduction} keeps the normalized boundaries within $C_A\Omega_n^{-2}$ of circles on $[0,2]$. This choice uses the minimum distance, so it is independent of how far the dipoles are followed. Write
\[ \varepsilon_n=\frac{a_n^2}{\ell_n^2}=\frac{2}{A\Omega_n^3}. \]

\paragraph{Departure from the rhombus.}

Fix the rhombus comparison neighborhood and its coefficient radius $R$ in Theorem~\ref{meta:reduction}. On $|x|\le1/2$, we have $\cos x\ge7/8$ and $|\log\cos x|\le Cx^2$. Fix a universal $C\ge1$ for these elementary estimates and the timing recurrence \eqref{array:timing-recursion}, for $B\ge4$ and total normalized errors at most $\varepsilon_*$. Fix $C_h\ge4C$, then choose $\varepsilon_*>0$ so that $CC_h\varepsilon_*\le1/4$, $(5/2)C_h\varepsilon_*\le1/2$, and $C\varepsilon_*\le1/8$. These constants are independent of the later choices $q_*,r_{\mathrm{out}},K$. All later constants multiplying forcing, phase, or delay errors are included before those errors are required to be at most $\varepsilon_*$. Fix also $A_*=1/4$ and $0<\sigma<\sqrt2/2$. In time $s=\Omega_n(t-t_0)$, the leading direct input is bounded by $C_{\mathrm{dir}}\kappa_ns$, because
\[ \frac{\Gamma_{n+1}^2s}{\Omega_n^2|d_n-d_{n+1}|^4}\asymp\kappa_ns. \]
The constant $C_{\mathrm{dir}}\ge1$ includes the projections onto the centroid and moment equations. It also covers velocity divided by $d_n$ and Hessian multiplied by $d_n$, since $|d_n-d_{n+1}|=d_n/2$. It is independent of $q_*$ and $r_{\mathrm{out}}$. Set
\[ M_0=4C_{\mathrm{dir}}\max\left\{1,\int_0^\infty e^{-(2\sqrt2-\sigma/2)s}(1+s)^3\,ds\right\}. \]
Fix the input bounds of Section~\ref{transfer:material-block} by
\[ \begin{gathered} |Z(0)|+W(0)\le2M_0\kappa,\qquad F_0(s)\le2M_0\kappa(1+s)^3e^{\sigma s/2},\\ \int_0^\infty e^{-2\sqrt2s}F_0(s)\,ds\le2M_0\kappa,\qquad \mathcal A\ge A_*. \end{gathered} \]
Here $F_0$ bounds the direct input and the sum of its projected terms $\kappa|f_q|,\kappa|f_p|,\kappa|f_\nu|$. Thus $\mathcal A\le4M_0$. Include the fixed coordinate constants for the first-moment and center estimates of Section~\ref{array:weights-section} in the same bounds. Section~\ref{transfer:material-block} now fixes $q_*>0$ and then $\kappa_0>0$ for departure at $Q=q_*$.

\paragraph{Dipole separation and earlier motion.}

Choose a point sufficiently far along the separating orbit that $Cr_{\mathrm{out}}^{-2}\le\varepsilon_*/4$. Here $C$ is the uniform constant in \eqref{transfer:material-ramp}, enlarged only by the fixed projections and elementary timing estimates. In particular, it does not depend on the segment chosen by $r_{\mathrm{out}}$. This fixes $r_{\mathrm{out}}$, $K$, and the trajectory segment from $Q=q_*$ to that distance. Fix a collision-free comparison neighborhood of this segment and its coefficient radius. The symbol $R$ denotes the radius for the comparison currently used. Enlarging the later neighborhood does not change the local input bounds above. Continue $W$ from its incoming value by \eqref{meta:history-restart}.

Fix also $0<c<2\sqrt2$ in the earlier-motion bound \eqref{array:prehistory}. Let $C_{\mathrm{sec}}$ bound the estimates on this segment and in Section~\ref{transfer:outgoing-material}. Let $C_{\mathrm{out}}\ge1$ include the bounds for growth near the rhombus and along this segment, and the bound for the prescribed scale-$N$ initial configuration. These constants depend only on the choices already made. Fix
\begin{equation} \label{array:history-constant}
M_{\mathrm{hist}}=2C_{\mathrm{out}}A_*^{-1}M_0.
\end{equation}
Section~\ref{array:phases-section} verifies that this bounds the earlier motion at each completed scale. The factor two leaves room for the interactions between quadrupoles. In particular, $M_{\mathrm{hist}}$ is fixed before $B$ and $\Lambda$.

The error $Cr_{\mathrm{out}}^{-2}$ changes the forcing coefficient $\mathcal A$, which is retained in the logarithmic delay. The error from the preceding trajectory segment is $C_{\mathrm{sec}}[\eta\log(1/\kappa)+\kappa^{3/4}]$. It tends to zero with $\Lambda^{-1}$ after the segment has been fixed. Thus the later dependence on $r_{\mathrm{out}}$ does not change its already chosen tolerance.

\paragraph{The frequency ratio and overall frequency.}

Choose $B>4$ sufficiently large, depending on $\alpha$, for the frequency and source comparisons in Sections~\ref{array:weights-section}--\ref{array:sources-section} and~\ref{array:induction-section}, and require
\[ \begin{gathered} B^\alpha>2,\qquad B^{1-\alpha}\ge2,\\ C_{\mathrm{sec}}(1+M_{\mathrm{hist}})(B^{-1}+B^{-2})\le\tfrac14\min\{M_0,\varepsilon_*\}. \end{gathered} \]
The ratio $\Omega_{n-1}/\Omega_n=B^{-1}$ controls the slower quadrupoles. The factors $B^{-1}$ and $B^{-2}$ also control the constant second moment and the earlier motion of the nearest source, respectively. All constants in these requirements have been fixed independently of $B$.

With $B$ fixed, choose $\Lambda$ large. Formula~\eqref{second:seed-delay} gives
\[ \log(1/\kappa_n)=4\Omega_{n+1}^\alpha+O_{\alpha,B}(n+1),\qquad \sup_n\kappa_n\longrightarrow0,\qquad \sum_n\Delta_n\longrightarrow0. \]
The spatial estimates below make $\mathsf S,\mathsf T$ and the interaction bounds tend to zero. The coefficient estimates in Sections~\ref{array:weights-section}--\ref{array:phases-section} give $\sup_n[\eta_n\log(1/\kappa_n)+\kappa_n^{3/4}]\to0$, where $\eta_n$ is \eqref{transfer:material-error}. Choose $\Lambda$ so that these errors, including their fixed factors $C_{\mathrm{sec}},M_{\mathrm{hist}}$, meet the stated smallness bounds. Require also
\[ \sup_n\kappa_n\le\kappa_0,\qquad C(\Omega_1^{-1}+\Omega_1^{-3})\le\sigma/2,\qquad \sum_n\Delta_n<1/4. \]
The middle inequality ensures $\sigma_c\le\sigma/2$ in the local integral estimate. Increase $\Lambda$ once more so that the total upper bound for the separation times in Section~\ref{transfer:material-block} satisfies
\[ \sum_{n\ge1}\frac1{\Omega_n}\left[\frac1{2\sqrt2}\log\frac4{A_*\kappa_n}+K+1\right]=\sum_{n\ge1}\Delta_n+\frac{(2\sqrt2)^{-1}\log(4/A_*)+1}{\Lambda(B-1)}<\frac12. \]
The choices are now fixed as $N\to\infty$. For each $N$, construct the reference on $[0,2]$ and the second initial configuration as in Sections~\ref{sec:first-solution}--\ref{sec:second-solution}. Compare them on $[0,1]$.

\paragraph{Spatial bounds and continuation.}

For fixed $k,m\ge0$ and $r>0$,
\[ \sum_{n\ge1}\Omega_n^k d_n^{-m}\ell_n^r\longrightarrow0\qquad(\Lambda\to\infty,\ B\text{ fixed}). \]
Indeed, the ratio of consecutive terms is $B^k2^m e^{-r(B^\alpha-1)\Omega_n^\alpha}\le1/2$ for large $\Lambda$. The sum is then at most $2\Omega_1^kd_1^{-m}\ell_1^r$, which tends to zero. The corresponding supremum has the same limit.

Stop at the first patch boundary exit in \eqref{meta:boundary-chart} or failure of
\[ |d_i^{(j)}(t)-d_k^{(j)}(t)|\ge\tfrac12|d_i-d_k|\qquad(i\ne k,\ j=1,2). \]
The geometric bootstrap bounds near the rhombus, along the fixed trajectory segment, and during dipole separation give diameter at most $C_{\mathrm{sec}}\ell_n(1+\Omega_nt)$ and minimum centroid distance at least $c_0\ell_n/2$. Sections~\ref{transfer:material-block}--\ref{transfer:outgoing-material} improve these bounds once their inputs are verified. For the latter section, the initial error is zero at scale $N$ and small by Section~\ref{transfer:material-block} at subsequent scales. Its other smallness requirements are bounded by
\[ r_{\mathrm{out}}^{-2}+\mathsf S(1+r_{\mathrm{out}}/\Omega_n)+C_A\Omega_n^{-4}. \]
These are small by the choices of $r_{\mathrm{out}}$ and $\Lambda$.

The series estimate gives
\[ \sup_i\frac{C_{\mathrm{sec}}\ell_i(1+\Omega_i)+2R\ell_i+2a_i}{d_i}\longrightarrow0. \]
Thus the comparison disks $B(d_i^{(j)}(t),2R\ell_i)$ and the separating quadrupoles' convex hulls stay away from every other quadrupole. The external fields are analytic there, with
\[ \sup|DU_{\mathrm{ext},i}^{(j)}|\le C\sum_{k\ne i}\frac{\Gamma_k}{|d_i-d_k|^2}\le C\mathsf S,\qquad \sup|D^2U_{\mathrm{ext},i}^{(j)}|\le C\sum_{k\ne i}\frac{\Gamma_k}{|d_i-d_k|^3}\le C\mathsf T. \]

We next control the quadrupole centers. For separating dipoles use $X_1,X_2,v_1,v_2$ from Section~\ref{transfer:outgoing-material}, so $M_1=\Gamma(v_1+v_2)=4\Gamma\delta$. Summing the four point-vortex velocities gives
\[ \dot d_{\mathrm{point,int}}=-\frac\Gamma2\bigl[K(v_1)+K(v_2)+K(X_1-X_2+2\delta)+K(X_2-X_1+2\delta)\bigr]. \]
Since $v_1+v_2=4\delta$, oddness of $K$ and $|DK(x)|\le C|x|^{-2}$ imply
\[ |\dot d_{\mathrm{point,int}}|\le C\Gamma|\delta|\bigl(\ell^{-2}+|X_1-X_2|^{-2}\bigr)\le C\Omega|M_1|/\Gamma. \]
The internal patch correction is at most $C\Omega\ell\varepsilon\Omega^{-2}$. The linear external Taylor term cancels when averaged over the four patches about $d$, so
\[ |\dot d-U(d)|\le C\Omega|M_1|/\Gamma+C\mathsf T(|X_1-X_2|^2+a^2)+C\Omega\ell\varepsilon\Omega^{-2}. \]
On the original interval $[0,1]$, Section~\ref{transfer:outgoing-material} gives $|M_1|\le C_{\mathrm{sec}}\Gamma\mathsf T\Omega\ell^2$ and $|X_1-X_2|\le C_{\mathrm{sec}}\Omega\ell$. Near the rhombus and along the fixed segment, the same calculation uses diameter $C_{\mathrm{sec}}\ell$ and gives a smaller bound. Hence, in all three parts of the motion,
\[ |\dot d_n-U_{\mathrm{ext},n}(d_n(t))|\le r_n,\qquad r_n=C_{\mathrm{sec}}\bigl(\mathsf T\Omega_n^2\ell_n^2+\ell_n\Omega_n^{-4}\bigr),\qquad \sup_n r_n/d_n\longrightarrow0. \]

The cutoff field $V$ from Lemma~\ref{prep:moment-bounds} satisfies $V(d_n(t))=U_{\mathrm{ext},n}(d_n(t))$ and $\|DV\|_{L^\infty}\le C\mathsf S$. Thus
\[ |\dot d_i(t)-\dot d_j(t)|\le C\mathsf S|d_i(t)-d_j(t)|+r_i+r_j. \]
Using $d_i+d_j\le3|d_i-d_j|$ for the fixed initial locations and integrating gives
\[ \frac{|d_i(t)-d_j(t)-(d_i-d_j)|}{|d_i-d_j|}\le e^{C\mathsf S t}-1+3t e^{C\mathsf S t}\sup_k\frac{r_k}{d_k}\longrightarrow0\qquad(0\le t\le1). \]
For large $\Lambda$ this is less than $1/4$, improving the separation assumption and controlling the change in the arguments of all interquadrupole vectors.

For each individual patch, the other three patches in its quadrupole contribute at most $C\Gamma_n/\ell_n^2=C\Omega_n$ to the weighted circulation integral in Theorem~\ref{meta:reduction}(iv). The other quadrupoles contribute at most $C\mathsf S$. The normalized external coefficient and material-derivative bounds are therefore $C\Omega_n$ and $C\Omega_n^2$. The theorem improves the patch boundary bounds from the original time zero. Consequently, while the centroid bounds hold, neither a patch boundary exit nor loss of separation between quadrupoles can occur first.

\subsection{Relative estimates and slower quadrupoles} \label{array:weights-section}

The linear shape rates are $2\sqrt2\Omega_i$ and $-2\sqrt2\Omega_i$, while the exponential rate in the physical-time integral $W_i$ is bounded independently of $i$. We use the centroid decomposition of Theorem~\ref{meta:reduction} and take second minus first as in Lemma~\ref{core:relative}. Thus $\triangle d_i(t)=d_i^{(2)}(t)-d_i^{(1)}(t)$, while $d_i=2^{-i}$ is the fixed initial location. In addition to the shape, angular-impulse, and corrected-phase differences and the integral $W_i$, use
\[ \widehat M_{1,i}=\frac{d_i\triangle M_{1,i}}{\Gamma_i\ell_i^2},\qquad \widehat d_i=\frac{\triangle d_i}{d_i}. \]
In the coordinates of Section~\ref{transfer:material-section}, set $Q_i=\triangle\widehat q_i$, $P_i=\triangle\widehat p_i$, and
\[ Y_i=|Q_i|+|P_i|+\frac{|\triangle I_i|}{2\Gamma_i\ell_i^2}+|\triangle\widehat\Theta_i|+W_i+|\widehat M_{1,i}|+|\widehat d_i|. \]
Write $Z_i=Y_i-W_i$ for the sum of the finite-coordinate differences. Both motions began at the original time zero from disks with the same radii, signed circulations, and standard normalized material labels $F(0,\theta)=e^{i\theta}$. Equation~\eqref{transfer:action-correction} bounds the difference between $D_i/\ell_i^2$ and $I_i/(2\Gamma_i\ell_i^2)$, with relative error $C\varepsilon_i\Omega_i^{-2}W_i$. Enlarge the centroid input in \eqref{transfer:core-memory} to include all these coordinates, as justified after Theorem~\ref{meta:reduction}. Each $W_i$ starts from $W_i(0)=0$ at the original time zero and retains its actual value at every later starting time and chart change. At $i<N$ the initial centroid configurations also agree, so $Y_i(0)=0$. Since $\triangle\delta_i/\ell_i=(\ell_i/(4d_i))\widehat M_{1,i}$, the stronger first-moment normalization decreases its coefficient in the even centroid equations and the integral input.

At distance comparable to $d_i$, this normalization gives the same scale for the first- and second-moment strains.
\[ \frac{|\triangle M_{1,i}|}{d_i^3}=\frac{\Gamma_i\ell_i^2}{d_i^4}|\widehat M_{1,i}|,\qquad \frac{|\triangle M_{2,i}|}{d_i^4}\le C\frac{\Gamma_i\ell_i^2}{d_i^4}Y_i. \]
The translation variable $\widehat d_i$ is normalized by the interquadrupole distance.

Choose the relative rhombus neighborhood to be a fixed small ball for $Y_i$. Lift the relative phase continuously from the identical initial data. The absolute reference phase may rotate freely. Stop at the first exit of any component of $Y_i$, or at loss of patch confinement or separation. A small bound on $Y_i$ therefore excludes every rhombus-neighborhood exit.

On the fixed compact segment after $Q_i=q_*$, retain the four material labels and their representation $d+\delta\pm z_1,d-\delta\pm z_2$, and use instead
\[ Y_i=\frac{|\triangle z_{1,i}|+|\triangle z_{2,i}|}{\ell_i}+W_i+|\widehat M_{1,i}|+|\widehat d_i|. \]
The two norms are uniformly equivalent near $Q_i=q_*$. Indeed, the four even centroid coordinates give local analytic coordinates for shape, angular impulse, and phase, with $|\triangle(I_i/(2\Gamma_i\ell_i^2)-D_i/\ell_i^2)|\le C\varepsilon_i\Omega_i^{-2}W_i$. The reference stays near a rotating rhombus, and the second motion stays near the fixed collision-free compact segment. The latter norm is defined throughout this passage, satisfies the same multipole bounds with constants depending on the segment, and continues the integral comparison of Theorem~\ref{meta:reduction}. Use the corresponding norm on each interval when estimating earlier motion.

Near $Q_i=q_*$, use the local relative estimates. On the rest of the outgoing segment, its shape has a fixed positive distance from the compact set of rotating rhombi, so $(|\triangle z_{1,i}|+|\triangle z_{2,i}|)/\ell_i\ge c_{\mathrm{sec}}>0$. Each separated normalized point field and individual external coefficient list is bounded by $C_{\mathrm{sec}}$. Their difference is therefore bounded by $2C_{\mathrm{sec}}c_{\mathrm{sec}}^{-1}(|\triangle z_{1,i}|+|\triangle z_{2,i}|)/\ell_i$. This proves the relative coefficient estimate without requiring a collision-free straight segment between the two configurations.

Write the external expansion about each quadrupole's own center as
\[ U_{\mathrm{ext},i}^{(j)}(d_i^{(j)}+z)-U_{\mathrm{ext},i}^{(j)}(d_i^{(j)})=S_i^{(j)}\bar z+\sum_{k\ge2}a_{k,i}^{(j)}\bar z^k. \]
Subtracting the exact first-moment equations gives
\[ \triangle \dot M_{1,i}=S_i^{(1)}\overline{\triangle M_{1,i}}+\triangle S_i\overline{M_{1,i}^{(2)}}+\sum_{k\ge2}\left(a_{k,i}^{(1)}\overline{\triangle M_{k,i}}+\triangle a_{k,i}\overline{M_{k,i}^{(2)}}\right). \]
Fix the time-independent bound
\[ \mathsf T_i=C\sum_{j\ne i}\frac{\Gamma_j}{|d_i-d_j|^3}, \]
where the geometric constant $C$ is large enough to bound both external Hessians on the comparison disks throughout the stopped interval. It also bounds the reference Hessian on its original interval. By the first-moment equation, $\delta_i^{(1)}(0)=0$ gives $|\delta_i^{(1)}(t)|\le C\mathsf T_i\ell_i^2t e^{C\mathsf S t}\le C\mathsf T_i\ell_i^2$ on $[0,1]$. Use the radius $R$ already fixed for the comparison neighborhood. In the rhombus neighborhood, $|\triangle M_{k,i}|\le CR^k\Gamma_i\ell_i^kY_i$. Also $d_i|M_{1,i}^{(2)}|/(\Gamma_i\ell_i^2)\le C$, by the reference bound and the bound on $\widehat M_{1,i}$. Hence
\[ |\dot{\widehat M}_{1,i}|\le C\mathsf{S}|\widehat M_{1,i}|+Cd_i\mathsf{T}_iY_i+C\left(|\triangle S_i|+d_i\sum_{k\ge2}R^k\ell_i^{k-2}|\triangle a_{k,i}|\right). \]
The last term is bounded by the external strain difference plus $d_i$ times the analytic external Hessian difference. Since $|d_i-d_j|\ge\frac12\max\{d_i,d_j\}$,
\[ d_i\mathsf T_i\le C\sum_{j<i}\frac{\Gamma_jd_i}{d_j^3}+C\sum_{j>i}\frac{\Gamma_j}{d_i^2}\le C\sum_{j\ne i}\frac{\Gamma_j}{d_j^2}\le C\mathsf S. \]
Thus the coefficient of a local shape or patch difference is at most $C\mathsf S$, and every forcing term is an actual external-field difference.

For the internal point field, direct summation of the four centroid equations gives
\[ \dot d_{\mathrm{point,int}}=-\frac \Gamma2\sum_{\epsilon_1,\epsilon_2\in\{-1,1\}}K(2\delta+\epsilon_1z_1+\epsilon_2z_2). \]
The expression is odd in $\delta$. On the fixed separated normalized neighborhood, its $\delta$ derivative is bounded by $C\Omega$ and its shape derivative by $C\Omega|\delta|/\ell$. Since $|\delta_i^{(1)}|\le C\mathsf T_i\ell_i^2$ and $|\triangle\delta_i|=\ell_i^2|\widehat M_{1,i}|/(4d_i)$, the mean-value formula gives
\[ \frac{|\triangle \dot d_{\mathrm{point,int}}|}{d_i}\le C\frac{\Gamma_i}{d_i^2}|\widehat M_{1,i}|+C\frac{\Omega_i\mathsf{T}_i\ell_i^2}{d_i}Y_i. \]
The offset coefficient is $C\Gamma_i/d_i^2$. The shape coefficient contains the small reference offset and is small because $\ell_i/d_i$ decays superexponentially. The relative patch contribution to the arithmetic-center equation is bounded by
\[ C\Omega_i\frac{\ell_i}{d_i}\varepsilon_i\bigl(W_i+\Omega_i^{-2}|Z_i|\bigr)+C\frac{\ell_i}{d_i}\varepsilon_i\Omega_i^{-2}\|\triangle A_i\|_R. \]
Here $Z_i=Y_i-W_i$, and $\triangle A_i$ is the actual external coefficient difference from Theorem~\ref{meta:reduction}(ii). The last term retains its instantaneous external input, with uniformly small factor $(\ell_i/d_i)\varepsilon_i\Omega_i^{-2}$. External averaging adds $C\mathsf T_i\ell_i^2/d_i$ times a local relative variable and the corresponding external-input difference. Since the even point equations are even in $\delta_i$, their offset contribution is a small reference factor times $\triangle\delta_i/\ell_i$, plus a quadratic difference.

Indeed, Taylor's formula gives
\[ C\frac{|\delta_i^{(1)}|}{\ell_i}\frac{|\triangle\delta_i|}{\ell_i}+C\left|\frac{\triangle\delta_i}{\ell_i}\right|^2\le C\mathsf{T}_i\ell_i\frac{|\triangle\delta_i|}{\ell_i}+C\left|\frac{\triangle\delta_i}{\ell_i}\right|^2. \]
The first term has a small coefficient, and the second is quadratic. Theorem~\ref{meta:reduction}(ii),(iii) bounds the remaining moment and centroid errors by the same finite-coordinate differences and $W_i$.

The external input norm at quadrupole $i$ is the sum of velocity divided by $d_i$, strain, and $d_i$ times Hessian. A shape, angular-impulse, phase, or patch difference at quadrupole $j$ changes its centered second moment by at most $C\Gamma_j\ell_j^2Y_j$. Its first-moment difference is at most $C\Gamma_j\ell_j^2|\widehat M_{1,j}|/d_j$. Translating its center differentiates the existing first and higher moment fields, since the total circulation remains zero. Since the reference satisfies $|M_{1,j}|\le C\Gamma_j\mathsf T\ell_j^2$, these translation terms satisfy the bounds below as well.

Use the weighted norm $\max_i d_i^4Y_i$, so the coefficient from $j$ to $i$ acquires the factor $(d_i/d_j)^4$.

For the sum of velocity/$d_i$, strain, and $d_i$ times Hessian, the second- and first-moment coefficients are bounded respectively by
\[ \begin{gathered} C\Gamma_j\ell_j^2\left(\frac1{d_i|d_i-d_j|^3}+\frac1{|d_i-d_j|^4}+\frac{d_i}{|d_i-d_j|^5}\right),\\ C\frac{\Gamma_j\ell_j^2}{d_j}\left(\frac1{d_i|d_i-d_j|^2}+\frac1{|d_i-d_j|^3}+\frac{d_i}{|d_i-d_j|^4}\right). \end{gathered} \]
Let $A_{ij}(t)\ge0$ be the sum of these coefficient bounds, including the translation terms below. For $j>i$, multiplication by $(d_i/d_j)^4$ bounds the two expressions by $C\Gamma_j\ell_j^2/d_j^4$ and $C\Gamma_j\ell_j^2d_i/d_j^5$, respectively. For $j<i$, both are bounded by $C\Gamma_j\ell_j^2/d_j^4$, since $d_i/d_j\le1$. Thus
\begin{equation} \label{array:row-bound}
\sup_{t,i}\sum_{j\ne i}A_{ij}(t)\left(\frac{d_i}{d_j}\right)^4\le \Xi,\qquad \Xi=C\sum_{j\ge1}\frac{\Omega_j\ell_j^4}{d_j^5}+C\mathsf{S}.
\end{equation}
For the same three components, translating the second-moment field gives a coefficient at most
\[ C\Gamma_j\ell_j^2d_j\left(\frac1{d_i|d_i-d_j|^4}+\frac1{|d_i-d_j|^5}+\frac{d_i}{|d_i-d_j|^6}\right). \]
Translating the reference first-moment field has coefficient at most
\[ C\mathsf{T}\Gamma_j\ell_j^2d_j\left(\frac1{d_i|d_i-d_j|^3}+\frac1{|d_i-d_j|^4}+\frac{d_i}{|d_i-d_j|^5}\right). \]
Multiplication by $(d_i/d_j)^4$ and the same split into $j>i$ and $j<i$ give the majorant in \eqref{array:row-bound}, using $\mathsf T\le1$ and $d_j\le1$. The constants are independent of the truncation. Include the local impulse, phase, first-moment, and translation terms in the local evolution estimate, whose growth is bounded by $C\exp(C\Omega_i(t-s))$ in the rhombus neighborhood.

Equation~\eqref{meta:history-restart} uses the external coefficient difference without its time derivative. In physical time, the centroid input to $W_i$ has coefficient $C\Omega_i$, the $W_i$ term in the centroid equation has coefficient $C\Omega_i\varepsilon_i$, and the exponential rate in the $W_i$ integral is uniformly bounded. Include these terms in the same local growth and weighted interaction estimates.

For the target, adjoin $\widehat M_{1,i},\widehat d_i$ to the angular-impulse and corrected-phase variables in Section~\ref{transfer:memory-saddle}. In recipient time their coefficients satisfy
\[ \frac{d_i\mathsf{T}_i}{\Omega_i}\le\frac{C\mathsf{S}}{\Omega_i},\qquad \frac{\Gamma_i}{\Omega_id_i^2}=\frac{2\pi \ell_i^2}{d_i^2}\le C\Omega_i^{-3},\qquad \frac{\ell_i}{d_i}\varepsilon_i\le C\varepsilon_i. \]
The remaining local center-error coefficients are smaller. After the displayed small coefficients are included in the remainders, the additional coordinates have zero unperturbed derivatives. Translation changes the shape equation by at most $d_i\mathsf T_i/\Omega_i$, since internal velocities depend only on relative positions. The reference offset also satisfies $|\delta_i^{(1)}|/\ell_i\le C\mathsf T_i\ell_i\le C\Omega_i^{-3}$. Thus the enlarged system satisfies the remainder inequalities and reference-offset bound of Section~\ref{transfer:material-block}, with its stated $\eta$.

\paragraph*{Slower quadrupoles.}

Fix a target $n$. The slower quadrupoles are $i<n$, which have smaller frequencies and remain in their local rhombus neighborhoods. Stop when the target reaches its prescribed dipole separation or a slower quadrupole leaves its neighborhood. All estimates below hold on this stopped interval.

For $i<n$, use the bound $e^{C_f\Omega_it}$ and $\Omega_i\le\Omega_n/B$. Choose $B$ so that $C_f\Omega_i<2\sqrt2\Omega_n$. The target estimate will retain the exact rate $2\sqrt2\Omega_n$, since a fixed change in that rate over the logarithmic departure time would change the effective coefficient by a power of $\kappa_n$.

Recall $Z_i=Y_i-W_i$. The change of receiving center in the velocity, strain, and Hessian components has coefficients bounded by $C\mathsf S$, $Cd_i\mathsf T_i$, and $Cd_i^2\sup|D^3U_{\mathrm{ext},i}|$, respectively. Kernel differentiation gives
\[ d_i^2\sup|D^3U_{\mathrm{ext},i}|\le C\sum_{j\ne i}\frac{\Gamma_jd_i^2}{|d_i-d_j|^4}\le C\mathsf S. \]
Thus the terms caused by changing the receiving state are bounded by $C(\mathsf S+d_i\mathsf T_i)Y_i$. Include them in the local centroid and integral coefficients. Let $F_i^{\mathrm{ext}}$ be the remaining physical-time external input norm, due to changes in the other quadrupoles. Then
\[ \|\triangle A_i\|_R\le C F_i^{\mathrm{ext}}+C\mathsf S Y_i. \]
Since $(\ell_i/d_i)\varepsilon_i\Omega_i^{-2}\le1$, the instantaneous patch term above is included in $CF_i^{\mathrm{ext}}$ and the small local coefficients. Theorem~\ref{meta:reduction} and \eqref{meta:history-restart} give, on every subinterval $[r,t]$ before the stop,
\[ \begin{aligned} Z_i(t)&\le Ce^{C\Omega_i(t-r)}Z_i(r)+C\int_r^t e^{C\Omega_i(t-s)}\bigl(\Omega_i\varepsilon_iW_i(s)+F_i^{\mathrm{ext}}(s)\bigr)\,ds,\\ W_i(t)&\le Ce^{\mu(t-r)}W_i(r)+C\int_r^t e^{\mu(t-s)}\bigl(\Omega_iZ_i(s)+F_i^{\mathrm{ext}}(s)\bigr)\,ds,\qquad \mu\le C\Omega_i. \end{aligned} \]
The constants here may depend on $A$ and the fixed coordinate neighborhoods, but not on $i,N$. Equation~\eqref{meta:history-restart} retains $W_i(r)$. Since $\varepsilon_i\le1$, addition and Gronwall give fixed constants $C_G,C_f$, chosen before $B$, such that
\begin{equation} \label{array:forced-propagator}
Y_i(t)\le C_Ge^{C_f\Omega_i(t-r)}Y_i(r)+C_G\int_r^t e^{C_f\Omega_i(t-s)}F_i^{\mathrm{ext}}(s)\,ds.
\end{equation}
This estimate includes the coefficient $C\Omega_i$ of the centroid input to $W_i$.

\begin{lemma} \label{array:follower-convolution}
Let $n\ge2$ and $0\le t_0\le T\le1$, where $T$ is a stopped endpoint as above. For every $i<n$, suppose the input in \eqref{array:forced-propagator} is bounded on $[t_0,T]$ by
\[ F_i^{\mathrm{ext}}(t)\le f_i(t)+A_{in}(t)Y_n(t)+\sum_{j<n,\ j\ne i}A_{ij}(t)Y_j(t),\qquad f_i(t)\ge0. \]
Here $f_i$ contains the forcing from completed sources, and $A_{in}Y_n$ bounds the current target's contribution. Set
\[ V_0=\max_{i<n}\left(\frac{d_i}{d_n}\right)^4Y_i(t_0),\qquad f(t)=\max_{i<n}\left(\frac{d_i}{d_n}\right)^4f_i(t),\qquad \beta_n=C_f\Omega_{n-1}+C_G\Xi. \]
The forcing returned by the slower quadrupoles to the target is bounded by
\[ \begin{aligned} \sum_{i<n}A_{ni}(t)Y_i(t)\le{}&C_G\Xi e^{\beta_n(t-t_0)}V_0+C_G\Xi\int_{t_0}^t e^{\beta_n(t-r)}f(r)\,dr\\ &+C_G\Xi^2\int_{t_0}^t e^{\beta_n(t-r)}Y_n(r)\,dr. \end{aligned} \]
The terms on the right come respectively from the slower-quadrupole states at $t_0$, new source inputs, and the target's contribution after $t_0$. For every $\gamma>\beta_n$, the last term defines a positive integral operator $\mathcal K_n$ on $[t_0,T]$ satisfying
\[ \sup_{t_0\le t\le T}\int_{t_0}^t C_G\Xi^2e^{-(\gamma-\beta_n)(t-r)}\,dr\le\frac{C_G\Xi^2}{\gamma-\beta_n}. \]
The exponentially weighted integrals of the first and second terms are bounded by
\[ \frac{C_G\Xi}{\gamma-\beta_n}V_0,\qquad \frac{C_G\Xi}{\gamma-\beta_n}\int_{t_0}^T e^{-\gamma(r-t_0)}f(r)\,dr, \]
respectively. These bounds use no continuation beyond the stopping time $T$.
\end{lemma}

\begin{proof}
Set
\[ V(t)=\max_{i<n}(d_i/d_n)^4Y_i(t). \]
By $\Omega_i\le\Omega_{n-1}$, \eqref{array:forced-propagator}, and \eqref{array:row-bound},
\[ V(t)\le C_Ge^{C_f\Omega_{n-1}(t-t_0)}V_0+C_G\int_{t_0}^t e^{C_f\Omega_{n-1}(t-r)}\bigl(f(r)+\Xi Y_n(r)+\Xi V(r)\bigr)\,dr. \]
Multiply by $e^{-C_f\Omega_{n-1}(t-t_0)}$ and apply Gronwall to obtain
\[ V(t)\le C_Ge^{\beta_n(t-t_0)}V_0+C_G\int_{t_0}^t e^{\beta_n(t-r)}\bigl(f(r)+\Xi Y_n(r)\bigr)\,dr. \]
Equation~\eqref{array:row-bound} gives $\sum_{i<n}A_{ni}Y_i\le\Xi V$, proving the three terms.

For an input $e^{\gamma(t-t_0)}$, $\gamma>\beta_n$, division of the returned term by the input gives
\[ C_G\Xi^2\int_{t_0}^t e^{-(\gamma-\beta_n)(t-r)}\,dr\le\frac{C_G\Xi^2}{\gamma-\beta_n}. \]
This proves the operator estimate. Integration of the initial-state term against $e^{-\gamma(t-t_0)}$ gives at most $C_G\Xi V_0/(\gamma-\beta_n)$. For the source term, interchanging the nonnegative integrals gives
\[ \int_{t_0}^T e^{-\gamma(t-t_0)}\int_{t_0}^t e^{\beta_n(t-r)}f(r)\,dr\,dt=\int_{t_0}^T e^{-\gamma(r-t_0)}f(r)\int_r^T e^{-(\gamma-\beta_n)(t-r)}dt\,dr. \]
Bound the inner integral by $(\gamma-\beta_n)^{-1}$. This proves the remaining estimate on the stopped interval.
\end{proof}

We also need an exact signed decomposition. If a vector $g$ satisfies $|g|\le D+H$, where $D,H\ge0$, then
\[ g=\frac{Dg}{D+H}+\frac{Hg}{D+H},\qquad \left|\frac{Dg}{D+H}\right|\le D,\qquad \left|\frac{Hg}{D+H}\right|\le H. \]
Set both terms to zero where $D+H=0$, since then $g=0$. Apply this observation to the actual slower-quadrupole input, including its analytic coefficient list and all reference centroid and moment projections, with the sum of their norms. Include the fixed projection and receiving-patch averaging constants in the preceding majorants. Take $D$ to be the inherited-state and completed-source bound in Lemma~\ref{array:follower-convolution}, and $H$ to be its target-return bound. The first term defines the signed direct input and the second is included in the remainders of \eqref{transfer:ode-remainders}. A common scalar multiplies the coefficient list and all its projections, so the split respects those linear maps and bounds its contribution to the input $F$ in the estimate for $W$. The receiving-state terms already separated above contribute $C\mathsf S Y_n/\Omega_n$ to the local part of $\mathcal R$ in Section~\ref{transfer:memory-saddle}.

These pieces are chosen along the actual stopped solution and may depend on it. The local estimates use their bounds and require neither a separate evolution nor derivatives of the pieces. The same construction applies when the starting time is zero.

Require $B$ and then $\Lambda$ to be large enough that $\beta_n\le\sigma\Omega_n/4$. Since $\sigma<2\sqrt2/4$, the lemma applies at $\gamma=2\sqrt2\Omega_n$, with physical forcing operator norm at most $C\Xi^2/\Omega_n$. Under $s=\Omega_n(t-t_0)$, a physical kernel $k_n(t,r)$ becomes $\Omega_n^{-2}k_n(t_0+s/\Omega_n,t_0+u/\Omega_n)$.

\[ \frac1{\Omega_n}\int_{t_0}^{t_0+s/\Omega_n}C_G\Xi^2e^{\beta_n(t_0+s/\Omega_n-r)}Y_n(r)\,dr=\frac{C_G\Xi^2}{\Omega_n^2}\int_0^s e^{(\beta_n/\Omega_n)(s-u)}Y_n(t_0+u/\Omega_n)\,du. \]

The normalized norm of the slower-quadrupole integral operator $\mathcal K_n$ is therefore $C\Xi^2/\Omega_n^2$, giving an accumulated relative contribution at most $C\Xi^2T/\Omega_n$ over physical time $T$. Together with the local receiving-state term, $\eta_{\mathrm{return}}\le C(\mathsf S/\Omega_n+\Xi^2/\Omega_n^2)$ in Section~\ref{transfer:memory-saddle}. The source-driven weighted integral gains $C\Xi/\Omega_n$. For $n=1$, the slower-quadrupole subsystem is empty and $\mathcal K_1=0$.

Retain also the source-index factor in the forcing norm. If $j>n$ satisfies $Y_j(t)\le M_{\mathrm{hist}}e^{c\Omega_j(t-\tau_j)}$ for $t\le\tau_j$, then its first-moment difference is bounded by $M_{\mathrm{hist}}\Gamma_j\ell_j^2d_j^{-1}e^{c\Omega_j(t-\tau_j)}$. Thus, in the weighted norm for all $i\le n$, the second-moment forcing bound is multiplied by $C(1+d_j^{-1})$. At a fixed recipient $i$, the factor is $C(1+d_i/d_j)$. For the nearest source,
\begin{equation} \label{array:spacing-frequency}
\sup_{n\ge1}\frac1{\Omega_nd_{n+1}}=\frac2\Lambda\sup_{n\ge1}(2/B)^n\longrightarrow0.
\end{equation}
Hence the factor $d_{n+1}^{-1}$ is offset by the gain $\Omega_n^{-1}$, and the nearest source's contribution through slower quadrupoles is relatively bounded by $C\Xi/(\Omega_nd_{n+1})$. For more distant source indices, retain $d_j^{-1}$ in the strength summation.

Since $\beta_n/\Omega_n\le\sigma/4$, the source estimates in Section~\ref{array:sources-section}, with the spatial factors retained, will bound their source-driven contribution by $C\kappa_n(1+s)^3e^{\sigma s/2}$. Keep the target's returned contribution as the integral operator above.

At $t=\tau_{n+1}$, retain the actual slower-quadrupole states. Lemma~\ref{array:follower-convolution} bounds their later response by contributions from these states, new source forcing, and target forcing after $\tau_{n+1}$. For $n\ge2$, put $j=n+1$ and
\[ V_0=\max_{i<n}(d_i/d_n)^4Y_i(\tau_j). \]
Suppose the preceding completed stage gives
\[ Y_i(\tau_j)\le C M_{\mathrm{hist}}\frac{\Gamma_j\ell_j^2}{\Omega_j|d_i-d_j|^4}\left(1+\frac{d_i}{d_j}+\frac{\Xi}{\Omega_jd_j}\right)+\text{older-source remainder}. \]
Since $|d_i-d_j|$ is comparable to $d_i$, $|d_n-d_{n+1}|=d_j=d_n/2$, and $\Gamma_j\ell_j^2/\Omega_j=\Gamma_j^2/(2\pi\Omega_j^2)$, this implies
\[ V_0\le C M_{\mathrm{hist}}\kappa_nB^{-2}(1+d_j^{-1})+\text{older-source remainder}. \]
By \eqref{array:row-bound} and the slower-quadrupole growth estimate, the resulting forcing in recipient time is at most $(C\Xi/\Omega_n)V_0e^{\sigma s/2}$. Its absolute weighted integral and coefficient in this bound are therefore at most $CM_{\mathrm{hist}}\kappa_nB^{-2}\Xi/(\Omega_nd_{n+1})$, plus the older-source contribution. Include these terms in the source input. The state at $\tau_j$ contains every earlier slower-quadrupole response, including that driven by the current target. This calculation requires no bound on the uncompleted target's entire previous motion. For the deepest source, $\tau_j=0$ and every slower quadrupole agrees with the reference, so $V_0=0$. For $n=1$, also set $V_0=0$ and omit the empty subsystem.

\subsection{Forcing from separated quadrupoles} \label{array:sources-section}

All estimates below are restricted to the current common existence interval, stopped at a geometric or patch-confinement exit. At stage $n$, suppose every quadrupole $j>n$ has first reached pair-center distance $r_{\mathrm{out}}\ell_j$ at time $\tau_j$ and satisfies
\begin{equation} \label{array:prehistory}
Y_j(t)\le M_{\mathrm{hist}}e^{c\Omega_j(t-\tau_j)},\qquad 0\le t\le\tau_j.
\end{equation}
Here $Y_j$ is the norm defined in Section~\ref{array:weights-section}, including shape, angular impulse, corrected phase, first moment, translation, and $W_j$, and $c$ was fixed in Section~\ref{array:choices}. In particular,
\[ |\triangle M_{2,j}(t)|\le CM_{\mathrm{hist}}\Gamma_j\ell_j^2e^{c\Omega_j(t-\tau_j)},\qquad |\triangle M_{1,j}(t)|\le M_{\mathrm{hist}}\frac{\Gamma_j\ell_j^2}{d_j}e^{c\Omega_j(t-\tau_j)}. \]
The initial quadrupole $N$ already has pair-center distance $r_{\mathrm{out}}\ell_N$ at $\tau_N=0$, so its bound is required only at time zero.

By Section~\ref{transfer:outgoing-material}, for $t\ge\tau_j$,
\[ M_{2,j}^{(2)}(t)=M_{2,j}^{(2)}(\tau_j)-\frac{2i}{\pi}\Gamma_j^2e^{2i\theta_{d,j}}(t-\tau_j)+\mathcal E_j(t),\qquad |\dot{\mathcal E}_j(t)|\le \varepsilon_{\mathrm{sep}}\Gamma_j^2, \]
where $\mathcal E_j(\tau_j)=0$ and $\varepsilon_{\mathrm{sep}}$ is smaller than the prescribed multiple of $\varepsilon_*$. The reference second moment and the value at $\tau_j$ are bounded by $C_{\mathrm{sec}}\Gamma_j\ell_j^2$. The first-moment identity and cancellation within each dipole give
\[ |M_{1,j}^{(2)}(t)|+|M_{1,j}^{(1)}(t)|\le C_{\mathrm{sec}} \mathsf{T} \Gamma_j\ell_j^2\bigl(1+\Omega_j(t-\tau_j)^2\bigr),\qquad \tau_j\le t\le1. \]
The linear term in $t-\tau_j$ is included in this bound. Before $\tau_j$, we use the relative estimate \eqref{array:prehistory}.

For a quadrupole $i<j$ near its rhombus, the resulting strain, velocity divided by $d_i$, and Hessian multiplied by $d_i$ are bounded, for $t\le\tau_j$, by
\[ CM_{\mathrm{hist}}\frac{\Gamma_j\ell_j^2}{|d_i-d_j|^4}\left(1+\frac{d_i}{d_j}\right)e^{c\Omega_j(t-\tau_j)}. \]
For $t\ge\tau_j$, they are bounded by
\[ C_{\mathrm{sec}}\frac{\Gamma_j^2}{|d_i-d_j|^4}\left[t-\tau_j+\Omega_j^{-1}+\mathsf{T}d_i\bigl((t-\tau_j)^2+\Omega_j^{-1}\bigr)\right]. \]
Since $\ell_j(1+\Omega_j)\ll|d_i-d_j|$ uniformly, these bounds include the finite spatial extent of the source. For $j\ge2$, the higher-moment error has relative size at most $C\Omega_j\ell_j/d_j$. If $B^\alpha>2$ and $\Lambda$ is sufficiently large,
\[ \sup_{j\ge2}\frac{\Omega_j\ell_j}{d_j}\le C\frac{\Omega_2\ell_2}{d_2}\le C_{\alpha,B}e^{-(B^\alpha-2)\Omega_1^\alpha}\mathsf{S}\le\mathsf{S}. \]
The geometric estimates also give a uniformly small change in the argument of the direct strain. Terms caused by the recipient's own displacement are included in its small diagonal coefficients.

For $n=j-1$, we have $|d_n-d_j|=d_j$, $d_n/d_j=2$, and $\Omega_n/\Omega_j=B^{-1}$. Hence
\begin{equation} \label{array:nearest-errors}
\begin{aligned} \frac{CM_{\mathrm{hist}}\Gamma_j\ell_j^2}{|d_n-d_j|^4}\int_0^{\tau_j}e^{(c\Omega_j-2\sqrt2\Omega_n)(t-\tau_j)}\,dt&\le CM_{\mathrm{hist}}\kappa_nB^{-2},\\ \frac{C_{\mathrm{sec}}\Gamma_j\ell_j^2}{|d_n-d_j|^4}\int_{\tau_j}^{\infty}e^{-2\sqrt2\Omega_n(t-\tau_j)}\,dt&\le C_{\mathrm{sec}}\kappa_nB^{-1}. \end{aligned}
\end{equation}
Indeed, $c\Omega_j-2\sqrt2\Omega_n\ge c\Omega_j/2$ for sufficiently large $B$, and $\Gamma_j\ell_j^2=\Gamma_j^2/(2\pi\Omega_j)$. The relevant integrals are
\[ \begin{aligned} \int_0^{\tau_j}e^{-2\sqrt2\Omega_n(t-\tau_j)}e^{c\Omega_j(t-\tau_j)}\,dt&\le\frac1{c\Omega_j-2\sqrt2\Omega_n},\\ \int_{\tau_j}^{\infty}e^{-2\sqrt2\Omega_n(t-\tau_j)}\,dt&=\frac1{2\sqrt2\Omega_n},\qquad \int_{\tau_j}^{\infty}e^{-2\sqrt2\Omega_n(t-\tau_j)}(t-\tau_j)\,dt=\frac1{8\Omega_n^2}. \end{aligned} \]
The term linear in $t-\tau_j$ has coefficient $O(\Gamma_j^2)$ and weighted integral comparable to $\kappa_n$. Thus the motion before $\tau_j$ and the bounded second moment after $\tau_j$ gain $B^{-2}$ and $B^{-1}$, respectively. The infinite integrals are nonnegative upper bounds. Actual signed forcing is integrated only to the current stopping time.

Integrating the constant and quadratic first-moment terms gives a further contribution, divided by $\kappa_n$, of at most
\[ C_{\mathrm{sec}}\left(\frac{\mathsf{T}d_n}{\Omega_n}+\frac{\mathsf{T}d_n}{B}\right). \]
For forcing through slower quadrupoles, the spatially weighted input has the additional factor $1+d_j^{-1}$ from Section~\ref{array:weights-section}. Its return is controlled by \eqref{array:spacing-frequency}.

\paragraph{Earlier sources.}

Suppose the completed delays are at most $C\log(1/\kappa_k)/\Omega_k$. For $j\ge n+2$, the contribution of quadrupole $j$, divided by $\kappa_n$ and weighted from $\tau_{n+1}$, is bounded by
\begin{equation} \label{array:older-source}
C_{\mathrm{sec}}(1+M_{\mathrm{hist}})\left(1+\frac{d_n}{d_j}+\frac{\Xi}{\Omega_nd_j}\right)\frac{\Gamma_j^2}{\Gamma_{n+1}^2}\exp\left(C_1\Omega_n\sum_{k=n+1}^{j-1}\frac{\log(1/\kappa_k)}{\Omega_k}\right).
\end{equation}
The factor $d_n/d_j$ bounds the first-moment contribution before $\tau_j$, and $\Xi/(\Omega_nd_j)$ bounds forcing through slower quadrupoles. The constants include the bounded first-moment terms after $\tau_j$ and the small return to the target. The same bound applies to the stable-shape and remaining finite-coordinate values and to $W_n(\tau_{n+1})$, since a slightly larger target exponential absorbs the polynomial factors and the bounded physical-time rate in $W_n$.

The delay and growth constants depend only on the fixed thresholds $q_*,r_{\mathrm{out}}$ once $B$ exceeds a fixed threshold. For each such $B$, take $\Lambda$ large enough that $\log(1/\kappa_k)\le5\Omega_{k+1}^\alpha$ for all $k$. Then
\[ \Omega_n\sum_{k=n+1}^{j-1}\frac{\log(1/\kappa_k)}{\Omega_k}\le\frac{5\Omega_{n+1}^\alpha}{B^{1-\alpha}-1}. \]
Also,
\[ \log\frac{\Gamma_j^2}{\Gamma_{n+1}^2}=-4\bigl(\Omega_j^\alpha-\Omega_{n+1}^\alpha\bigr)+2(j-n-1)\log B. \]
The logarithm of the spatial factor in \eqref{array:older-source} is at most $C+j\log2+\log(1+\Xi/\Omega_n)$. Since $j\ge n+2$, the logarithm of the circulation ratio plus the exponent in \eqref{array:older-source} is at most
\[ -\left(4-\frac4{B^\alpha}-\frac{5C_1}{B^\alpha(B^{1-\alpha}-1)}\right)\Omega_j^\alpha+2(j-n-1)\log B. \]
Choose $B$ so that the parenthesis is at least three, then increase $\Lambda$ to absorb the spatial and index terms, uniformly in $n,j$. Summing \eqref{array:older-source} over $j\ge n+2$ gives
\begin{equation} \label{array:older-sum}
C_{\mathrm{sec}}(1+M_{\mathrm{hist}})e^{-c\Omega_{n+2}^\alpha}.
\end{equation}
This is uniform in $n$ and $N$, and retains the first-moment factor $2^{j-n}$. All these contributions are estimated in absolute value.

\paragraph{Pointwise bounds.}

Under the completed-stage hypotheses, write $\mathcal F_{n\leftarrow j}$ for the nonnegative physical input norm from quadrupole $j$, including all projected equations, and $\mathcal F_n^{\mathrm{source}}$ for an upper bound on their combined input, including the response through slower quadrupoles. Set $t_0=\tau_{n+1}$ and $s=\Omega_n(t-t_0)$. For $j=n+1$, the preceding profiles give
\[ \frac{\mathcal F_{n\leftarrow j}(t_0+s/\Omega_n)}{\Omega_n}\le C_{\rm dir}\kappa_ns+C_{\rm sec}\kappa_n\left[B^{-1}+\mathsf{T}d_n\left(\frac{s^2}{\Omega_n}+B^{-1}\right)\right]. \]
Here $C_{\rm dir}$ includes the small material and geometric errors in the linear term. The remaining terms are the constant second-moment term and the constant and quadratic first-moment terms. Since $\Gamma_j^2/(\Omega_n^2|d_n-d_j|^4)\asymp\kappa_n$, these remaining terms are at most $C_{\rm sec}(B^{-1}+\mathsf{T})\kappa_n(1+s)^2$.

For $j\ge n+2$, the completed delays give
\[ 0\le\Omega_n(\tau_{n+1}-\tau_j)\le\frac{C\Omega_{n+1}^\alpha}{B^{1-\alpha}-1}. \]
Thus its input, divided by $\Omega_n$, is at most
\[ C_{\rm sec}\kappa_n(1+s)^2(1+d_j^{-1})\frac{\Gamma_j^2}{\Gamma_{n+1}^2}\left[1+\Omega_n(\tau_{n+1}-\tau_j)\right]^2. \]
The logarithms of the circulation ratio, spatial factor, and age polynomial are bounded by $-4(\Omega_j^\alpha-\Omega_{n+1}^\alpha)+2(j-n-1)\log B$, $C+j\log2$, and $C+2\alpha\log\Omega_{n+1}$, respectively. The calculation proving \eqref{array:older-sum} therefore gives
\[ \sum_{j\ge n+2}\frac{\mathcal F_{n\leftarrow j}(t_0+s/\Omega_n)}{\Omega_n}\le C_{\rm sec}(1+M_{\rm hist})e^{-c\Omega_{n+2}^\alpha}\kappa_n(1+s)^2. \]
The same estimate applies to the weighted input to the slower quadrupoles. For $j=n+1$, retain the factor $1+d_{n+1}^{-1}$. For $j\ge n+2$, this factor is already included in the sum.

Choose $\beta_n/\Omega_n\le\sigma/4$. Lemma~\ref{array:follower-convolution} bounds the return of these inputs after $t_0$, divided by $\Omega_n$, by
\[ C\kappa_n\left[\frac{\Xi}{\Omega_nd_{n+1}}+\frac{\Xi}{\Omega_n}C_{\rm sec}(1+M_{\rm hist})e^{-c\Omega_{n+2}^\alpha}\right](1+s)^3e^{\sigma s/4}. \]
This follows by integrating $(1+s)^2$ against $e^{\beta_n(s-r)/\Omega_n}$. All interactions among slower quadrupoles are included in $\beta_n$.

The preceding stage bounds the actual norm $V_0$ from Lemma~\ref{array:follower-convolution} by
\[ V_0\le CM_{\rm hist}\kappa_nB^{-2}(1+d_{n+1}^{-1})+C_{\rm sec}(1+M_{\rm hist})\kappa_ne^{-c\Omega_{n+2}^\alpha}. \]
Its homogeneous return, divided by $\Omega_n$, is at most $(C\Xi/\Omega_n)V_0e^{\sigma s/4}$. The target norm at $t_0$ is at most $CM_{\rm hist}\kappa_nB^{-2}$ plus the same older-source remainder. The value $V_0$ includes every earlier response, including the response to quadrupole $n$.

Combining these bounds gives
\[ \frac{\mathcal F_n^{\mathrm{source}}(t_0+s/\Omega_n)}{\Omega_n\kappa_n(1+s)^3e^{\sigma s/2}}\le{} C_{\rm dir}+C_{\rm sec}(1+M_{\rm hist})\Biggl(B^{-1}+\mathsf{T} +\sup_k\frac{\Xi}{\Omega_kd_{k+1}}+e^{-c\Omega_3^\alpha}\Biggr). \]
The target norm at $t_0$ is at most $C_{\rm sec}(1+M_{\rm hist})(B^{-2}+e^{-c\Omega_3^\alpha})\kappa_n$. Choose $B$, then $\Lambda$, so that the error coefficient in the preceding display is at most $C_{\rm dir}$ and this initial-value coefficient is at most $M_0$. The source input is then bounded by $2C_{\rm dir}\kappa_n(1+s)^3e^{\sigma s/2}$, and its absolute weighted integral is at most $2C_{\rm dir}\kappa_n\int_0^\infty e^{-(2\sqrt2-\sigma/2)s}(1+s)^3\,ds\le M_0\kappa_n/2$. These bounds lie strictly inside the input class fixed in Section~\ref{array:choices}.

\subsection{The history constant and the phases} \label{array:phases-section}

We verify the choice \eqref{array:history-constant}. Recall that $M_0$ and $A_*=1/4$ were fixed before $q_*$. At $Q=q_*$, the signed delay formula gives $e^{2\sqrt2\Omega_n(t-\tau_{n+1})}\le Cq_*/(\kappa_nA_*)$. An absolute weighted input $\kappa_n(M_0+\text{error})$ therefore gives a positive comparison bound $Cq_*A_*^{-1}(M_0+\text{error})$ at this crossing. The products involving stable shape, impulse, phase, first-moment, and translation differences contribute $\exp(Cq_*)$, the small accumulated coefficients contribute at most two, and the fixed outgoing segment contributes a fixed bound for the linear evolution. These are the fixed factors included in $C_{\mathrm{out}}$ in Section~\ref{array:choices}, together with the initial scale-$N$ bound. The factor $A_*^{-1}$ is kept explicit.

The linear nearest-source term has absolute weighted integral at most $M_0\kappa_n$, independently of $M_{\mathrm{hist}}$. Its constant term, the motion before $\tau_{n+1}$, earlier outgoing quadrupoles, and forcing through slower quadrupoles add at most $C_{\mathrm{sec}}M_{\mathrm{hist}}\zeta_{\Lambda,B}\kappa_n$, where
\[ \zeta_{\Lambda,B}=B^{-1}+B^{-2}+\mathsf{T}+\sup_n\frac{\Xi}{\Omega_nd_{n+1}}+e^{-c\Omega_3^\alpha}. \]
The same bound controls the pointwise coefficient, by \eqref{array:older-sum}.

Choose $B$, then $\Lambda$, so that
\[ C_{\mathrm{sec}}M_{\mathrm{hist}}\zeta_{\Lambda,B}\le\tfrac12M_0 \]
and the corresponding signed weighted errors are below the prescribed margin. The input norm is then below $2M_0\kappa_n$, and the new constant in \eqref{array:prehistory} is at most
\[ C_{\mathrm{out}}A_*^{-1}\bigl(M_0+C_{\mathrm{sec}}M_{\mathrm{hist}}\zeta_{\Lambda,B}\bigr)\le\tfrac34M_{\mathrm{hist}}. \]
Steps 3--6 below prove this estimate on the stopped solution.

\paragraph*{Phase and timing.}
Use the phases, predicted times, and initial configurations fixed in Section~\ref{second:configuration}. Also use $h_n$ from Section~\ref{second:cascade-estimates}. By \eqref{eq:positive-coefficient}, the normalized direct coefficient is
\[ \cos\bigl(4\arg(d_n-d_{n+1})-2\theta_n^{(1)}(\tau_{n+1})-2\theta_{d,n+1}+\arctan(1/\sqrt2)-\pi\bigr). \]
By \eqref{second:phase-choice}, its argument vanishes at the ideal phases, since $d_n-d_{n+1}>0$.

Make the choices in Section~\ref{array:choices} so that
\[ C\varepsilon_{\mathrm{sep}}+C_{\mathrm{sec}}\left(M_{\mathrm{hist}}\zeta_{\Lambda,B}+\sup_n\left[\eta_n\log(1/\kappa_n)+\kappa_n^{3/4}\right]+\mathsf{S}+\sup_n r_n/d_n\right)\le\varepsilon_*, \]
where $\eta_n$ is \eqref{transfer:material-error}, including the return, first-moment, and translation coefficients estimated above. The displayed $C$ is universal, and $C_{\mathrm{sec}}$ includes every fixed segment factor. By \eqref{transfer:material-ramp}, $C\varepsilon_{\mathrm{sep}}$ is bounded by the already chosen uniform multiple of $r_{\mathrm{out}}^{-2}$ plus terms tending to zero with $\Lambda^{-1}$, including $C_{\mathrm{sec}}[\eta_n\log(1/\kappa_n)+\kappa_n^{3/4}]$. All remaining terms are made small by choosing $B$, then $\Lambda$. Thus this budget bounds the normalized forcing, phase, and delay errors by the tolerance fixed in Section~\ref{array:choices}. In particular, Section~\ref{transfer:material-block} gives, at every completed stage,
\[ \theta_{d,n}-\vartheta_n=-h_n+O(\varepsilon_*). \]
Both timing and angle errors vanish at scale $N$. At time $\tau_{n+1}$, when quadrupole $n+1$ first reaches its prescribed dipole separation, the phases in the next weighted coefficient are
\[ \begin{aligned} \theta_n^{(1)}(\tau_{n+1})&=\vartheta_n+\Omega_n\Delta_n-\beta-B^{-1}h_{n+1}+O(\varepsilon_*),\\ \theta_{d,n+1}&=\vartheta_{n+1}-h_{n+1}+O(\varepsilon_*). \end{aligned} \]
Since $\Omega_n/\Omega_{n+1}=B^{-1}$, substitution into the cosine gives
\[ \mathcal A_n=\cos\bigl(2(1+B^{-1})h_{n+1}\bigr)+O(\varepsilon_*). \]
For $|2(1+B^{-1})h_{n+1}|\le1/2$, the choices above make the absolute $O(\varepsilon_*)$ error at most $1/8$. Hence $\mathcal A_n\ge3/4$, and the delay formula yields
\[ h_n=B^{-1}h_{n+1}-\frac1{2\sqrt2}\log\cos\bigl(2(1+B^{-1})h_{n+1}\bigr)+O(\varepsilon_*). \]
Since $|\log\cos x|\le Cx^2$ on this interval, we obtain
\begin{equation} \label{array:timing-recursion}
|h_n|\le B^{-1}|h_{n+1}|+C|h_{n+1}|^2+C\varepsilon_*.
\end{equation}
Here $C$ is the elementary constant fixed before $\varepsilon_*$ in Section~\ref{array:choices}. If $|h_{n+1}|\le C_h\varepsilon_*$, that choice keeps the cosine argument in $[-1/2,1/2]$ and bounds the right side by $3C_h\varepsilon_*/4$. Starting from $h_N=0$, this proves $|h_n|\le C_h\varepsilon_*$ at every completed stage. The preceding phase formula then gives the claimed angle error.

\subsection{The stopped inductive step} \label{array:induction-section}

We prove by descending induction that every quadrupole $n$ first reaches pair-center distance $r_{\mathrm{out}}\ell_n$ within the deterministic duration of Section~\ref{transfer:material-block}, satisfies \eqref{array:prehistory} and \eqref{array:timing-recursion}, and supplies the actual values for the next stage. All estimates before completion are restricted to the stopped interval.

\paragraph*{Step 1. Initial values and earlier times.}
At time zero, quadrupole $N$ has pair-center distance $r_{\mathrm{out}}\ell_N$ and its prescribed angle, and satisfies the bound $M_{\mathrm{hist}}$. Suppose the conclusions hold for $j>n$, while every $i\le n$ remains near its rhombus. Starting at $t_0=\tau_{n+1}$, stop at $Q_n=\pm q_*$, another target or slower-quadrupole chart boundary, a geometric or patch-confinement exit, or the upper time of Section~\ref{transfer:memory-saddle}. Eliminate only $i<n$ by Lemma~\ref{array:follower-convolution}. Their response to completed sources is an input, and their response to $n$ is bounded by the positive integral operator. The preceding stage supplies the actual shape, impulse, phase, first-moment, and translation differences, together with $W_n(t_0)$, bounded by $CM_{\mathrm{hist}}\kappa_nB^{-2}$ plus the older-source remainder. For $n=N-1$, $t_0=0$ and all recipient and slower-quadrupole differences vanish. At every index below $N$, the comparison begins with identical centroids and normalized disk boundaries at the original time zero.

For each completed stage retain its deterministic duration bound from Section~\ref{transfer:material-block}. Since $\tau_N=0$, the choice in Section~\ref{array:choices} gives
\[ \tau_{n+1}+\frac1{\Omega_n}\left[\frac1{2\sqrt2}\log\frac4{A_*\kappa_n}+K+1\right] \le\sum_{k=n}^{N-1}\frac1{\Omega_k}\left[\frac1{2\sqrt2}\log\frac4{A_*\kappa_k}+K+1\right]<\frac{1}{2}. \]
Thus the interval required by Section~\ref{transfer:material-block}, including continuation up to $\tau_n$, is contained in $[0,1/2]$. The original disks have identical normalized labels, radii, and signed circulations, as required there. The value $W_n(t_0)$ is retained.

On $[0,t_0]$, the completed stages keep every $i\le n$ near its rhombus. Their local rates are at most $C\Omega_n$, so the spatial estimate \eqref{array:row-bound} bounds the linear evolution by $Ce^{\gamma_nt}$, where $\gamma_n=C_f\Omega_n+C\Xi$ and $C_f$ is fixed. Choose $B$ so that $c\Omega_{n+1}>2\gamma_n$. The external inputs $k>n$ have the already proved profiles.

Convolution of the nearest source bound before $\tau_{n+1}$ costs $C/\Omega_{n+1}$. Its direct spatial factor is $1+d_n/d_{n+1}$. The indirect term has an additional $\Xi/\Omega_{n+1}$ applied to $1+d_{n+1}^{-1}$. Thus its contribution, for every $t\le t_0$, is at most
\[ CM_{\mathrm{hist}}\frac{\Gamma_{n+1}\ell_{n+1}^2}{\Omega_{n+1}|d_n-d_{n+1}|^4}\left(1+\frac{d_n}{d_{n+1}}+\frac{\Xi}{\Omega_{n+1}d_{n+1}}\right)\le CM_{\mathrm{hist}}\frac{\kappa_n}{B^2}. \]
For $k\ge n+2$, the bound before $\tau_k$ has the factor $e^{\gamma_n(t_0-\tau_k)}$. After $\tau_k$, set $v=t-\tau_k$. The profile is bounded by $C\Gamma_k^2[v+\Omega_k^{-1}+\mathsf{T}(v^2+\Omega_k^{-1})]$ times its spatial factor. Its integral against $e^{-\gamma_nv}$ costs $C/\Omega_n^2$, since $\gamma_n\asymp\Omega_n$, $\Omega_k\ge\Omega_n$, and $\mathsf{T}\le1$. Hence the older sources contribute at most
\[ \frac{C_{\mathrm{sec}}(1+M_{\mathrm{hist}})}{\Omega_n^2d_n^4}\sum_{k\ge n+2}(1+d_k^{-1})\Gamma_k^2e^{\gamma_n(\tau_{n+1}-\tau_k)}. \]
After division by $\kappa_n/B^2$, \eqref{array:older-source} applies with an additional fixed factor $B^2$ and spatial factor $1+d_k^{-1}$. Its logarithm is at most $C+k\log2$, and the completed delays bound $t_0-\tau_k$ by the same sum as before. Increasing $\Lambda$ therefore bounds this contribution by $C_{\mathrm{sec}}(1+M_{\mathrm{hist}})(\kappa_n/B^2)e^{-c\Omega_{n+2}^\alpha}$. Consequently,
\begin{equation} \label{array:pre-reset}
\sup_{0\le t\le\tau_{n+1}}Y_n(t)\le CM_{\mathrm{hist}}\frac{\kappa_n}{B^2},\qquad \Omega_n\int_0^{\tau_{n+1}}Y_n(t)\,dt\le C_{\alpha,B}M_{\mathrm{hist}}\frac{\kappa_n\Omega_n^\alpha}{B^2}=o(\kappa_n^{3/4}).
\end{equation}
The integral estimate uses the completed delay bound $\tau_{n+1}\le C_{\alpha,B}\Omega_n^{\alpha-1}$. The small-o bound is uniform in $n$, since $2^n\le\Omega_n$ and
\[ \kappa_n^{1/4}\Omega_n^\alpha\le C_{\alpha,B}\Omega_n^{\alpha+1}e^{-\Omega_{n+1}^\alpha}\longrightarrow0. \]
When $n=N-1$, $t_0=0$ and both left sides vanish. Since the quadratic point remainders are bounded by $C\Omega_nY_n^2$, \eqref{array:pre-reset} also bounds the integral of their coefficients. The reference and patch coefficients have the bounds already proved.

\paragraph*{Step 2. The signed weighted forcing.}
Let $S$ be the current stopping time in normalized units. Using the exact split after Lemma~\ref{array:follower-convolution}, let $f_q^{\mathrm{act}}$ be the direct forcing in the normalized $Q$ equation, divided by $\kappa_n$. It includes the inherited and completed-source contributions and leaves the target return in the remainder. Let $f_q^{\mathrm{id}}$ be the ideal nearest linear profile, with the source axis fixed and the recipient rotating uniformly from its actual phase at $t_0$. Define
\[ f_q^{[S]}(s)=\begin{cases}f_q^{\mathrm{act}}(s),&0\le s\le S,\\ f_q^{\mathrm{id}}(s),&s>S.\end{cases} \]
\[ A_S=\frac{Q_n(t_0)}{\kappa_n}+\int_0^\infty e^{-2\sqrt2 s}f_q^{[S]}(s)\,ds,\qquad A_{\mathrm{id}}=\int_0^\infty e^{-2\sqrt2 s}f_q^{\mathrm{id}}(s)\,ds. \]
The direct forcing includes the reference receiving-patch average in Section~\ref{transfer:material-block}. Its additional normalized weighted error is $C(a_n^2/\ell_n^2)\Omega_n^{-2}=O(\Omega_n^{-5})\le C\eta_n\log(1/\kappa_n)$. The identity
\[ A_S=A_{\mathrm{id}}+\frac{Q_n(t_0)}{\kappa_n}+\int_0^S e^{-2\sqrt2 s}\bigl(f_q^{\mathrm{act}}(s)-f_q^{\mathrm{id}}(s)\bigr)\,ds \]
and the initial-value and weighted source bounds give $|A_S-A_{\mathrm{id}}|\le C\varepsilon_*$, uniformly in $S$. The phase estimate, with cosine at least $7/8$ and total coefficient error at most $1/8$, gives $A_S\ge3/4>A_*$. Both pieces of $f_q^{[S]}$ satisfy the same polynomial and slow-exponential bound. For $0\le r\le S$,
\[ \frac{Q_n(t_0)}{\kappa_n}+\int_0^r e^{-2\sqrt2 s}f_q^{\mathrm{act}}(s)\,ds=A_S-\int_r^\infty e^{-2\sqrt2 s}f_q^{[S]}(s)\,ds. \]
\[ \left|\int_r^\infty e^{-2\sqrt2 s}f_q^{[S]}(s)\,ds\right|\le C(1+r)^3e^{-(2\sqrt2-\sigma/2)r}. \]
By Section~\ref{transfer:memory-saddle}, these identities exclude $Q_n=-q_*$ and reaching the deterministic upper time before $Q_n=q_*$, provided the other exits are excluded below. All constants are independent of $S$. The extension after $S$ is used only in the displayed integrals, while the return operator remains restricted to $[0,S]$. At $Q_n=q_*$, the tail is $O(\kappa_n^{3/4})$ and $A_S$ is the coefficient in the delay formula.

At $Q_n=q_*$ set $\mathcal A_n=A_S$. To bound earlier motion, eliminate the slower quadrupoles from the original time zero, when their relative states vanish. Let $F_n(t)\ge0$ bound all projected centroid, moment, and patch inputs from indices greater than $n$, including their responses through slower quadrupoles, and let $\mathcal K_n$ be the positive return operator starting at zero. It includes the earlier response to $n$, which is therefore omitted from $F_n$. In the local formulation starting at $t_0$, that response is instead included in the homogeneous slower-quadrupole input. Lemma~\ref{array:follower-convolution} bounds the forcing-level return norm by $C\Xi^2/\Omega_n$ and the weighted integral of a source-driven response by $C\Xi/\Omega_n$ times its weighted source norm. Thus, for every stopped time $T$,
\[ \int_0^T e^{-2\sqrt2\Omega_n(t-\tau_{n+1})}F_n(t)\,dt\le\bigl(M_0+C_{\mathrm{sec}}M_{\mathrm{hist}}\zeta_{\Lambda,B}\bigr)\kappa_n<2M_0\kappa_n. \]
The bound includes the inputs to every component of $Y_n$. Together with the pointwise bounds and the ideal extension after $S$, it verifies the fixed input class without continuing the physical motion beyond the stop.

\paragraph*{Step 3. A comparison bound for the recipient.}
We use the signed equation to control departure and a nonnegative comparison function to bound $Y_n$ at every earlier time. Here $Q=Q_n=\triangle\widehat q_n$, $P=P_n=\triangle\widehat p_n$, and $\nu=(\triangle I_n/(2\Gamma_n\ell_n^2),\triangle\widehat\Theta_n,\widehat M_{1,n},\widehat d_n)$ collects the other finite differences defined in Section~\ref{array:weights-section}. Use the sum of their absolute values for $|\nu|$.

The changes of shape and phase in Section~\ref{transfer:linearization} remove the nonlinear terms depending only on the unstable shape variable. The terms $QP$ and $Q\nu$ cost at most $\exp(Cq_*)$, since the normalized-time integral of $|Q|$ before $Q_n=q_*$ is $O(q_*)$. Their return to the unstable equation uses the estimate $|P|+|\nu|\le C\eta\kappa e^{2\sqrt2s}+C\kappa e^{\sigma s}$ in Section~\ref{transfer:material-block}. The remaining accumulated coefficient and integral-error terms are bounded by
\begin{equation} \label{array:accumulated-error}
C_{\mathrm{sec}}\left[\mathsf{S} T+\Omega_n^{-3}\Omega_nT+\frac{\Xi^2T}{\Omega_n}+\kappa_n^{3/4}\right]
\end{equation}
while $\kappa_ne^{2\sqrt2\Omega_n(t-\tau_{n+1})}\le Cq_*$. The completed delays and current upper time give
\[ T\le C\frac{\log(1/\kappa_n)}{\Omega_n}\le C_{\alpha,B}\Omega_n^{\alpha-1}. \]
Since $\alpha<1$, the accumulated errors tend to zero uniformly as $\Lambda\to\infty$. More explicitly, \eqref{transfer:material-error} and the normalized return bound $C\Xi^2/\Omega_n^2$ give
\[ \eta_n\log(1/\kappa_n)\le C_{\alpha,B}\left[(\mathsf S+\mathsf S^2)\Omega_n^{\alpha-1}+\Omega_n^{\alpha-3}+\mathsf T\ell_n\Omega_n^\alpha+\Xi^2\Omega_n^{\alpha-2}\right]=o_{\Lambda\to\infty}(1). \]
The term containing $\ell_n$ is small by the series estimate in Section~\ref{array:choices}. The four terms in \eqref{array:accumulated-error} come from the external and reference coefficients, patch feedback, slower-quadrupole return, and forcing of $P$ and $\nu$, respectively. In particular,
\[ \varepsilon_n\Omega_nT\le C\Omega_n^{-3}\log(1/\kappa_n)=O_{\alpha,B}(\Omega_n^{\alpha-3}),\qquad \frac{\Xi^2T}{\Omega_n}\le C_{\alpha,B}\Xi^2\Omega_n^{\alpha-2}. \]
Before $t_0$, the nonlinear coefficient integral is bounded by \eqref{array:pre-reset}. The fixed factor $\exp(Cq_*)$ is included in $C_{\mathrm{out}}$.

In normalized time the unperturbed linear operator is $\operatorname{diag}(2\sqrt2,-2\sqrt2,\mathcal N)$, where $\mathcal N^2=0$ describes the impulse-to-phase coupling in Section~\ref{transfer:material-section}. The first-moment and translation coordinates have zero unperturbed equations. Their small coupling is included in $\alpha_n$ below. Since impulse-to-shape terms vanish at the rhombus, the linear evolution is bounded by $C[e^{2\sqrt2s}+1+s]\le C_Ze^{2\sqrt2s}$ for $s\ge0$. Recall $Z_n=Y_n-W_n=|Q|+|P|+|\nu|$, and write $(E_a*f)(t)=\int_0^te^{a(t-r)}f(r)\,dr$. Then
\[ \begin{aligned} Z_n&\le C_ZE_{2\sqrt2\Omega_n}*[F_n+\alpha_n Z_n+C\varepsilon_n\Omega_nW_n+\mathcal K_n(Y_n)],\\ W_n&\le C_WE_\mu*[\Omega_nZ_n+F_n+\mathcal K_n(Y_n)],\qquad \mu=O(1). \end{aligned} \]
The nonnegative coefficient $\alpha_n(t)$ contains the reference, source-evaluation, and nonlinear centroid terms. Its integral before $t_0$ is controlled by \eqref{array:pre-reset}. Afterwards the simultaneous estimates of Section~\ref{transfer:memory-saddle} bound it by $Cq_*$ plus the small terms in \eqref{array:accumulated-error}. The integral-kernel rate $\mu$ is bounded independently of $n$.

Substitute the inequality for $Z_n$ into that for $W_n$. The identity
\[ \Omega_n E_\mu*E_{2\sqrt2\Omega_n}=\frac{\Omega_n}{2\sqrt2\Omega_n-\mu}(E_{2\sqrt2\Omega_n}-E_\mu)\le C E_{2\sqrt2\Omega_n} \]
holds with a uniform constant. Since $E_\mu\le E_{2\sqrt2\Omega_n}$, we obtain
\[ Y_n(t)\le C\int_0^t e^{2\sqrt2\Omega_n(t-r)}\left[F_n(r)+(\alpha_n(r)+\varepsilon_n\Omega_n)Y_n(r)+\mathcal K_n(Y_n)(r)\right]dr. \]
Define the nonnegative comparison function
\[ R_n(t)=e^{2\sqrt2\Omega_nt}\max_{0\le r\le t}e^{-2\sqrt2\Omega_nr}Y_n(r). \]
Then $Y_n\le R_n$ and $R_n(r)\le e^{-2\sqrt2\Omega_n(t-r)}R_n(t)$ for $r\le t$. In particular,
\[ \mathcal K_n(Y_n)(r)\le C_G\Xi^2\int_0^r e^{\beta_n(r-u)}R_n(u)\,du\le\frac{C\Xi^2}{\Omega_n}R_n(r). \]
Multiply the preceding inequality for $Y_n$ by $e^{-2\sqrt2\Omega_nt}$. Its right side is nondecreasing, so taking the maximum over earlier times gives
\[ e^{-2\sqrt2\Omega_nt}R_n(t)\le C\int_0^t e^{-2\sqrt2\Omega_nr}F_n(r)\,dr+C\int_0^t\left[\alpha_n(r)+\varepsilon_n\Omega_n+\frac{\Xi^2}{\Omega_n}\right]e^{-2\sqrt2\Omega_nr}R_n(r)\,dr. \]
Gronwall costs only $C\int_0^T\alpha_n+C\varepsilon_n\Omega_nT+C\Xi^2T/\Omega_n$, already bounded above. Step 2 therefore gives
\begin{equation} \label{array:majorant}
\begin{aligned} Y_n(t)&\le R_n(t),\qquad R_n(r)\le e^{-2\sqrt2\Omega_n(t-r)}R_n(t)\quad(0\le r\le t),\\ e^{-2\sqrt2\Omega_n(t-\tau_{n+1})}R_n(t)&\le C\bigl(M_0+C_{\mathrm{sec}}M_{\mathrm{hist}}\zeta_{\Lambda,B}\bigr)\kappa_n. \end{aligned}
\end{equation}
This estimate uses absolute values. The signed coefficient $\mathcal A_n$ is used only to determine the departure time.

Apply the same argument with $T=t_0=\tau_{n+1}$. The nearest source contributes only its motion before $\tau_{n+1}$, whose weighted input is bounded by the first inequality in \eqref{array:nearest-errors}. Its response through slower quadrupoles has the additional small factor $C\Xi/(\Omega_nd_{n+1})$. Equations~\eqref{array:older-source}--\eqref{array:older-sum} bound the other sources. The coefficient integral in Gronwall uses only \eqref{array:pre-reset} and the reference, patch, and return bounds on $[0,t_0]$. Hence
\[ \sup_{0\le t\le t_0}e^{2\sqrt2\Omega_n(t_0-t)}Y_n(t)\le CM_{\mathrm{hist}}\kappa_nB^{-2}+C_{\mathrm{sec}}(1+M_{\mathrm{hist}})\kappa_ne^{-c\Omega_{n+2}^\alpha}. \]
Increasing $\Lambda$ absorbs the last term into the first. This proves the bound used in Section~\ref{second:cascade-estimates} and Figure~\ref{fig:second-dependencies}, without using the target's later motion.

\paragraph*{Step 4. Slower quadrupoles and the crossing $Q_n=q_*$.}
By \eqref{array:majorant} and the deterministic upper time, $R_n(t)\le C_{\mathrm{sec}}q_*$. We prove, for every current stopped time,
\begin{equation} \label{array:follower-current}
Y_i(t)\le C_{\mathrm{sec}}\frac{\Gamma_n\ell_n^2}{\Omega_n|d_i-d_n|^4}\left(1+\frac{d_i}{d_n}+\frac{\Xi}{\Omega_nd_n}\right)R_n(t)+\mathcal R_{i,n}(t),\qquad i<n.
\end{equation}
Here $\mathcal R_{i,n}$ is the nonnegative comparison contribution from completed indices greater than $n$.

Fix the actual coefficients and iterate the slower-quadrupole comparison inequalities from time zero. Write the resulting nonnegative bounds as $Y_i\le Y_i^{\mathrm{target}}+\mathcal R_{i,n}$, where $Y_i^{\mathrm{target}}$ contains the inputs from $n$ and $\mathcal R_{i,n}$ contains those from completed indices. Both comparison terms have zero initial value. The multipole estimates and $Y_n\le C_{\mathrm{sec}}R_n$ give coefficients
\[ D_{in}=C\frac{\Gamma_n\ell_n^2}{|d_i-d_n|^4}\left(1+\frac{d_i}{d_n}\right),\qquad \max_{i<n}\left(\frac{d_i}{d_n}\right)^4D_{in}\le C\frac{\Gamma_n\ell_n^2}{d_n^4}(1+d_n^{-1}). \]
By $R_n(r)\le R_n(t)e^{-2\sqrt2\Omega_n(t-r)}$ and Lemma~\ref{array:follower-convolution},
\[ \max_{i<n}\left(\frac{d_i}{d_n}\right)^4Y_i^{\mathrm{target}}(t)\le\frac{C_G}{2\sqrt2\Omega_n-\beta_n}\max_{j<n}\left(\frac{d_j}{d_n}\right)^4D_{jn}\,R_n(t). \]
Apply \eqref{array:forced-propagator} again to index $i$. Its direct input is $D_{in}R_n$, and \eqref{array:row-bound} bounds the remaining input by $\Xi$ times the preceding weighted norm. Integration gives
\[ \begin{aligned} Y_i^{\mathrm{target}}(t)&\le\frac{C_GD_{in}}{2\sqrt2\Omega_n-C_f\Omega_i}R_n(t)\\ &\quad+\left(\frac{d_n}{d_i}\right)^4\frac{C_G^2\Xi}{(2\sqrt2\Omega_n-C_f\Omega_i)(2\sqrt2\Omega_n-\beta_n)}\max_{j<n}\left(\frac{d_j}{d_n}\right)^4D_{jn}\,R_n(t). \end{aligned} \]
Both denominators are comparable to $\Omega_n$, since $i<n$, and $|d_i-d_n|\asymp d_i$. In particular,
\[ \int_0^t e^{C_f\Omega_i(t-r)}R_n(r)\,dr\le\frac{R_n(t)}{2\sqrt2\Omega_n-C_f\Omega_i}. \]

Substitution of the bounds for $D_{in}$ gives
\[ Y_i^{\mathrm{target}}(t)\le{} C_{\mathrm{sec}}\frac{\Gamma_n\ell_n^2}{\Omega_n|d_i-d_n|^4} \left(1+\frac{d_i}{d_n}+\frac{\Xi}{\Omega_nd_n}\right)R_n(t). \]
Adding $\mathcal R_{i,n}$ proves \eqref{array:follower-current}. The same calculation holds on the compact outgoing segment with a fixed larger multipole constant, provided its continuation satisfies $R_n(r)\le e^{-2\sqrt2\Omega_n(t-r)}R_n(t)$.

For $n\ge2$, the subsystem $i<n$ has linear evolution bounded by $Ce^{\beta_n(t-r)}$, where $\beta_n=C_f\Omega_{n-1}+C_G\Xi\asymp\Omega_{n-1}$. The contribution of $k>n$ before $\tau_k$ costs $C\Gamma_k^2/\Omega_k^2$. After factoring out $\Gamma_k^2e^{\beta_n(t-\tau_k)_+}$, integrating the linear, constant, and quadratic terms in $t-\tau_k$ gives $C\beta_n^{-2}$, $C(\beta_n\Omega_k)^{-1}$, and $C\mathsf{T}\beta_n^{-3}$, respectively. Since $\beta_n\ge1$, $\mathsf{T}\le1$, and $\Omega_k\ge\beta_n$,
\[ \mathcal R_{i,n}(t)\le\frac{C_{\mathrm{sec}}(1+M_{\mathrm{hist}})}{\beta_n^2d_i^4}\sum_{k>n}(1+d_k^{-1})\Gamma_k^2e^{\beta_n(t-\tau_k)_+},\qquad i<n. \]
The local time bound, including continuation up to $\tau_n$, gives
\[ t-\tau_{n+1}\le C_T\log(1/\kappa_n)/\Omega_n. \]
Together with the completed delays, $\beta_n\le C\Omega_n/B$, and $B^{1-\alpha}\ge2$, this yields
\[ \beta_n(t-\tau_k)_+\le C_2B^{-1}\Omega_{n+1}^\alpha,\qquad k>n, \]
with $C_2$ fixed before $B$. Hence
\[ \log\frac{\Gamma_k^2}{\Gamma_n^2}+\beta_n(t-\tau_k)_+\le-\left(4-\frac4{B^\alpha}-\frac{C_2}{B}\right)\Omega_k^\alpha+2(k-n)\log B. \]
Choose $B$ so that the parenthesis is at least three, then increase $\Lambda$ to absorb $\log(1+d_k^{-1})$ and the index terms. The sum is at most $C\Gamma_n^2e^{-c\Omega_{n+1}^\alpha}$. Since $\beta_n\asymp\Omega_{n-1}$ and $|d_{n-1}-d_n|=d_n$,
\[ \mathcal R_{n-1,n}(t)\le C_{\mathrm{sec}}(1+M_{\mathrm{hist}})\kappa_{n-1}e^{-c\Omega_{n+1}^\alpha}. \]
For $n=1$, there are no slower quadrupoles.

For $n\ge2$, use $\Gamma_n\ell_n^2/\Omega_n=2\pi\ell_n^4$ and $d_i\ge2d_n$ in \eqref{array:follower-current} to obtain
\[ \sup_{i<n}Y_i(t)\le C_{\mathrm{sec}}(1+M_{\mathrm{hist}})\left[\frac{\ell_n^4}{d_n^5}\left(1+\frac{\Xi}{\Omega_n}\right)+\frac{\Gamma_n^2}{\beta_n^2d_n^4}e^{-c\Omega_{n+1}^\alpha}\right]=o_{\Lambda\to\infty}(1). \]
The convergence is uniform in $n$ for fixed $\alpha,B$. Choose $\Lambda$ so that this is less than half the radius of the slower-quadrupole chart. The norm $Y_i$ includes every shape, impulse, phase, first-moment, translation, and integral component, so no such exit occurs. The geometric estimates improve patch confinement and separation. The estimates of Section~\ref{transfer:memory-saddle}, with Step 2, exclude $Q_n=-q_*$, the remaining target chart boundaries, and the deterministic upper time. Thus $Q_n=q_*$ is reached at
\[ T_n=\tau_{n+1}+\frac1{2\sqrt2\Omega_n}\log\frac{q_*}{\kappa_n \mathcal A_n}+O_{\mathrm{sec}}(\Omega_n^{-1}). \]

\paragraph*{Step 5. Reach the prescribed dipole separation.}
Continue along the fixed compact outgoing segment, stopping at a slower-quadrupole exit or the boundary of its noncollision neighborhood. Use the physical centroid norm of Section~\ref{array:weights-section} and retain $W_n(T_n)$ as prescribed after Theorem~\ref{meta:reduction}. In normalized time, the centroid derivatives and the coefficient of centroid input to $W_n$ are bounded. The return from $W_n$ has coefficient $O(\varepsilon_n)$, and the first-moment and translation coefficients have the bounds above. The linear evolution on this interval is therefore bounded independently of $n$ and $\kappa_n$. Multiply $R_n$ on $[0,T_n]$ by a fixed norm-equivalence factor so that $Y_n(T_n)\le R_n(T_n)$ in this norm, and extend it continuously by
\[ \dot R_n=K_{\mathrm{sec}}\Omega_nR_n+C_{\mathrm{sec}}F_n+C_{\mathrm{sec}}\mathcal K_n(R_n),\qquad K_{\mathrm{sec}}\Omega_n\ge2\sqrt2\Omega_n. \]
In these coordinates,
\[ \limsup_{h\downarrow0}\frac{Y_n(t+h)-Y_n(t)}h\le C_{\mathrm{sec}}\Omega_nY_n(t)+C_{\mathrm{sec}}F_n(t)+C_{\mathrm{sec}}\mathcal K_n(Y_n)(t). \]
Choose $K_{\mathrm{sec}}\ge\max\{C_{\mathrm{sec}},2\sqrt2\}$ and enlarge the return coefficients by the same comparison constants. Positive comparison gives $Y_n\le R_n$ and $\dot R_n\ge2\sqrt2\Omega_nR_n$. Combining this with \eqref{array:majorant} gives, also across $T_n$,
\[ R_n(r)\le e^{-2\sqrt2\Omega_n(T_n-r)}R_n(T_n)\le e^{-2\sqrt2\Omega_n(t-r)}R_n(t),\qquad r\le T_n\le t. \]
In normalized time $s=\Omega_n(t-\tau_{n+1})$,
\[ s=\frac{\log(1/\kappa_n)}{2\sqrt2}+O_{\mathrm{sec}}(1),\qquad \frac{F_n(t)}{\Omega_n}\le C_{\mathrm{sec}}\kappa_n(1+s)^3e^{\sigma s/2}\le C_{\mathrm{sec}}\kappa_n^{3/4}. \]
The bound for the linear evolution on the outgoing segment and the small return coefficient imply
\[ R_n(t)\le C_{\mathrm{sec}}\bigl(R_n(T_n)+\kappa_n^{3/4}\bigr)\le C_{\mathrm{sec}},\qquad 0\le\Omega_n(t-T_n)\le C_{\mathrm{sec}}. \]
The finite-time comparison in Section~\ref{transfer:material-block} improves the compact noncollision bounds. At every current time, \eqref{array:follower-current} follows again from the physical multipole bounds and $R_n(r)\le R_n(t)e^{-2\sqrt2\Omega_n(t-r)}$. The displayed bound on $R_n$ excludes a slower-quadrupole exit. Hence the pair-center distance reaches $r_{\mathrm{out}}\ell_n$ at $\tau_n$, within the deterministic duration used in Step 1. All fixed compact-passage factors are included in $C_{\mathrm{out}}$.

\paragraph*{Step 6. Bounds for the next stage.}
The signed delay formula and \eqref{array:majorant} give
\[ \begin{aligned} \kappa_ne^{2\sqrt2\Omega_n(T_n-\tau_{n+1})}&\le Cq_*/A_*,\\ R_n(T_n)&\le C_{\mathrm{sec}}q_*A_*^{-1}(M_0+C_{\mathrm{sec}}M_{\mathrm{hist}}\zeta_{\Lambda,B}). \end{aligned} \]
For $t\le T_n$, the comparison gives $R_n(t)\le R_n(T_n)e^{-2\sqrt2\Omega_n(T_n-t)}$. Since $\Omega_n(\tau_n-T_n)\le C_{\mathrm{sec}}$,
\[ e^{-2\sqrt2\Omega_n(T_n-t)}\le e^{C_{\mathrm{sec}}}e^{-c\Omega_n(\tau_n-t)},\qquad c<2\sqrt2. \]
On $[T_n,\tau_n]$, the same exponential weight is bounded below by a fixed positive constant. Including these and the compact-passage factors in $C_{\mathrm{out}}$, the constant in \eqref{array:prehistory} for $n$ is at most
\[ C_{\mathrm{out}}A_*^{-1}\bigl(M_0+C_{\mathrm{sec}}M_{\mathrm{hist}}\zeta_{\Lambda,B}\bigr)\le\tfrac34M_{\mathrm{hist}}. \]
This proves \eqref{array:prehistory}, including the first moment, translation, and $W_n$. The time and angle estimates give \eqref{array:timing-recursion} with the stated strict margin.

The estimate \eqref{array:follower-current} also includes $W_i$. Indeed, the direct centroid response to $D_{in}R_n$ costs $\Omega_n^{-1}$. The external input to the equation for $W_i$ costs the same factor, while its centroid input contributes $\Omega_i/\Omega_n\le1$. This is also contained in \eqref{array:forced-propagator}. Retain the resulting value of $W_i$ at the next starting time.

For $n\ge2$, set $i=n-1$ in \eqref{array:follower-current}. The spatial ratio is bounded, and \eqref{array:spacing-frequency} controls the last term. At $\tau_n$,
\[ Y_{n-1}(\tau_n)\le CM_{\mathrm{hist}}\kappa_{n-1}B^{-2}+C_{\mathrm{sec}}(1+M_{\mathrm{hist}})\kappa_{n-1}e^{-c\Omega_{n+1}^\alpha}. \]
Here
\[ \frac{\Gamma_n\ell_n^2}{\Omega_n|d_{n-1}-d_n|^4}=\frac{\Gamma_n^2}{2\pi\Omega_n^2|d_{n-1}-d_n|^4}\asymp\kappa_{n-1}\left(\frac{\Omega_{n-1}}{\Omega_n}\right)^2=\kappa_{n-1}B^{-2}. \]
These are the actual values of every component of $Y_{n-1}$ for the next stage. For $i<n-1$, retain the factor $1+d_i/d_n$ in \eqref{array:follower-current}.

Descending from $N-1$ to $1$, the new constant in \eqref{array:prehistory} is at most $3M_{\mathrm{hist}}/4$, the timing error is at most $3C_h\varepsilon_*/4$, and every slower quadrupole stays within half its chart radius. The next recipient has the bound just proved. The estimates of Section~\ref{transfer:outgoing-material} keep the completed dipoles separating, while the geometric and patch boundary estimates continue. Thus the induction closes with constants independent of $N$.

The deterministic-duration sum gives $\tau_1<1/2$, and \eqref{array:timing-recursion} gives $|\tau_n-\tau_{n,N}|\le C_h\varepsilon_*/\Omega_n$. The separation and geometric estimates keep the dipoles in every quadrupole separating through time 1.

\begin{proposition} \label{array:finite-induction}
With the choices above, for every $N\ge2$ there are two finite patch solutions on $[0,1]$, a prepared reference and a second solution, whose initial data agree in every quadrupole $n<N$. The reference remains in its rhombus charts. The second solution reaches pair-center distance $r_{\mathrm{out}}\ell_n$ at each scale $n$, in descending order, with $\tau_1<1/2$, and satisfies the geometric and patch bounds uniformly in $N$. Its times $\tau_n$ obey \eqref{array:timing-recursion}, with $|h_n|\le C_h\varepsilon_*$, and its dipole angles satisfy $\theta_{d,n}-\vartheta_n=-h_n+O(\varepsilon_*)$. Every completed quadrupole satisfies \eqref{array:prehistory}, and the slower quadrupoles satisfy \eqref{array:follower-current} before their intended stages.
\end{proposition}

Together with the reference construction and prescribed initial replacement, Proposition~\ref{array:finite-induction} proves Proposition~\ref{second:finite-solutions}.

\subsection{The fixed observation window} \label{array:observation-section}

Since $\tau_1<1/2$, the estimates of Section~\ref{transfer:outgoing-material} apply for at least half a unit of time. The separation bound gives $|U_{\mathrm{ext},1}(d_1(t))|\le Cd_1^{-1}\sum_{j\ge2}\Gamma_j$. Integrating the center residual equation, for each solution separately, gives
\[ \sup_{0\le t\le1}|d_1(t)-d_1|\le Cd_1^{-1}\sum_{j\ge2}\Gamma_j+C_{\mathrm{sec},A}\bigl(\mathsf{T}\Omega_1^2\ell_1^2+\ell_1\Omega_1^{-4}\bigr)=o(\ell_1). \]
At time 1, the pair centers are separated by a distance comparable to $\Omega_1\ell_1$, the dipole lengths and patch radii are $O(\ell_1)$, and their arithmetic center is $d_1+o(\ell_1)$.

Each pair center is at distance $R(1)/2$ from the arithmetic center, and each patch is within $C\ell_1$ of its pair center. Hence
\[ c\Omega_1\ell_1-C\ell_1-o(\ell_1)\le |x-d_1|\le C\Omega_1\ell_1+C\ell_1+o(\ell_1) \]
for every point of the four outer patches of the second solution. The four outer patches of the reference lie within $C_{\mathrm{sec}}\ell_1+o(\ell_1)$ of $d_1$.

For large $\Lambda$, these patches lie between radii $c\Omega_1\ell_1/2$ and $2C\Omega_1\ell_1$ for the second solution and within radius $(C_{\mathrm{sec}}+1)\ell_1$ for the reference. Choose radii, independent of $N$, such that
\[ 2(C_{\mathrm{sec}}+1)\ell_1<r_-<\frac c4\Omega_1\ell_1,\qquad 3C\Omega_1\ell_1<r_+<\frac{d_1}{16}. \]
All four outer patches of the second solution at time 1 lie in $\{r_-<|x-d_1|<r_+\}$, and those of the reference lie in $\{|x-d_1|<r_-/2\}$. The annulus contains no other quadrupole, and its boundary has positive distance from both indicated supports, uniformly in $N$. It contains total absolute vorticity $4\Gamma_1$ for the second solution and zero for the reference, as required in Section~\ref{second:observation-section}.

\section{Limit of solution sequences and nonuniqueness} \label{sec:limit}

We keep $\alpha,A,B,\Lambda$ fixed and extract the two solution limits. The limits have the same initial datum. Renormalization and conservation of norms give the strong convergence needed to pass \eqref{second:fixed-observation} to the limit.

\subsection{Common initial data and vorticity norms} \label{sec:norm-growth}

By Proposition~\ref{prep:reference}, the initial disk centers of the first family lie in compact sets at each fixed scale. Choose a diagonal subsequence along which all these centers converge, again indexed by $N$. The radii and signed amplitudes are fixed, and the limiting disks remain separated. Thus each finite collection of patches converges in every finite $L^p$. Let $\omega_0$ be the signed sum of the limiting disks. Equation~\eqref{eq:lp-sum} bounds both tails, so
\[ \limsup_{N\to\infty}\|\omega_N^{(1)}(0)-\omega_0\|_{L^p}\le2\left(8\pi A^{p-1}\sum_{n>m}\Omega_n^{4p-3}e^{-2\Omega_n^\alpha}\right)^{1/p}\longrightarrow0\quad(m\to\infty). \]
Equation~\eqref{eq:initial-difference} then gives $\omega_N^{(k)}(0)\to\omega_0$ in every finite $L^p$ for $k=1,2$. The choices for different $N$ need not be compatible. The common datum is compactly supported, has integral zero, and satisfies
\[ \|\omega_0\|_{L^p}^p=8\pi A^{p-1}\sum_{n\ge1}\Omega_n^{4p-3}e^{-2\Omega_n^\alpha}<\infty. \]
The estimate in Section~\ref{sec:quartic-growth} gives $cp^{4/\alpha}\le\|\omega_0\|_{L^p}\le Cp^{4/\alpha}$ for $p\ge2$.

\subsection{Compactness and strong convergence}

\begin{lemma} \label{limit:compactness}
Let $0<T<\infty$ and suppose
\[ \partial_t\omega_j+\operatorname{div}(u_j\omega_j)=0,\qquad u_j=K*\omega_j\quad\text{on }[0,T]. \]
Assume that $\omega_j:[0,T]\to L^2$ is weakly continuous, has zero integral and support in a ball independent of $j,t$, and satisfies
\[ \|\omega_j(t)\|_{L^p}=\|\omega_j(0)\|_{L^p},\qquad 0\le t\le T,\quad 1\le p<\infty. \]
Suppose that
\[ \omega_j(0)\longrightarrow\omega_0\quad\text{in }L^p,\qquad 1\le p<\infty. \]
Then a subsequence satisfies $u_j\to u$ in $C([0,T];L^2)$ and $\omega_j\to\omega$ in $C([0,T];L^p)$ for every finite $p$. The limit satisfies $u=K*\omega$, has initial datum $\omega_0$, and obeys renormalized vorticity transport, conservation of every finite $L^p$ norm, and the weak momentum equation.
\end{lemma}

\begin{proof}
Let $B(0,R)$ contain all vorticity supports. The Biot--Savart law and $\partial_tu_j=-K*\operatorname{div}(u_j\omega_j)$ give
\[ \sup_j\left(\|u_j\|_{L^\infty_t(H^1_x\cap L^\infty_x)}+\|\partial_tu_j\|_{L^\infty_tL^2_x}\right)\le C. \]
Zero circulation also gives $|u_j(t,x)|\le C|x|^{-2}$ for $|x|>2R$. Local compactness, time equicontinuity, and this tail bound yield $u_j\to u$ in $C([0,T];L^2)$ along a subsequence. Since $\omega_j=\operatorname{curl}u_j$ is uniformly bounded in $L^\infty_tL^2_x$, its limit satisfies
\[ \sup_{0\le t\le T}|\langle\omega_j(t)-\omega(t),\phi\rangle|\longrightarrow0\qquad(\phi\in L^2). \]
The support, circulation, and finite $L^p$ bounds pass to $\omega$, with $\omega(0)=\omega_0$ and $u=K*\omega$. Strong velocity convergence permits passage to the vorticity equation.

Since $u\in L^\infty_t(H^1_x\cap L^\infty_x)$ and $\operatorname{div}u=0$, the DiPerna--Lions renormalization theorem~\cite{DiPernaLions} applies and gives $\|\omega(t)\|_{L^p}=\|\omega_0\|_{L^p}$ for every $t\in[0,T]$ and $1\le p<\infty$. Weak continuity and conservation of the $L^2$ norm give strong $L^2$ continuity of $\omega$ and each $\omega_j$.

The set $\{\omega(t):0\le t\le T\}$ is compact in $L^2$, so the uniform weak convergence above also holds with the test $\phi=\omega(t)$. Consequently,
\[ \|\omega_j(t)-\omega(t)\|_{L^2}^2=\|\omega_j(0)\|_{L^2}^2-\|\omega_0\|_{L^2}^2-2\langle\omega_j(t)-\omega(t),\omega(t)\rangle\longrightarrow0 \]
uniformly in $t$. The common support gives convergence in $C([0,T];L^p)$ for $1\le p<2$, and interpolation with the uniform $L^{2p}$ bound gives it for $p>2$.

Finally, $u\in H^1\cap L^\infty$ implies $\operatorname{curl}\operatorname{div}(u\otimes u)=\operatorname{div}(u\omega)$ in distributions. Testing the vorticity equation with compactly supported stream functions gives the weak momentum equation with initial velocity $K*\omega_0$.
\end{proof}

Extend each $\omega_N^{(k)}$ globally by Yudovich theory~\cite{Yudovich1963}. These solutions agree with the constructed solutions on $[0,1]$, are $L^2$-continuous, and conserve circulation and every finite $L^p$ norm. By \eqref{eq:lp-sum}, $\|u_N^{(k)}(t)\|_{L^\infty}\le C$ and $\operatorname{supp}\omega_N^{(k)}(t)\subset B(0,R+Ct)$, with $C,R$ independent of $N,k,t$.

Section~\ref{sec:norm-growth} gives the common initial limit $\omega_N^{(k)}(0)\to\omega_0$. Applying Lemma~\ref{limit:compactness} to both families and taking a common diagonal subsequence gives, for $k=1,2$, every finite $T$, and $1\le p<\infty$,
\[ \omega_N^{(k)}\longrightarrow\omega^{(k)}\quad\text{in }C([0,T];L^p),\qquad \omega^{(1)}(0)=\omega^{(2)}(0)=\omega_0. \]
The lemma gives renormalization, norm conservation, and the weak momentum equation, with $u^{(k)}=K*\omega^{(k)}$. Strong $L^1$ convergence at time 1 preserves the two values $0$ and $4\Gamma_1$ in \eqref{second:fixed-observation}. Since $0\le\chi\le1$,
\[ \|\omega^{(2)}(1)-\omega^{(1)}(1)\|_{L^1}\ge\int\chi\bigl(|\omega^{(2)}(1)|-|\omega^{(1)}(1)|\bigr)\,dx=4\Gamma_1>0. \]
This proves Theorem~\ref{thm:main}.

\noindent
\textbf{Hyungjun Choi}\\
Department of Mathematics, Princeton University\\
\href{mailto:hyungjun.choi@princeton.edu}{hyungjun.choi@princeton.edu}, \href{https://web.math.princeton.edu/~hc9325/}{web.math.princeton.edu/\textasciitilde hc9325/}
\end{document}